\documentclass{article}

\usepackage{amsmath, amsfonts, amssymb, geometry, amscd, color, enumerate, bbm, bm, graphicx, mathtools}
\usepackage{amsthm}
\usepackage{authblk}
\usepackage[parfill]{parskip}
\usepackage{times}

\makeatletter
\let\save@mathaccent\mathaccent
\newcommand*\if@single[3]{%
  \setbox0\hbox{${\mathaccent"0362{#1}}^H$}%
  \setbox2\hbox{${\mathaccent"0362{\kern0pt#1}}^H$}%
  \ifdim\ht0=\ht2 #3\else #2\fi
  }
\newcommand*\rel@kern[1]{\kern#1\dimexpr\macc@kerna}
\newcommand*\widebar[1]{\@ifnextchar^{{\wide@bar{#1}{0}}}{\wide@bar{#1}{1}}}
\newcommand*\wide@bar[2]{\if@single{#1}{\wide@bar@{#1}{#2}{1}}{\wide@bar@{#1}{#2}{2}}}
\newcommand*\wide@bar@[3]{%
  \begingroup
  \def\mathaccent##1##2{%
    \let\mathaccent\save@mathaccent
    \if#32 \let\macc@nucleus\first@char \fi
    \setbox\z@\hbox{$\macc@style{\macc@nucleus}_{}$}%
    \setbox\tw@\hbox{$\macc@style{\macc@nucleus}{}_{}$}%
    \dimen@\wd\tw@
    \advance\dimen@-\wd\z@
    \divide\dimen@ 3
    \@tempdima\wd\tw@
    \advance\@tempdima-\scriptspace
    \divide\@tempdima 10
    \advance\dimen@-\@tempdima
    \ifdim\dimen@>\z@ \dimen@0pt\fi
    \rel@kern{0.6}\kern-\dimen@
    \if#31
      \overline{\rel@kern{-0.6}\kern\dimen@\macc@nucleus\rel@kern{0.4}\kern\dimen@}%
      \advance\dimen@0.4\dimexpr\macc@kerna
      \let\final@kern#2%
      \ifdim\dimen@<\z@ \let\final@kern1\fi
      \if\final@kern1 \kern-\dimen@\fi
    \else
      \overline{\rel@kern{-0.6}\kern\dimen@#1}%
    \fi
  }%
  \macc@depth\@ne
  \let\math@bgroup\@empty \let\math@egroup\macc@set@skewchar
  \mathsurround\z@ \frozen@everymath{\mathgroup\macc@group\relax}%
  \macc@set@skewchar\relax
  \let\mathaccentV\macc@nested@a
  \if#31
    \macc@nested@a\relax111{#1}%
  \else
    \def\gobble@till@marker##1\endmarker{}%
    \futurelet\first@char\gobble@till@marker#1\endmarker
    \ifcat\noexpand\first@char A\else
      \def\first@char{}%
    \fi
    \macc@nested@a\relax111{\first@char}%
  \fi
  \endgroup
}
\makeatother

\renewcommand{\Re}{\mathop{\rm Re}\nolimits}
\renewcommand{\Im}{\mathop{\rm Im}\nolimits}

\theoremstyle{plain}
\newtheorem{theorem}{Theorem}[section]
\newtheorem{lemma}[theorem]{Lemma}

\newtheorem{corollary}[theorem]{Corollary}

\theoremstyle{definition}

\theoremstyle{remark}
\newtheorem{remark}[theorem]{Remark}

\newtheorem*{remnonum}{Remark}{\it}{\rm}

\newtheoremstyle{RHP}  		
  {3pt}					
  {3pt}					
  {\itshape}				
  {}						
  {\bfseries}				
  {}						
  {.5em}					
  {\thmnote{#3} Riemann-Hilbert Problem #2 } 

\theoremstyle{RHP}
\newtheorem{rhp}{}[section]

\newtheoremstyle{DBAR}  		
  {3pt}					
  {3pt}					
  {\itshape}				
  {}						
  {\bfseries}				
  {}						
  {.5em}					
  {\thmnote{#3} $\overline \partial$ Problem #2 } 

\theoremstyle{DBAR}
\newtheorem{DBAR}{}[section]

\newcommand{\R}{{\mathbb R}}

\newcommand{\N}{{\mathbb N}}
 \def\im{{\rm i}}
\newcommand{\C}{\mathbb{C}}

\newcommand{\La}{ {\mathcal{L} } }
\newcommand{\B}{ {\mathcal{B } } }

\newcommand{\bigo}[1]{\mathcal{O} \left( #1 \right)}
\newcommand{\littleo}[1]{ {o} \left( #1 \right) }

\newcommand{\pd}[3][ ]{\frac{\partial^{#1} #2}{\partial #3^{#1} } }

\newcommand{\lp}{\left(}
\newcommand{\rp}{\right)}
\newcommand{\lb}{\left[}
\newcommand{\rb}{\right]}

\newcommand{\eps}{\epsilon}
\newcommand{\sig}{ {\sigma_3} }
\newcommand{\dbar}{ {\overline{\partial}} }

\newcommand{\vect}[1]{\bm{#1}}
\newcommand{\triu}[2][1]{\begin{pmatrix} #1 & #2 \\ 0 & #1 \end{pmatrix}}
\newcommand{\tril}[2][1]{\begin{pmatrix} #1 & 0 \\ #2 & #1 \end{pmatrix}}

\newcommand{\twovec}[2]{\begin{pmatrix} #1 \\ #2  \end{pmatrix}}

\newcommand{\mk}[1]{  m^{(#1)}  }
\newcommand{\vk}[1]{ { V^{(#1)} } }

\DeclareMathOperator{\Tr}{Tr}
\DeclareMathOperator*{\res}{Res}
\DeclareMathOperator{\sgn}{sgn}
\DeclareMathOperator{\adj}{adj}

\DeclareMathOperator{\Diag}{diag}
\DeclareMathOperator{\sol}{sol}
\DeclareMathOperator{\supp}{supp}

\newcommand{\ie}{\textit{i.e.}}
\newcommand{\cf}{\textrm{c.f.}}

\def\({\left(}
\def\){\right)}
\def\<{\left\langle}
\def\>{\right\rangle}

\numberwithin{equation}{section}

\title{On the asymptotic stability of $N$--soliton solutions of the defocusing nonlinear Schr\"odinger equation}

\author[1]{Scipio Cuccagna}
\author[2]{Robert Jenkins}

\affil[1]{Department of Mathematics and Geosciences,  University
of Trieste, Trieste, 34127  Italy\\
\texttt{scuccagna@units.it}}

\affil[2]{Department of Mathematics, University of Arizona  ,
Tucson 85721 USA \\
\texttt{rjenkins@math.arizona.edu}}

\begin{document}
\maketitle

\begin{abstract}
We consider the Cauchy problem for the defocusing nonlinear Schr\"odinger (NLS) equation for finite density type initial data.
Using the $\dbar$ generalization of the nonlinear steepest descent method of Deift and Zhou we derive the leading order approximation to the solution of NLS for large times in the solitonic region of space--time, $|x|<2 t$, and we provide bounds for the error which decay as $ t \to \infty$ for a general class of initial data whose difference from the non vanishing background possesses a fixed number of finite moments and derivatives. Using properties of the scattering map of NLS we derive, as a corollary, an asymptotic stability result for initial data which are sufficiently close to the $N$-dark soliton solutions of NLS.
\end{abstract}

\section{Introduction and statement of main results}

We consider the Cauchy problem for the defocusing nonlinear Schr\"odinger (NLS) equation on the real line with finite density initial data:
\begin{gather}
		 \im q_{t }+  q_{xx }  -  2(|q|^2-1)  q=0 		\label{eq:nls} \\
		 q(x,0)=q_0(x), \quad
		 \lim _{x\to \pm \infty}q_0 (x) = \pm 1.			\label{eq:nls1}
\end{gather}

\medskip

\begin{remnonum}
The usual form of the NLS equation is $\im u_t + u_{xx} + 2 \sigma u |u|^2 = 0$ where $\sigma = 1$ is called the \textit{focusing} and $\sigma = -1$ the \textit{defocusing} NLS equation. The change of variables $q(x,t) = u(x,t) e^{  2\im   t}$ reduces the defocusing NLS equation to \eqref{eq:nls}. This form has the advantage that solutions of \eqref{eq:nls} which satisfy \eqref{eq:nls1} are asymptotically time independent as $x\to \infty$.
\end{remnonum}

\medskip

\begin{remnonum}
Some authors have referred to \eqref{eq:nls} as the Gross-Pitaevskii (GP) equation \cite{BGS,BGS1,BGSS,BGSS1,gallo,GZ,GS}. The general one dimension GP equation, which appears in the modeling of Bose-Einstein condensates on a nonzero background, is $\im \psi_t + \psi_{xx} - 2( |\psi |^2 -1) \psi + V(x) \psi = 0$ where $\psi$ is the wave function of a single particle and $V$ is an external potential. 
Equation \eqref{eq:nls} is the integrable case of the 1D GP equation in which the particle is free, \ie, $V \equiv 0$. 
\end{remnonum}

It is an elementary fact that solutions of the linear Schr\"odinger equation $i q_t+q_{xx} = 0$ disperse, \ie,\, $q(x,t) = \bigo{ t^{-1/2} }$ as $t \to \infty$.
Once nonlinear effects are included, \textit{soliton} solutions appear. These special solutions do not disperse. Instead, the nonlinear effects balance the dispersive, to create solutions which persist for all time.
For initial data $q_0(x)$ which vanish sufficiently quickly as $|x| \to \infty$ only the focusing NLS equation supports soliton solutions. For the finite density type of data considered in \eqref{eq:nls}-\eqref{eq:nls1} the defocusing equation also possesses soliton solutions.
Let $\partial D(0,1) =\{ z\in \C : |z|=1\}$.
For $ z_0 \in \partial D(0,1) \cap \C^+$
define
\begin{equation} \label{eq:defsol}
	\sol (x, t ; z_0) := -\im  z_0 \left( \im z_{0R}
    	+z_{0I} \tanh \left ( z _{0I} (x  - 2 z  _{0R}t) \right ) \right )   
	\text{ where $z _{0R}=\Re z_0$ and $z _{0I}=\Im  z_0$}.
\end{equation}
Then $ q(x,t) = \sol ( x-x_0, t ; z_0 )$ is a traveling wave solution of \eqref{eq:nls} satisfying
$\lim_{x \to \infty} q(x,t)=1$  and $\lim_{x \to - \infty} q(x,t)=z_0^2$.
We call these solutions 1-\text{solitons}, or simply solitons.
More generally, given a collection of distinct points $\{ z_k \}_{k=0}^{N-1} \in \partial D(0,1) \cap \C^+$ one can construct more elaborate exact soliton solutions $q^{(sol),N}(x,t)$, called $N$-\textit{solitons} which, instead of dispersing, resemble the sum of $N$ individual $1$-solitons at sufficiently large times.
Such solutions are constructed in Appendix \ref{sec:msol}.

The soliton resolution conjecture is the vaguely stated, but widely believed statement that the evolution of generic initial data for many globally well posed nonlinear dispersive equations will in the long time limit resolve into a finite train of solitons plus a dispersing radiative component.
For most dispersive evolution equations this is a wide open and active area of research \cite{tao,CM,MaM0,MaM1}.
The situation is somewhat better understood in the integrable setting where the inverse scattering transform (IST) gives one much stronger control on the behavior of solutions than purely analytic techniques \cite{DT,GT,TVZ,DKKZ,DP,CP}. Even among the integrable evolutions, most results concern initial data with sufficient decay at spatial infinity, but there have been some recent studies concerning non vanishing initial data \cite{EGKT,BKS,J,V1,V2}.

As we review below, problem \eqref{eq:nls}--\eqref{eq:nls1} is integrable--- as discovered by Zakharov--Shabat--- and its solution can be characterized in terms of an IST \cite{ZS1}.
Briefly, the Lax-pair representation of \eqref{eq:nls} (\cf \ \eqref{eq: lax pair}) encodes the solution of NLS as a time evolving potential in a certain spectral problem, \eqref{eq:zs}, on the line.
In analogy to the standard Sturm--Liouville theory for Schr\"odinger operators on the line,
see for example \cite{DT},  the \textit{scattering map} associates  to $q_0$ a
discrete spectrum, formed by a finite number of  \textit{poles}   $\{ z_j \}_{j=0}^{N-1} \subset \partial D(0,1) \cap \C^+$ and for each pole  an  associated   \textit{coupling constant} $c_j \in \im z_j \R_+$. In addition to the discrete data, the scattering map associates to $q_0$ the so called  \textit{reflection coefficient}, $r$, defined along the continuous spectrum of the scattering operator, \ie\ , 
$r:\R \to \C$, which is a sort of Fourier transform of  $q_0$ satisfying $|r(z)|<1$ for any $z\neq \pm 1$ with $r(0)=0$.
The collection $\left\{ r(z), \{ z_j, c_j \}_{j=0}^{N-1} \right\}$ is called the \textit{scattering data} associated with $q_0$.
In terms of scattering data, soliton solutions correspond to \textit{reflectionless} potentials $q_0$ for which the scattering map gives $r(z) \equiv 0$; the scattering data of a $1-$soliton is $\{ 0 , \{z_0, c_0\} \}$ and for an $N-$soliton $\{0, \{ z_k ,c_k \}_{k=0}^{N-1} \}$.

The essential fact is that the evolution of the scattering data is trivial, and an \textit{inverse scattering map} can be constructed in terms of a Riemann-Hilbert problem where the spatio-temporal dependence appears only parametrically.
This characterization of the inverse map is ideally suited to rigorous  asymptotic analysis via the Deift-Zhou steepest descent method and has been key to deriving detailed asymptotic expansions of NLS and other integrable evolutions in various asymptotic regimes \cite{BM,CG,DZ2,DZ3,JM,TVZ}. 

The long time asymptotic behavior of the defocusing NLS equation with finite density data has been studied previously. In a series of papers \cite{V1,V2,V4}  Vartanian computed both the leading and first correction terms in the asymptotic expansion of the solution $q(x,t)$ and `partial masses' $\int_{\pm \infty}^x ( 1 - |q(x,t)|^2) dx$ of \eqref{eq:nls}-\eqref{eq:nls1} as $x,t \to \pm \infty$ with $|\xi| = |x/2t| >1$ (outside the soliton `light cone') and $|\xi|< 1$ (inside the soliton `light cone'). 
In particular, in \cite{V2}, it is shown that when the initial data generates discrete data $\{z_k, c_k \}_{k=0}^{N-1}$ (in our notation), then in a frame of reference moving at one of the soliton speeds, \ie\, $x + 2 \Re( z_j) t = \bigo{1}$ for any $j=0,1,\dots,N-1$ , the solution $q(x,t)$ of \eqref{eq:nls}-\eqref{eq:nls1} is asymptotically described to leading order by a 1-soliton. 
Our first result below, Theorem~\ref{thm:main1}, describes the leading order asymptotic behavior of the solution $q(x,t)$ uniformly in any closed sector within the `light cone', that is with $|\xi| \leq \xi_0 < 1$. 
Our formulation \eqref{eq:thmmain1} is consistent with Vartanian's description, but is formulated such that we give a more holistic description of the solution as an $N$-soliton with a fixed set of poles whose coupling constants slowly modulate due to the interaction of the soliton components with the reflection coefficient. 
Expressing this $N$-soliton solution in separated form for $t \gg 1$ , \eqref{soliton separation} reduces to Vartanian's leading order asymptotics in the frames of reference $x + 2 \Re( z_j) t = \bigo{1}$ defined by the individual solitons. 
From a technical perspective our $\dbar$ methods greatly simplify the analytical arguments needed to prove results. 
Moreover, our results hold for a much larger class of initial data than was considered in \cite{V1,V2,V4} (\cf\ Remark~\ref{rem:vartanian}). 
Finally, we believe that our methods should be more easily adapted to considering the so called collisionless shock region $|x/2t| \approx 1$ where $|r(z)| \to 1$ which we plan to consider in the near future. 

\subsection{Results}
Our first result is a verification of the soliton resolution conjecture for \eqref{eq:nls} for initial data of finite density type \eqref{eq:nls1} which possess a certain number of derivatives and moments.
To state the theorem precisely we introduce the Japanese bracket $\langle x \rangle := \sqrt{1+|x| ^2} $ and the normed spaces:
\begin{equation*}
\begin{aligned}
	& L^{p ,s}(\R ) \text{  defined with  }
  		\| q \| _{L^{p,s}(\R) }  :=  \|   \langle x \rangle ^s   q \| _{L^p ( \R )} ;  \\
  	& W ^{k,p}( \R )\text{ defined with }
		\| q \|_{W^{k,p}(\R)} := \sum_{j=0}^k \| \partial^j \! q \|_{L^p(\R)},
		\text{ where $\partial^j \! u$ is the $j^{th}$ weak derivative of $u$;} 	\\
 	& H^k( \R ) \text{  defined with  }
  		\| q \| _{H^k( \R ) }  :=  \|   \langle x \rangle ^k   \widehat{q} \| _{L^2 ( \R )},
		\text{ where $\widehat{u}$ is the Fourier transform of $u$;} \\
	& \Sigma _k :=L^{2 ,k}( \R )\cap   H^k ( \R )  .
\end{aligned}
\end{equation*}
We also set $\C^\pm =\{ z\in \C : \pm \Im z >0\} $ and $\R_+ = (0,\infty)$. 


\begin{theorem}\label{thm:main1}
Consider initial data $q_0\in  \tanh \left (  x \right) +\Sigma _4$ with associated scattering data $ \{ r(z), \{z_j, c_j \}_{j=0}^{N-1} \}$. 
Order the $z_j$  such that
\begin{equation}\label{eq:ord1}
	\Re z_0 > \Re z_1 > \dots > \Re z_{N-1},
\end{equation}
let $\xi = x/2t$, and define 
\begin{equation}\label{acquired phase}
	\alpha(\xi) = \frac{1}{2\pi} \int_0^\infty \frac{ \log(1 - |r(s)|^2)}{s} ds
	+ 2 \sum_{\mathclap{k:\, \Re z_k > \xi}} \arg z_k. 
\end{equation}
Fix $\xi _0\in (0,1)$, then there exist $t_0=t_0(q_0,\xi_0 )$ and $C=C(q_0,\xi_0 )$ such that the solution   $q(x,t)$ of \eqref{eq:nls}-\eqref{eq:nls1} satisfies
\begin{equation}\label{eq:thmmain1}
	\left| q(x,t) - e^{i \alpha(\xi)} q^{(sol),N}(x,t) \right | \leq  C t^{-1}
	\text{ for all $t > t_0$ and $|\xi|  \le \xi_0  $.}
\end{equation}
Here $q^{(sol),N}(x,t)$ is the  $N$-soliton  with associated scattering data
$\{ \widetilde{r} \equiv 0,  \{ z_j ,\widetilde{c}_j \}_{j=0}^{N-1}  \}$ where
\begin{equation}\label{new phase}
	\widetilde{c}_j  = c_j \exp \lp -\frac{1}{\im \pi }
	\int_0^\infty \log (1 - |r(s)|^2) \lp \frac{1}{s-z_j} - \frac{1}{2s} \rp ds \rp.
\end{equation}

Moreover, for $t>t_0$ and $|\xi| < \xi_0$, the N-soliton solution separates in the sense that
\begin{equation}\label{soliton separation}
	q(x,t) = e^{i \alpha(1)}  \left[ 1 + \sum_{k=0}^{N-1} \lp \prod_{j < k}  z_j^2 \rp \lb \sol(  x- x_k, t ; z_k) - 1 \rb \right] + \bigo{ t^{-1} },
\end{equation}
where $\sol( x,t; z)$ is the one soliton defined by \eqref{eq:defsol} and
\begin{equation}\label{delta+}
	\begin{gathered}
	x_{k} = \frac{1}{2 \Im(z_{k})} \lp \log \lp \frac{|c_{k} |}{2 \Im(z_{k})}
		\prod_{\substack{{\ell:\,  \Re(z_{\ell} ) > \xi } \\ \ell \neq k}}
		\left| \frac{ z_{k} - z_\ell}{z_{k} z_\ell - 1} \right|^2 \rp
		- \frac{\Im(z_{k})}{\pi} \int_0^\infty \frac{\log(1 - | r(s) |^2)}{| s- z_{k}|^2} ds \rp .
	\end{gathered}
\end{equation}
\end{theorem}

\begin{remark}\label{rem:regularity}
The smoothness and decay properties of the reflection coefficient
needed in the proof of Theorem \ref{thm:main1}, which follow from our hypotheses
on $q_0$, are proved in Section \ref{sec:ab}.
Specifically, for $q_0 \in  \tanh (x) + \Sigma _m$ with $m=2$
we will prove $r\in L^2(\R )$ and  $\| \log (1 - | r|^2 ) \|_{L^p(\R)}$ for $p \geq 1$.
$m=3$ implies  $q_0 \in  \tanh (x) + L ^{1,2}(\R)$ which in turn allows us to show that $r(z) \in H^1(\R)$; additionally for $m=3$ we show (Lemma \ref{lem:simple}) that the scattering map has a finite discrete spectrum. This improves \cite{DPMV} where finiteness of the spectrum was proved for $q_0 \in  \tanh (x) + L ^{1,4}(\R)$.
The condition that $q_0 \in \tanh(x) + \Sigma_m$ with $m=4$ is needed only to bound the $\dbar$ derivatives of our extensions of the reflection coefficient in Lemma~\ref{lem:extR}; specifically it allows us to use \eqref{eq:a13}  with $n=2$.
\end{remark}

\begin{remark}
The restriction  $| {x} | < 2 {t}    $ in Theorem \ref{thm:main1} is used only to limit the length of this paper.
This is the critical regime for studying the soliton resolution of the solution as the soliton speeds $v_j$ in the scaling of \eqref{eq:nls}-\eqref{eq:nls1} are bounded by $| v_j | < 2$.
The steepest descent method of Deift and Zhou used in this paper can also be used to study the behavior of $q( x,t)$ as $t \to \infty $ in the rest of space--time.
\end{remark}

\begin{remark}
The two terms in \eqref{delta+} for the asymptotic phase shifts $x_j$ have clear interpretations. The first term gives the phase shift due to interactions between the solitons. The second term is a retarding factor due to the interaction of the soliton component with the radiative component of the solution.
\end{remark}

\begin{remark}\label{rem:vartanian}
Long time asymptotic results for \eqref{eq:nls}-\eqref{eq:nls1} were previously obtained by Vartanian in \cite{V1,V2,V4} under the assumption that $q_0(x) - \tanh(x)$ is Schwartz class, and that the reflection coefficient $r(z)$ satisfies $|r( \pm1) | < 1$.  This is a non-generic situation in that for most data $|r(\pm1)| = 1$, as we review below.
The hypothesis $\| r \| _{L^\infty (\R )}<1 $ is used crucially in \cite{V1,V2,V4} in the context of some standard factorizations in  the steepest descent method, see for example formula
 (0.23) in  \cite{DZ3},  which we write in \eqref{factorizations}--\eqref{LDU factor} and which display  factors $(1-|r (z)|^2) ^{-1}$ that would be singular at $\pm1$  if
$|r(\pm1)| = 1$.
Our methods remove the non-generic condition $\| r \| _{L^\infty (\R )} <1 $  and can handle a wider class of initial data while also requiring less technical estimates along the way.
\end{remark}

Observe that if we take $z_0 = \im$ in \eqref{eq:defsol} then we have the stationary solution
of \eqref{eq:nls}-\eqref{eq:nls1}
\[
	\sol(  x,t ; \im  ) = \tanh(x)
\]
which is called the \textit{black soliton} in analogy with the nonlinear optics application where $|q|^2$ represents the intensity of the light wave.
When $z_0 \neq \im$ the solution is a non-stationary dark soliton, becoming increasingly `whiter' as $z_0 \to \pm 1$. 
There is a substantial body of work treating the \textit{orbital} stability of the black soliton, see
 \cite{GZ,DG,BGSS1,BGS1,GS} and therein.
The  \textit{asymptotic} stability of a dark soliton, the case $z_0\neq \im $, is discussed in
 \cite{BGS}, while the case of the black soliton $z_0= \im $ is discussed in \cite{GS}.
 Orbital stability of multi--solitons is considered in  \cite{BGS1}.

A corollary  of  Theorem~\ref{thm:main1}   is the following
asymptotic stability type result  for the  multi--solitons.

\medskip

\begin{theorem}\label{thm:main2}
Consider an $M$--soliton  $q^{(sol),M} (x ,t)$ satisfying both
boundary conditions in \eqref{eq:nls1} and
let $\{ 0 , \{ z_j,c_j \} _{j=0}^{M-1} \} $ denote its reflectionless scattering data. 
There exist $ \varepsilon _0>0$ and $C>0$ such that for any initial datum $q_0$ of
problem \eqref{eq:nls}--\eqref{eq:nls1} with
\begin{equation} \label{eq:thmmain00}
	\epsilon :=\| q_0  -  q^{(sol),M} (x ,0) \| _{\Sigma _4}   < \varepsilon _0
\end{equation}
the initial data $q_0$ generates scattering data $\{ r', \{z_j' ,c_j' \} _{j=0}^{N-1} \}$ for some finite $N \geq M$
(for both sets of discrete data we use the convention that $j<k$ implies $ \Re z_j >  \Re z_k$ and $ \Re z_j' >  \Re z_k'$).  
Of the discrete data of $q_0 $,  exactly $M$ poles are close to discrete data of $q^{(sol),M}$. Any additional poles (as $N\geq M$) are close to either $-1$ or $1$. 
Specifically, there exists an $L\in \{ 0,..., N-1\}$ satisfying  $L+M \le N -1$  for which 
we have
\begin{equation}\label{eq:prop1}
	\max _{0  \le j  \le M-1  } ( | z_j- z _{j+L}' | +   | c_j   - c_{j+L}'   |)
	+ \max _{j > M+L}| 1+ z_j' |  + \max  _{j < L}| 1- z_j' |    < C \epsilon  .
\end{equation}
Furthermore, $q_0$ has  reflection coefficient $r' \in H ^1(\R ) \cap W ^{2,\infty}(\R \backslash ( -\delta _0,\delta _0 ))$  for any $\delta _0 >0$.

\noindent Set $\xi := x/2t$ and fix $ \xi _0\in (0,1)$ such that $\{ \Re z_j \}_{j=0}^{M-1} \subset (-\xi_0, \xi_0)$.
Then there exist $t_0(q_0, \xi_0)>0$, $C=C(q_0, \xi_0)>0$ and 
$\{ x_{k+L} \}_{k=0}^ {M-1} \subset \R$ such that 
for $t>t_0(q_0,\xi_0)$, $|\xi |\le \xi _0$ and
$\alpha(\xi)$ as defined by \eqref{acquired phase},
the following inequality holds:
\begin{equation}\label{soliton separation0}
	  \left| q(x,t) - e^{i\alpha(1)} (\prod_{j \le  L}  z_ {j } ^{\prime 2})
	  \lb 1 + \sum_{k=0}^{M-1} (\prod_{j =0}^{k-1}  z_ {j+L} ^{\prime 2})   
	  ( \sol ( x- x_{k+L}, t ; z_{k+L}'   ) - 1 ) \  \rb   \right| 
	  \le C { t^{-1} }.
\end{equation}

\end{theorem}
\begin{remark}
Notice that in the region  $|\xi |\le \xi _0$ the extra solitons for $j<L$ and  for $j\ge L+M$ approach constant values exponentially fast in time and their contribution inside the l.h.s. of \eqref{soliton separation0} would be exponentially small in $t$.
Equation \eqref{soliton separation0}  is written
considering only   discrete data  close to those of $ q^{(sol),M}  $ in order
to emphasize that the latter is asymptotically stable,
up to some ``phase shifts" and small changes of the velocities of the solitons,  in the region  $|\xi |\le \xi _0$.
\end{remark}

\begin{remark}\label{rem:blacksol}
Theorem \ref{thm:main2} yields when $M=1$ and  $q^{(sol),1} (x,t)=\tanh (x) $ an asymptotic
stability result for the \textit{black soliton}.
Minor modifications in our arguments yield asymptotic
stability  results also for  \textit{dark solitons} and for $N$--solitons with
boundary conditions different from \eqref{eq:nls1}.
\end{remark}

Our proofs of Theorem~\ref{thm:main1} and Theorem~\ref{thm:main2} take advantage of the integrability of \eqref{eq:nls}-\eqref{eq:nls1}.
Integrability allows one access to the inverse scattering transform (IST) machinery. This was the approach of G\'erard--Zhang \cite{GZ} in their discussion of  orbital stability of solutions near the black soliton.
We recall that the IST provides a representation of a solution $q(x,t)$
of an integrable equation in terms of its scattering data
reminiscent  of the Fourier representation formula
$q(x,t)= \int _\R e^{\im t k^2 + \im xk} \widehat{q}_0(k) dk$ 
for the linear Schr\"odinger equation.
One can then envisage that solutions $q(x,t)$ of an integrable equation
might be estimated by nonlinear analogues of the stationary phase method
or other classical tools used in asymptotic analysis. The steepest descent method of Deift--Zhou does exactly this.

We would like to highlight some of the technical aspects of the manuscript.
Our proof uses the $\dbar$ method for contour deformation introduced in McLaughlin--Miller  \cite{MM1,MM2} and Dieng--McLaughlin \cite{DM}, which allows us to consider initial data with only a small amount of regularity while simultaneously simplifying many of the necessary estimates of more standard steepest descent.
The new ingredient in our problem compared to those in the above mentioned $\dbar$ papers is the presence of solitons. 
Our procedure for accounting for the soliton contribution to the IST was in part inspired by \cite{DKKZ} and \cite{GT}.
The main technical problem we face, mentioned in Remark \ref{rem:vartanian}, is the fact that $ |r( \pm 1)| =1$ generically.
Indeed, we show in Appendix \ref{app:sing} that this happens generically  even when $q_0  -  \tanh \left ( x \right )$ is compactly supported and small.
We solve this problem by an appropriate adaptation of the  $\dbar$ method, at the cost of some loss
of regularity with respect to the standard $\dbar$ method of \cite{DM}.
This is described in detail in Lemma~\ref{lem:extR}.

Unfortunately, currently the IST is not well suited to explore cases when the metric used in \eqref{eq:thmmain00} is as weak as  in \cite{GS}, where, assuming that there is a way to associate to
$q_0$   scattering data, we should expect infinitely many poles  concentrating near
the points $\pm 1$ (for somewhat  related material see \cite{LCT}).
This is a situation we do not consider.
Instead, in the case when $\epsilon$  in \eqref{eq:thmmain00} is finite, we know by \cite{DPMV} that there is only a finite number of poles.
So we are very far from the very general set up considered in \cite{BGS,GS,BGS1} and in some obvious
respect our asymptotic stability result in the special case of solitons is much weaker than \cite{GS}.

Nonetheless, in the case treated here of solutions of \eqref{eq:nls} with sufficient regularity and finite higher momenta, the steepest descent method provides more information on the asymptotic behavior of $q(x,t)$ as $t \to \infty$ than in \cite{BGS,GS}.
Furthermore, we treat the case  of $N$--solitons for any $N$.
It is also interesting to explore the same problem using completely different theoretical
frameworks (we recall that \cite{BGS,GS} follow  the arguments introduced by Martel and Merle,
see \cite{MaM0}--\cite{MaM1} and therein). Obviously the approach in \cite{BGS,GS},
not based on a direct use of the integrability of \eqref{eq:nls}--\eqref{eq:nls1},
appears more amenable to extension to non integrable NLS's and so stronger than ours also in this respect.
On the other hand, apart from questions on the correct formulation of the problem and some technical complications in Sect. \ref{sec: step4}, the arguments in the present
paper are  technically rather elementary.
Considering that, from the viewpoint of scattering data, distinct integrable systems
might not be very different from each other, perhaps similar arguments
apply to other systems.
As for the integrability of \eqref{eq:nls}--\eqref{eq:nls1} and the non robustness of this condition in real life, we remark that we expect that   it might be  possible  to   extend the analysis to non  integrable systems, like in \cite{DZ1}, although admittedly this part of the theory seems at its infancy.
For a recent paper on this topic we refer to \cite{BN}.
Our paper was written independently of  \cite{GS}, which we learned about only after finishing the  mathematical part   of our paper.

\section{Plan of the proof} \label{sec:plan}

We prove Theorems~\ref{thm:main1} and \ref{thm:main2} by applying the inverse scattering transform (IST) to the NLS equation \eqref{eq:nls}-\eqref{eq:nls1}.

In Sections~\ref{sec:jost} we review the integrable structure of \eqref{eq:nls}.
The Lax-pair \eqref{eq: lax pair} gives one an eigenvalue problem \eqref{eq:zs} in which the solution $q(x,t)$ of NLS appears as a potential.
We construct \textit{Jost solutions} of \eqref{eq:zs}, certain normalized solutions of \eqref{eq:zs}: $\psi_1^-(z; x,t)$ and $\psi_2^+(z; x,t)$, $k=1,2$, holomorphic for $\Im z > 0$ with derivatives in $\Re z$ and $\Im z$ extending continuously to $\widebar{\C^+}$ and $\psi_1^+(z; x,t)$ and $\psi_2^-(z; x,t)$, $k=1,2$, holomorphic for $\Im z < 0$ with derivatives in $\Re z$ and $\Im z$ extending continuously to $\widebar{\C^-}$.
We enumerate several properties of these solutions under various assumptions on the smoothness and decay rate of $q(x,t) - \tanh(x)$.
Implicit to this construction is that we have global solutions of \eqref{eq:nls}, $q \in \tanh(x) + \Sigma_4$. This is shown in Appendix~\ref{sec:existence}.

In Section~\ref{sec:ab} we describe how one constructs the scattering data from these Jost functions.
The Wronskian $\det [\psi_1^-(z;x,t),\psi_2^+(z;x,t) ]$ is shown to be independent of both $x$ and $t$, and its zeros are precisely the discrete spectrum of \eqref{eq:zs} for $z \in \C^+$.
These numbers each encode a single soliton component of the solution of \eqref{eq:nls}.
The total number of solutions of $\det [\psi_1^-(z;x,t),\psi_2^+(z;x,t) ] = 0$ is finite provided $q_0 \in \tanh(x) + L^{1,2}(\R)$, (\cf\  Lemma~\ref{lem:simple} below).
The totality of the scattering data generated by $q_0( x)$ consist of the zeros  $\{z_j \} _{j=0}^{N-1}$ of
the Wronskian, where $N \ge 1$ is finite, of their corresponding coupling constants  $\{c_j \} _{j=0}^{N-1}$, and of the reflection coefficient $r(z)$, which we will  show  belongs in
$H^1(\R)$ and satisfies additional estimates proved in Sect.~\ref{sec:ab}.
In particular we show that generically we have (\cf\ \eqref{eq:r sing}) $\lim_{z \to \pm 1} r(z) = \mp 1$.
The situation in which \eqref{eq:r sing} does not hold is simpler.
Another issue that appears in Section~\ref{sec:ab} is that the map from initial data $q_0$ to scattering data is not continuous at the soliton solutions. In appendix $C$ we show that even compactly supported perturbations of the single black soliton can be multisolitonic in that the perturbed Wronskian $\det [\psi_1^-(z;x,t),\psi_2^+(z;x,t) ]$ can have up to two new zeros in $\C^+$. The new zeros however are very close to $z = \pm 1$ corresponding to nearly white solitons. In particular we have a perturbative result in Lemma~\ref{lem:numbzeros}.

In Sect.~\ref{sec:rh} we define a Riemann--Hilbert problem (RHP) for a sectionally meromorphic function $m(z; x, t)$ and describe how the solution of \eqref{eq:nls}-\eqref{eq:nls1} can be recovered from the solution $m(z; x,t)$ of the RHP.
We initiate the long time analysis of \eqref{eq:nls} in Sect \ref{sec:time} by using the $\dbar$ generalization of the Deift-Zhou steepest descent procedure following the ideas in \cite{DM}.
This proceeds as a series of three explicit transformations $m(z) \mapsto \mk{1}(z) \mapsto \mk{2}(z) \mapsto \mk{3}(z)$ such that the final unknown $\mk{3}(z)$ is a continuous function in the complex plane with an asymptotically small $\dbar$ derivative uniformly in the complex plane.
This allows one to prove the existence of $\mk{3}$ using functional analytic methods and the theory of the solid Cauchy transform.

In Sect.~\ref{sec:step1} we introduce the first transformation, a set of conjugations and interpolations such that the new unknown $\mk{1}$ has no poles following the ideas in \cite{DKKZ,GT,BKMM}.
The second transformation is the heart of the steepest descent method, where
appropriate factorizations of the jump matrices of the RHP on the real line are introduced and certain non-analytic extensions of these factorizations are used to deform the jumps onto contours in the plane on which they are asymptotically small.
The main issue here is that $|r(1)|=1$ introduces singular factors in the factors \eqref{LDU factor} which are part of the matrix factorizations in \eqref{factorizations}--\eqref{LDU factor}, which play a central role in the theory. 
Nevertheless, in Sect.~\ref{subsec:lense} we construct extensions whose $\dbar$ derivatives satisfy particular bounds, analogous to the those proved in the case of vanishing initial data \cite[Proposition 2.1]{DM}. These bounds are later used to control certain solid Cauchy integral operators that appear later in the inverse analysis.

Section~\ref{subsec:rem} contains the third transformation which gives the leading order asymptotic behavior of the solution.
In Lemma~\ref{lem: m sol} we show that if one ignores the $\dbar$-component of $\mk{2}$ what remains is a trivial conjugation of the RHP corresponding to an $N$-soliton whose reflectionless scattering data $\{ 0 ,\{ z_k , \widetilde{c}_k \}_{k=0}^{N-1} \}$ is known exactly. The poles $z_k$ are the same as those generated by the original initial data $q_0$ given in \eqref{eq:nls1}, but the connection coefficients $\widetilde{c}_k$ are modifications of the original $c_k$ by an amount which depends upon the reflection coefficient coefficient (\cf\ \eqref{new phase}). We solve this $N$-soliton problem exactly, so that we have a single expression for the asymptotic behavior of the solution uniformly for $|x| < 2t$ for large $t$.
We then given a long time asymptotic expansion for the $N$-soliton solution depending on the ratio $\xi = x/2t$ which gives the soliton component of the soliton resolution conjecture.

Finally, in Sect.  \ref{sec: step4}  we prove the existence of the function $\mk{3}$ and estimate its size in a way similar to Sect. 2.4--2.5 in \cite{DM} using the bounds on the $\dbar$ derivatives of the extensions constructed previously in Section~\ref{subsec:lense}.
Summing up the estimates yields the proof of Theorem \ref{thm:main1} in Sect. \ref{sec:pfmain1}.

\section{Jost functions} \label{sec:jost}
In this section we state without proof the details of the forward scattering transform for defocusing NLS for step-like initial data. The results are well known and the interested reader can find pedagogical and detailed treatments in the literature, see \cite{FT,BC,DT}.

The integrability of \eqref{eq:nls} follows from its Lax pair representation
\begin{subequations} \label{eq: lax pair}
\begin{align}
\label{eq:lax1}	v_x = \La v, \\
\label{eq:lax2} \im v_t = \B v.	
\end{align}
\end{subequations}
The $2\times 2$ matrices $\mathcal{L}$ and $\mathcal{B}$ are given by
\begin{subequations}\label{eq:laxmat}
	\begin{align}
		\label{eq:laxL}	\La = \La(z;x,t) &= \im \sigma_3 ( Q - \lambda(z)) \\
		\label{eq:laxB} \B = \B(z;x,t) &= -2 \im \lambda(z) \La - (Q^2 - I)\sigma_3 + \im Q_x
	\end{align}
\end{subequations}
where,
\begin{equation*}
	Q = Q(x,t) = \begin{pmatrix} 0 &   \overline{q}(x,t) \\ q(x,t) & 0 \end{pmatrix},
	\qquad
	\lambda(z) = \frac{1}{2} (z + z^{-1}),
\end{equation*}
and $\sigma_3$ is the third Pauli matrix:
\begin{equation}\label{eq:pauli}
\sigma_1 = \begin{pmatrix} 0 & 1 \\ 1 & 0 \end{pmatrix}, \qquad
\sigma_2 = \begin{pmatrix*}[r] 0 & -\im \\ \im & 0 \end{pmatrix*}, \qquad
\sigma_3 = \begin{pmatrix*}[r] 1 & 0 \\ 0 & -1 \end{pmatrix*}.
\end{equation}

The commutativity of the mixed partials of $v$, which is the compatibility condition for a simultaneous solution of \eqref{eq: lax pair}, is equivalent to
\begin{equation} \label{eq:lax0}
	\im( \im \La_t - \B_x + [\La,\B] ) = -\im \sigma_3 Q_t + Q _{xx} - 2(Q^2-I) Q =0,
\end{equation}
which is just a matrix reformulation of \eqref{eq:nls}.

Fix $q(x)$ such that $\displaystyle \lim _{x\to \pm \infty}q (x)=\pm 1$ (appropriate reformulations
of what follows hold for different boundary values in $\partial D(0,1)$).
Writing \eqref{eq:lax1} as an eigenvalue equation gives
\begin{equation}  \label{eq:zs}
	\im \sigma _3 v _x + Q v -\lambda(z) v=0.
\end{equation}
Let
\begin{gather}
\label{eq:eqB}
	B_\pm = B_{\pm}(z) =  I \pm \sigma_1 z^{-1}
\shortintertext{and}
\label{eq:zs1}
 	X^\pm(x,z) = B_{\pm}(z) e^{-\im \zeta(z) x \sigma_3}
\end{gather}
where
\begin{equation} \label{eq:zeta}
	\zeta(z) = \frac{1}{2} (z - z^{-1}).
\end{equation}
Then $X^{\pm  }$  are the solutions of \eqref{eq:lax1} obtained by replacing $Q(x)$ by
$\pm Q_{+ }$ in \eqref{eq:zs} with $ Q_{+ } = \lim_{x\to +\infty} Q(x) = \sigma_1$. 
We define \emph{Jost functions}, $\psi_j^\pm(z;x), j=1,2,$ to be the column vector solutions of \eqref{eq:zs}
whose values approach those of the $j^{th}$ column of \eqref{eq:zs1} as 
$x \to \pm \infty$.
The existence of such solutions, and their analytic properties as functions of $z$, is the subject of the following Lemma.

\begin{lemma}
\label{lem:jf1}
Let $q(x)$ be such that $ q -  \tanh(x) \in L^1(\R  )$.
Then for $z \in \R \backslash \{0,1,-1\}$ the system \eqref{eq:zs} admits solutions
\begin{equation} \label{eq:jf1}
	\begin{aligned}
		&    \psi _1^{\pm } (  z;x) =  m _1^{\pm} (z;x ) e^{-\im \zeta(z) x}   
	 	\quad \text{and} \quad
	 	\psi _2^{\pm} (z;x ) = m _2^{\pm} (z;x) e^{ \im \zeta(z) x}    
	 \end{aligned}
\end{equation}
such that
\begin{equation} \label{eq:jf2}
	\begin{aligned}
		&   \lim _{x\to \pm \infty}  m _1 ^{\pm} (z;x)
		= \begin{pmatrix} 1  \\  \pm z^{-1}   \end{pmatrix}  
		\quad \text{and} \quad
 		\lim _{x\to \pm \infty}  m _2^{\pm } (z;x)
		= \begin{pmatrix} \pm z^{-1} \\    1 \end{pmatrix} .
	\end{aligned}
\end{equation}
Both $\psi _1^{+} (z;x)$ and  $\psi _2^{-} (z;x)$ extend analytically into solutions
of \eqref{eq:zs} for $z \in \C^-$ and  $\psi _2^{+} (z;x)$ and $\psi _1^{-} (z;x)$ extend into solutions
of \eqref{eq:zs} for $z \in \C^+$.

Here $m_1^\pm(z;x)$ and $m_2^\pm(z;x)$ are the unique solution of the integral equations
\begin{align}
	 & \label{eq:eqm11}
 		m  _1^{\pm} (z;x)   = \begin{pmatrix} 1  \\     {\pm} z^{-1} \end{pmatrix}
		+ \int_{\pm \infty}^x X^\pm(x,z) X^\pm(y,z)^{-1}
		\im  \sigma_3 ( Q(y) \mp \sigma_1) m_1^{\pm}(z;y) e^{\im (x-y) \zeta(z)} dy, \\
	&  \label{eq:eqm12}
		m _2^{\pm} (z;x ) = \begin{pmatrix} \pm z^{-1}   \\    1 \end{pmatrix}
		+ \int_{\pm \infty}^x X^\pm(x,z) X^\pm(y,z)^{-1}
		\im  \sigma_3 ( Q(y) \mp \sigma_1) m_2^{\pm}(z;y) e^{-\im (x-y) \zeta(z)} dy.
\end{align}

Furthermore for any $x_0\in \R$ we have that $ z \to m _1^{\pm} (z;x) $
is a continuous map from $ \overline{\C^\mp} \backslash \{  -1,0,1 \}$   (with analytic restriction  in $ \C^\mp $)
into $C^1([x_0, \infty ), \C ^2)\cap W^{1,\infty} ([x_0, \infty ), \C ^2)$ in the + case
and $C^1(  (-\infty , x_0]  , \C ^2)\cap W^{1,\infty}((-\infty , x_0], \C ^2)$ in the -- case.
Similarly, we have that $z \to m _2 ^{\pm}(z;x) $
is a continuous map from $ \overline{\C^\pm} \backslash \{ -1,0,1 \}$ (whose restriction in $ \C^\pm$ is analytic) into $C^1([x_0, \infty ), \C ^2)\cap W^{1,\infty}([x_0, \infty ), \C  ^2)$ in the + case
and $C^1(  (-\infty , x_0]  , \C ^2)\cap W^{1,\infty} ((-\infty , x_0], \C ^2)$  in the -- case.
\end{lemma}

\begin{lemma} \label{lem:jf1b}
Given $n \in \N_0$ and $q - \tanh(x) \in L^{1,n}(\R)$, the map
$q \to   \frac{\partial ^{n}}{\partial z^{n}} m  _1^{+ } (z; \cdot\,  ) $,
with $m^+_1$ as defined in Lemma~\ref{lem:jf1}, is locally Lipschitz continuous from
\begin{equation} \label{eq:eqm13}
	\begin{aligned} &
		\tanh(x) + L ^{1,n}(\R  ) \to L^{\infty}_{loc} ( \overline{\C^-} \backslash \{  -1,0,1  \},
		C^1([x_0, \infty ), \C  ^2)\cap W^{1,\infty} ([x_0, \infty ), \C ^2)) .
	\end{aligned}
\end{equation}
Additionally, the maps
$z \to \frac{\partial ^{n}}{\partial z ^{n}}m _1^{\pm} (z; x  ) $
are continuous from $ \overline{\C^\mp} \backslash \{  -1,0,1\}$
(with analytic restriction  in $ \C^\mp $)
to $C^1([x_0, \infty ), \C ^2)\cap W^{1,\infty} ([x_0, \infty ), \C ^2)$

Similar statements to \eqref{eq:eqm13} hold  for  $q\to   m  _2^{+ } (z; \cdot\,  ) $ and for
$q\to   m  _j^{- } (z; \cdot  \,) $ for $j=1,2$.

\noindent
Specifically, there exists an increasing function $F_{n}(t)$, independent of $q$, such that  \begin{equation} \label{eq:eqm21n_111}
 	\begin{aligned} &
		\left |  \partial ^n _{z} [m_1 ^{+}(z; x) ]    \right |  \le  F_{n}((1+|x| ) ^{n}
		\| q - 1   \| _{L^{1,n}(x, \infty  )}), \quad z \in \overline{\C^-} \backslash \{ -1,0,1\}.
	\end{aligned}
\end{equation}
Furthermore, given potentials $q$ and $\widetilde{q}$
sufficiently close together we have for each $ z \in \overline{\C^-} \backslash \{ -1,0,1\}$,
\begin{equation} \label{eq:eqm21nn_1}
	\begin{aligned} &
	\left |  \partial ^n _{z} [m_1 ^{+}(z; x)- \widetilde{m}  _1 ^{+}(z; x ) ]\right |
	\le  \| q-  \widetilde{q}   \|_{L^{1,n}(x, \infty  )} F_{n}((1+| x | )^{n}
	\| q-  1 \| _{L^{1,n}(x, \infty  )}) .
	\end{aligned}
\end{equation}
Similar estimates hold for the other Jost functions.
\end{lemma}

The above lemmas suggests that the Jost functions exhibit singular behavior for $z$ near $-1,0$, or $1$. The singular behavior of these solutions at $z=0$ plays a non-trivial and unavoidable role in our analysis.
However, as the following lemma makes clear, if the initial data $q$ has an additional finite first moment, then the singularities of the Jost functions at $z =\pm 1$ are removable.

\begin{lemma} \label{lem:jf3}
Given $n \in \N_0$, let $q -  \tanh(x)  \in L ^{1, n+1}(\R)$ and
let $K$ be a compact neighborhood of $  \{  -1,1 \}$  in  $ \overline{\C^-} \backslash \{ 0 \} .$
Set $x^\pm =\max \{ \pm x,0 \}$.
Then there exists a $C$  such that for $z \in K$ we have
\begin{equation} \label{eq:eqm31}
	\begin{aligned}
		& \left | m_1^{+} (z; x)   - \twovec{ 1 }{ z^{-1} } \right |
		\le C  \langle x^- \rangle  e^{C\int _x^{\infty}  \langle y-x \rangle | q (y)- 1 | dy }
		\|  q - 1  \| _{L^{1,1}(x , \infty )},
	\end{aligned}
\end{equation}
\ie,\
the map $ z \to m _1^{+} (z; x) $ extends as a continuous map
to the points $\pm 1$ with values in  \\ $C^1(   [x_0 , \infty )  , \C)\cap W^{1,\infty}
( [x_0 , \infty ), \C)$ for any preassigned $x_0\in \R$.
Furthermore, the map  $q\to   m  _1^{+ } (z; \cdot ) $ is locally Lipschitz continuous from
\begin{equation} \label{eq:eqm13jf3}
	\begin{aligned}
		& \tanh(x) + L ^{1,1}(\R  ) \to L^{\infty}  ( \overline{ \C^-} \backslash \{0\} ,
		C^1([x_0, \infty ), \C)\cap W^{1,\infty} ([x_0, \infty ), \C )).
	\end{aligned}
\end{equation}
Analogous statements hold for  $ m _2^{+ } ( z; x) $ and for
$ m_j^{- } ( z; x) $ for $j=1,2$.

The maps
$ z  \to \partial _{ z }^{n} m _1^{+} (z; x ) $
 and $q \to  \partial _{ z }^{n} m _1^{+} (z; x)$,
also satisfy analogous statements and we have, as in \eqref{eq:eqm21n},
\begin{equation} \label{eq:eqm21n_1}
	\begin{aligned}
	&	\left | \partial_{ z }^{n}  m  _1^{+} (z; x )   \right |
		\le  F _{ n} \lp (1+|x|)^{n+1} \|  q - 1  \|_{L^{1,n+1}(x , \infty )} \rp,
		\quad z \in K.
	\end{aligned}
\end{equation}
\end{lemma}

The final lemma in this section concerns the behavior of the Jost functions as $|z| \to \infty$.
Set
\begin{equation} \label{eq:d+}
	\begin{aligned}
	&   	D_+ (x)= \|   {q}  - 1 \| _{W ^{2,1} (x,\infty )}
			(1+ \|   {q}  - 1   \| _{W ^{2,1} (x,\infty )}) ^2
			e^{\|  q-  1  \| _{L^1 (x , \infty )}}, \\
	&   	D_- (x)=  \|   {q}  + 1 \| _{W ^{2,1} (-\infty,x)}
			(1+ \|   {q}  + 1   \| _{W ^{2,1} (-\infty, x)}) ^2	
			e^{\|  q +1  \| _{L^1 (-\infty, x)}}.
	\end{aligned}
\end{equation}

\begin{lemma} \label{lem:zinfty}
Suppose that $q- \tanh(x)  \in L^{1}(\R  )$ and that $q' \in W^{1,1}(\R)$.
Then as $z \to \infty$ with $\Im z \le  0$ we have
\begin{align}
\label{eq:winfty1}
	m_1^+(z; x) &= e_1  +
	\frac{1}{z} \twovec{ \phantom{-} \im \int_x^\infty \lp 1 - |q(y)|^2 \rp dy }{q(x) }
	+ \bigo{D_+(x) z^{-2}},  \\
\label{eq:winfty2}
	m_2^-(z; x) & = e_2 +
	\frac{1}{z} \twovec{ \widebar q(x) }{ \phantom{-} \im \int^x_{-\infty} \lp 1 - | q(y) |^2  \rp dy}
	+ \bigo{D_-(x) z^{-2}},
\end{align}
and for $\Im z \ge  0$ as $z \to \infty$ we have 	
\begin{align}
\label{eq:winfty3}
	m_1^-(z; x) &= e_1 +
	\frac{1}{z} \twovec{ -\im \int^x_{-\infty} \lp 1 - | q(y) |^2 \rp dy} { q(x) }
	+ \bigo{D_-(x) z^{-2}}, \\
\label{eq:winfty4}
	m_2^+(z; x) &= e_2 +
	\frac{1}{z} \twovec{\widebar q(x) } { -\im \int_x^\infty \lp  1 -| q(y)  |^2  \rp dy}
	+ \bigo{D_+(x) z^{-2}},
\end{align}
where the constant in each $\bigo{D_\pm(x) z^{-2}}$ is independent of $z$.

If $q - \tanh(x) \in L^{1,n}(\R)$ as well, then there exists an increasing function $F_{n}(t)$ independent of $q$ such that as $z \to \infty$
 \begin{equation} \label{eq:eqm21n}
 	\begin{aligned} &
		\left |  \partial ^j _{z} [m_1 ^{+}(z; x) ]    \right |  \le  |z |^{-1}  F_{n}((1+|x| ) ^{n}
		\| q - 1   \| _{L^{1,n}(x, \infty  )}).
	\end{aligned}
\end{equation}
Finally, given two potential $q$ and $\widetilde{q}$
sufficiently close together we have
\begin{equation} \label{eq:eqm21nn}
	\begin{aligned} &
	\left |  \partial ^n _{z} [m_1 ^{+}(z; x )- \widetilde{m}  _1 ^{+}(z; x ) ]\right |
	\le  | z |^{-1} \| q-  \widetilde{q}   \|_{L^{1,n}(x, \infty  )} F_{n}((1+| x | )^{n}
	\| q-  1 \| _{L^{1,n}(x, \infty  )}) \text{ for $0\le j \le n$}.
	\end{aligned}
\end{equation}
Similar estimates hold for the other Jost functions.

\end{lemma}

The previous lemma and the symmetry \eqref{eq:symm2} imply the following corollary which
describes the singularities of the Jost solutions at the origin.
\begin{corollary}
\label{lem:w0}
Let $q$ be as in Lemma \ref{lem:zinfty}. Then for $z \in \C^-$, as $z\to 0$ we have
 \begin{equation} \label{eq:w01}
 	\begin{aligned} &
		m_1 ^{+}(z; x)  = \frac{1}{z} e_2  + \bigo{1}   
		\quad \text{and} \quad
		m_2 ^{-}(z; x)   = -\frac{1}{z} e_1  + \bigo{1}
	\end{aligned}
\end{equation}
where $| \bigo{ 1} |\le F (\|   {q}  - 1  \| _{W ^{2,1} ( x, \infty)}  ) $,
and for $z \in \C^+$, as $z \to 0$ we have
\begin{equation} \label{eq:w02}
	\begin{aligned} &
	m_1 ^{-}(z; x)  = -\frac{1}{z} e_2  +  \bigo{1}
	\quad \text{and} \quad
	m_2 ^{-}(z; x)   = \frac{1}{z} e_1  +  \bigo{1}
	\end{aligned}
\end{equation}
where $|\bigo{1} | \le F (\|   {q}  + 1\| _{W ^{2,1} ( - \infty ,x)}  ) $
for some growing functions $F(t)$.
\end{corollary}

\section{The scattering data}
\label{sec:ab}

We start with the following elementary lemma.
\begin{lemma} \label{lem:symm}
Let $q - \tanh(x) \in L^1(\R)$. Then
\begin{enumerate}[1.]
	\item For $z \in \R \backslash \{ -1,0,1 \}$ both of the matrix-valued functions
	\begin{equation}\label{eq:JostMatrix}
		\Psi^\pm(z; x ) = \( \psi^\pm_1(z; x),\ \psi^{\pm}_2(z; x) \) = \( m_1^\pm(z; x),\ m_2^\pm(z; x) \)
		e^{ - \im \zeta(z) x \sigma_3}
	\end{equation}
	are nonsingular solutions of \eqref{eq:zs} and
	\begin{equation} \label{eq:Jostdet}
		\det \Psi_\pm = \det \Psi_\pm(z) = 1 - z^{-2}.
	\end{equation}

	\item For $z \in \overline{\C^+} \backslash \{ -1,0,1 \}$ the Jost functions $\psi_j^\pm$ satisfy the symmetries
	\begin{align}
	\label{eq:symm1}
		&\left\{ \begin{aligned}
			\psi_1^-(z; x) &= \sigma_1 \overline{\psi_2^-(\overline{z}; x)} \\
			\psi_2^+(z; x) &= \sigma_1 \overline{\psi_1^+(\overline{z}; x)} \bigskip
		\end{aligned} \right. 
	\shortintertext{and}
	\label{eq:symm2}
		& \left\{ \begin{aligned}
			\psi_1^-(z; x) &= - z^{-1} \psi_2^-(  z^{-1} ; x) \\
			\psi_2^+(z; x) &= \phantom{-} z^{-1} \psi_1^+(  z^{-1} ; x).
		\end{aligned} \right.
	\end{align}
	
\end{enumerate}
\end{lemma}

\begin{proof} 
The matrices $\Psi^\pm$ are solutions of \eqref{eq:zs}, which follows from Lemma~\ref{lem:jf1}. To establish \eqref{eq:Jostdet} and thus that $\Psi^\pm$ is nonsingular, observe that $\Tr(\La) = 0$, where $\La$ is the matrix \eqref{eq:laxL} appearing in \eqref{eq:lax1}, so that $\det \Psi^\pm (z; x) = \det \Psi^\pm(z)$. Finally, $\lim_{x \to \pm \infty} \det \Psi^\pm = \det B_\pm = 1 - z^{-2}$.
	
To prove the symmetries \eqref{eq:symm1}-\eqref{eq:symm2} start with $z \in \R \backslash \{-1,0,1\}$. The symmetries of the Lax matrix: $\La(z) = \sigma_1 \overline{ \La(\widebar z)} \sigma_1 = \La(z^{-1})$ and of the ``free" solution:
$X^\pm(x,z) = \sigma_1 \overline{X^\pm(x,\widebar z)} \sigma_1 = \pm z^{-1} X^\pm(x,z^{-1})\sigma_1$
imply that for $z \in \R \backslash \{-1,0,1 \}$ the Jost matrices satisfy
\begin{equation*}
	\Psi^\pm(z; x) = \sigma_1 \overline{\Psi^\pm(  \widebar  z; x)} \sigma_1
	= \pm z^{-1} \Psi( z^{-1} ; x) \sigma_1.
\end{equation*}
Analytically extending each column vector solution $\psi_j^\pm(z; x)$ off the real axis into the half plane indicated by Lemma~\ref{lem:jf1} gives \eqref{eq:symm1}-\eqref{eq:symm2}.

\end{proof}

\begin{corollary}
	Let $q - \tanh(x) \in L^1(\R)$. Then each of the Jost functions $\psi_k^\pm(z; x)$ satisfy
	\begin{equation} \label{eq:symm23}
		\overline{\psi_k^\pm (  \overline{z^{-1}} ; x )} = \pm z \sigma_1 \psi_k^\pm(z; x)
	\end{equation}
	upon reflecting $z$ through the unit circle in the half-plane in which each Jost function is defined.
\end{corollary}

The columns of $\Psi^+(z; x)$ and $\Psi^-(z; x)$ each form a solution basis of \eqref{eq:zs} for $z \in \R \backslash \{-1,0,1\}$. It follows that the matrices must satisfy the linear relation
\begin{equation}\label{eq:coeff}
	\Psi^-(z; x) = \Psi^+(z; x) S(z), \qquad
	S(z) = \begin{pmatrix} a(z) & \overline{ b(z) } \\ b(z) & \overline{a(z)} \end{pmatrix},
	\quad z \in \R \backslash \{-1,0,1\}
\end{equation}
where the form of the \emph{scattering matrix} $S(z)$ follows from \eqref{eq:symm1}.
The scattering coefficients $a(z)$ and $b(z)$ define   the \emph{reflection coefficient}
\begin{equation} \label{eq:r}
	r(z) := \frac{ b(z) }{ a(z) }.
\end{equation}	
The following lemma records several important properties of   $a(z)$ and $b(z)$.

\begin{lemma} \label{lem:coeff}
Let $z \in \R \backslash \{-1,0,1 \}$ and $a(z)$, $b(z)$, and $r(z)$ be the   data in \eqref{eq:coeff}-\eqref{eq:r} generated by some $q \in \tanh(x) + L^1(\R)$. Then
\begin{enumerate}[1.]
	\item The scattering coefficients can be expressed in terms of the Jost functions as
	\begin{equation}\label{eq:wron2}
		a(z) = \frac{   \det [ \psi_1^-(z; x), \psi_2^+(z; x) ] }{ 1- z^{-2} },
		\qquad
		b(z) = \frac{ \det [ \psi_1^+(z; x), \psi_1^-(z; x) ]  }{ 1- z^{-2} } .
	\end{equation}
It follows that $a(z)$ extends analytically to $z \in \C^+$
	while $b(z)$ and $r(z)$ are defined only for $z \in \R \backslash \{-1,0,1 \}$.
	
	\item For each $z \in \R \backslash \{-1,0,1 \}$
	\begin{equation}\label{eq:coeff3}
		| a(z) |^2 - | b(z) |^2 = 1.
	\end{equation}
	In  particular, for  $z \in \R \backslash \{-1,0,1 \}$
we have
\begin{equation}\label{eq:coeff31}
		| r(z) |^2  =  1- | a(z) | ^{-2}<1.
	\end{equation}

	\item The scattering data satisfy the symmetries
	\begin{equation}\label{eq:symm31}
		-\overline{ a(\widebar z^{-1} ) }  =  a(z), \qquad
		-\overline{ b(\widebar z^{-1} ) } =  b(z), \qquad
		\overline{ r(\widebar z^{-1} ) } =  r(z)
	\end{equation}
	wherever they are defined.

	\item If additionally $q' \in W^{1,1}(\R)$, then for $z \in \overline{\C^+}$,
	\begin{align}
	\label{eq:as11}	
		& \lim _{ z \to \infty } \left (    a(z) - 1  \right )z
		=\im  \int _{\R } \left (   |q(x)|^2 - 1  \right )  dx, \\
	\label{eq:as01}
		& \lim _{ z \to 0 } ( a(z)+1) z ^{-1}
		= \im \int _{\R } \left (   |q(x)| ^2 - 1 \right )  dx,
	\end{align}
	and for $z \in \R$
	\begin{equation}\label{eq:asb11}
		\begin{aligned}
			& |   b(z)|= \bigo{ |z| ^{-2} } \text{  as $|z|\to \infty $,}  \\
			&  |   b(z)|= \bigo{ |z| ^{ 2}} \text{  as $|z|\to 0 $.}
		\end{aligned}
	\end{equation}
\end{enumerate}
\end{lemma}	

\begin{proof}
The first property follows from applying Cramer's rule to \eqref{eq:coeff} and using \eqref{eq:Jostdet}; one then observes that Lemma~\ref{lem:jf1} implies that the formula for $a(z)$ is analytic for $z \in \C^+$. The second property is just the fact that $\det S = 1$ which follows from taking the determinant on each side of \eqref{eq:coeff} using \eqref{eq:Jostdet}. The symmetry conditions follow immediately from \eqref{eq:wron2} after using \eqref{eq:symm1}-\eqref{eq:symm2};
for instance
\begin{align*}
	 \overline{a( \widebar z^{-1} )} = \frac{ \det \left[
	 \overline{ \psi_1^-(  \widebar z^{-1} ; x)} ,\ \overline{ \psi_2^+(  \widebar z^{-1}; x)} \right] }
	 { 1 - z^2}
	 &= \frac{1}{1- z^{2}} \det \left[  \sigma_1 . ( -z\psi_1^-(z; x), z\psi_2^+ (z; x) ) \right] \\
	 &= -\frac{1}{1-z^{-2}} \det \left[ \psi_1^-(z; x), \, \psi_2^+ (z; x)  \right] = - a(z).
\end{align*}
To prove \eqref{eq:as11} first observe that
\[
	| q(y) \pm 1 |^2 - ( 2 \pm q(y) \pm \widebar q(y))  = |q(y)^2| - 1.
\]
Inserting \eqref{eq:winfty3}-\eqref{eq:winfty4} from Lemma~\ref{lem:zinfty}
into \eqref{eq:wron2} gives
\[
\begin{aligned}
	(1-z^{-2}) a(z)
	&= \det \begin{bmatrix}
 	1 + \im z^{-1} \int_{-\infty}^x \lp |q(y)|^2 - 1 \rp dy &  z^{-1}  \widebar q(x) \\
	z^{-1} q(x) & 1 + \im z^{-1} \int_x^\infty \lp |q(y)|^2 -1 \rp dy
	\end{bmatrix} + \bigo{z^{-2}}  \\
	&= 1 + \im z^{-1} \int_\R  \lp |q(y)|^2 -1 \rp \, dy + \bigo{z^{-2}}.
\end{aligned}
\]
To prove \eqref{eq:as01} write $z = \widebar \varsigma^{-1}$ and use
\eqref{eq:symm31} and \eqref{eq:as11}; the formulae for $b(z)$ in \eqref{eq:asb11} are proved similarly.
\end{proof}

Though Lemma~\ref{lem:jf3} gives conditions on $q$ which guarantee that the Jost functions $\psi_j^\pm(z;x)$ are continuous for $z \to \pm 1$, the scattering coefficients $a(z)$ and $b(z)$ will generally have simple poles at these points due to the vanishing of the denominators in \eqref{eq:wron2}. Moreover, their residues are proportional: the symmetry \eqref{eq:symm2} implies that $\psi_1^+(\pm1;x) = \pm \psi_2^+(\pm1;x)$, which in turn gives
\begin{equation}\label{eq:coeff sing}
	\begin{aligned}
		a(z) &= \phantom{\pm}\frac{a_\pm}{z \mp 1} + \bigo{1}, \\
		b(z) &= \mp \frac{a_\pm}{z \mp 1} + \bigo{1},
	\end{aligned}
	\qquad a_\pm = \det [ \psi_1^-( \pm1 ; x),\, \psi_2^+( \pm1  ; x) ].
\end{equation}	
In this generic situation the reflection coefficient remains bounded at $z = \pm 1$ and we have
\begin{equation}\label{eq:r sing}
	\lim_{z \to \pm 1} r(z) = \mp 1.
\end{equation}

The next lemma show that, given data $q_0$ with sufficient smoothness and decay properties, the reflection coefficient will also be smooth and decaying.

\begin{lemma} \label{lem:a3bis}
For any given
$q \in \tanh ( x) +L ^{1, 2}(\R  ) $,  $q'\in  W^{1,1}(\R  )$
  we have $r \in H^1(\R   ) $.
\end{lemma}

\begin{proof}
Because $\| r \| _{  L^\infty (\R )}\le 1$ and,
by Lemma \ref{lem:coeff}, we have  $r(z) = \bigo{z^{-2}}$ as $z\to \pm \infty$ it's clear that 
\begin{equation}\label{eq:L2}
   q\in \tanh ( x) +\Sigma _2 \Rightarrow  r\in L^2 (\R ) .
\end{equation}
It remains to show that the derivative $r'$ is also $L^2(\R)$.

For any $\delta_0>0$ sufficiently small, the maps
\begin{equation}\label{eq:wronsk}
 \text{$q \to \det [ \psi ^{-} _1(z; x),\psi _2^{+} (z; x) ] $
\quad and \quad
$q \to \det [ \psi ^{+} _1(z; x),\psi _1^{-} (z; x) ]  $}
\end{equation}
are locally Lipschitz maps from
\begin{equation} \label{eq:a13}
	\{ q: q' \in W^{1,1}(\R  ) \text{ and } q \in  \tanh(x) + L ^{1, n+1}(\R )  \}
	\to W^{n,\infty} ( \R \backslash (-\delta _0 ,\delta _0 ))   \text{  for $n\ge 0$.}
\end{equation}
Indeed, $q \to \psi ^{+} _1(z,0)$  is, by Lemmas \ref{lem:jf1b} and \ref{lem:jf3} (\cf\ in particular \eqref{eq:eqm13} and \eqref{eq:eqm21n_1}), a locally Lipschitz  map with values in
$W^{n,\infty} (\overline{\C}^- \backslash D(0,\delta _0 ), \C^2  )$.
For  $q \to \psi ^{+} _2(z,0)$  and $q \to \psi ^{-} _1(z,0)$ the same is true but with $\C _-$ replaced by $\C _+$.
This and \eqref{eq:as11}--\eqref{eq:asb11}    implies that
$q\to r(z)$ is a locally Lipschitz map from the domain in \eqref{eq:a13}
 into
\begin{equation*}
 W^{n,\infty} ( I _{\delta _0}) \cap H^{n } ( I _{\delta _0}) \text{ with }
  I _{\delta _0}:= \R \backslash ((-\delta _0 ,\delta _0 )   \cup (1-\delta _0 ,1+\delta _0 ) \cup  (-1-\delta _0 ,-1 +\delta _0 ) .
\end{equation*}
Now fix $\delta _0>0$ so  small that the 3 intervals
$\text{dist}(z, \{\pm 1\}) \le  \delta _0 $
and $|z|\le \delta _0  $ have empty intersection. In the complement of their union
\begin{equation}\label{3bis}
	| \partial ^{j}_{z} r(z)| \le
	C_{\delta _0 } \langle z \rangle ^{-1}
	\text{ for $j=0,1 $}
\end{equation}
by \eqref{eq:eqm21nn}, its analogues for the other Jost functions, and the discussion above.

\noindent
Let $|z-1|< \delta _0$.
Then, using the $a_+$ in \eqref{eq:coeff sing} we have
\begin{equation} \label{eq:refl13}
	\begin{aligned} &
  	r (z)  =  \frac{b(z)}{a(z)}  =
	\frac{ \det [ \psi ^{+} _1(z; x),\psi _1^{-} (z; x ) ] }{\det [ \psi ^{-} _1(z; x),\psi _2^{+} (z; x) ] }
	= \frac{-a_++\int _1^z   F (s)ds} { a_++\int _1^z G (s)ds}
	\end{aligned}
\end{equation}
for
$F (z) =\partial _z\det [  \psi ^{+} _1(z; x),\psi _1^{-} (z; x ) ]$
and
$G (z)=\partial _z\det [ \psi ^{-} _1(z; x),\psi _2^{+} (z; x) ] $.
If  $a_+\neq 0$ then  it is clear from the above formula that $r'(z)$
is defined and bounded around $1$.

\noindent If $a_+=0$  we have
\begin{equation} \label{eq:rq}
	\begin{aligned} &
 		 r (z)  =       \frac{ \int _{1}^{z} F(s) ds} { \int _{1}^{z} G(s) ds }    .
	\end{aligned}
\end{equation}
Now, $a_+=0$  is the same as  $[ \psi ^{-} _1(z; x),\psi _2^{+} (z; x) ] _{|z=1}=0$.
Differentiating \eqref{eq:ida} at $z=1$    we get \begin{equation*}
   2 a(1) =   \partial _z\det[ \psi_1^-(z;x), \psi_2^+(x;z) ]  _{|z=1} =G(1).
\end{equation*}
This implies that $G(1)\neq 0$, since otherwise $| a(1) |^2 - | b(1) |^2 = 1$, which holds
by continuity at $z=1$, would not be true. It follows that the derivative $r'(z)$ is bounded
near 1.

The same discussion holds at $-1$. At $z=0$ we can use the symmetry $r(z ^{-1})=\overline{r}(z)$ to conclude that $r$ vanishes at the origin. It follows that $r' \in L^2(\R)$.
\end{proof}

We also have the following result, which is used later in the proof.
\begin{lemma}\label{lem:log bound}
For any initial data $q_0$ such that $q_0 - \tanh (x) \in \Sigma_2$ the reflection coefficient satisfies
\[
	\| \log( 1 - |r|^2 ) \|_{L^p(\R)} < \infty \text{ for any $p \geq 1$.}
\]

\end{lemma}

\begin{proof}
Fix $\delta \in (0,1)$. Let $K = \{ z \in \mathbb{R} \, : \, 1- |r(z)|^2 \in [\delta, 1] \}$ and let $\chi$ denote the indicator function of $K$. As $r \in L^2(\R)$ clearly $(1-\chi)$ has finite support containing intervals surrounding $z = \pm 1$.  Using the concavity of the logarithm, for $z \in K$ we have
$
	|\log( 1 - |r(z)|^2)|  \leq M_\delta  |r(z)|^2,\ M_\delta = \frac{\log(1/\delta)}{1-\delta}.
$
The previous inequality and the identity $1 - |r(z)|^2 = |a(z)|^{-2}$ give
\begin{gather*}
	\left\| \log(1 - | r |^2 ) \right\|_{L^p(\R)} =
	\left\| \chi \log(1 - | r |^2 ) \right\|_{L^p(\R)}
	+ \left\| (1-\chi) \log(1 - | r |^2 ) \right\|_{L^p(\R)} \\
	\leq M_\delta \| \chi\, r^{2} \| _{L^p(\R)}
	+  \left\| (1-\chi) \log(  | a |^2 ) \right\|_{L^p(\R )}  \\
	\leq M_\delta \| r \|_{L^2(\R)}^{2/p}
	+ \left\| (1-\chi) \log( | a |^2 ) \right\|_{L^p(\R)}, 	
\end{gather*}
where the last step follows from observing that $\| r \|_{L^\infty(\R)} \leq 1$.

To estimate the second term we observe that by the identity
\begin{equation}\label{eq:ida}
   (z^2-1) a(z) = z^2 \det[ \psi_1^-(z;x), \psi_2^+(x;z) ]
\end{equation}
we have
$(z^2-1) a(z)   \in L^\infty_{loc}(\R)$
for initial data $q_0 \in \tanh(x) + L^{1,1}(\R)$.
It follows that
\begin{gather*}
	\left\| (1-\chi) \log( | a |^2 ) \right\|_{L^p(\R)} \leq
	\left\| (1-\chi) \log \left( \frac{1}{|\lozenge^2-1|} \right) \right\|_{L^p(\R)}
	+ \left\| (1-\chi) \log \left( \left| (\lozenge^2-1) a \right| \right) \right\|_{L^p(\R)} \\
	\leq \left\| (1-\chi) \log \left( \frac{1}{|\lozenge^2-1|} \right) \right\|_{L^p(\R)}
	+ \left\| \log \left( \left| (\lozenge^2-1) a \right| \right) \right\|_{L^\infty_{loc}(\R)}
	\| 1-\chi \|_{L^1(\R)}^{1/p}.
\end{gather*}
\end{proof}

\begin{remark}
\label{rem:reg1} In terms of the regularity needed, among other things, in the latter proofs in this paper
we use often the fact, proved in Lemma \ref{lem:a3bis}, that $r\in H^1(\R )$.
Another fact, used only to prove inequality \eqref{eq:extR21}, is  that
the Wronskians in \eqref{eq:wronsk} 
have bounded  derivatives  up to order 2 in  a small  neighborhood of $z=1$  in $\R$.
For both facts it is sufficient to require that  $q_0 \in \tanh(x) + \Sigma_4$.
\end{remark}

We conclude this subsection with a result  on  small perturbations of an $N$-soliton solution.
\begin{lemma} \label{lem:numbzeros}
Given an $M$-soliton $q^{(sol),M}(x,t)$ and initial data $q_0(x)$ satisfying the hypotheses of Theorem \ref{thm:main2} the number of solutions
$z\in \C _+$ of $\det \left[ \psi_1^-(z;x,0), \psi_2^+( z;x,0) \right] =0$,
where the Jost functions correspond to $q_0(x)$, is at least $M$ and is finite when $q_0$ satisfies the hypotheses of Theorem \ref{thm:main2} for $\varepsilon _0$ small enough.
Furthermore \eqref{eq:prop1} holds.
\end{lemma}
\begin{proof}
In Lemma~\ref{lem:simple} it is proved that when $q_0 - \tanh(x)  \in L ^{1,2}(\R )$ then the number of zeros is finite.
The other statements are elementary consequences of the theory which we review in Sections \ref{sec:jost}--\ref{sec:ab}, and specifically of the Lipschitz dependence of the Jost functions in terms of
$q_0$ and of the specific form of the determinants in the case of a multisoliton, which follows immediately from \eqref{eq:wron2} and the formula for $a(z)$, see \eqref{eq:afactor} below.
\end{proof}

\subsection{The discrete spectrum}
At any zero  $z=z_k \in \overline{\C^+}$ of $a(z)$ it follows from \eqref{eq:wron2} that the pair $\psi_1^-(  {z}_k; x  )$ and $\psi_2^+( {z}_k; x )$ are linearly related; the symmetry \eqref{eq:symm1} implies that $\psi_2^-(\overline{z}_k; x)$ and $\psi_1^+(\overline{z}_k; x)$ are also linearly related. That is, there exists a constant $\gamma_k \in \C$ such that
\begin{equation} \label{eq:gamma}
	\psi_1^-(z_k, x  )  = \gamma_k \psi_2^+(z_k; x),
	\qquad
	\psi_2^-(\overline{z}_k; x)  = \widebar \gamma_k \psi_1^+(\overline{z}_k; x).
\end{equation}
These $\gamma_k$ are called the \emph{connection coefficients} associated to the discrete spectral values $z_k$.

If $z_k \in \C^+$ then it follows that $\psi_1^-( {z}_k; x)$ and $\psi_2^-(\overline{z}_k; x)$ are $L^2(\R)$ eigenfunctions of \eqref{eq:zs} with eigenvalue $\lambda(z_k)$ and $\widebar \lambda(z_k)$ respectively.
If $z_k \in \R$ then $\psi( {z}_k; x)$ is bounded but not $L^2(\R)$ and we say that $z_k$ is an embedded eigenvalue.
However, it follows from \eqref{eq:coeff3} and \eqref{eq:as01} that $|a(z)| \geq 1$ for $z \in \R \backslash \{ -1,1 \}$, so the only possible embedded eigenvalues are $\pm 1$.
Then as \eqref{eq:zs} is self-adjoint, the non-real zeros of $a(z)$ in $\C^+$ are restricted to the unit circle, i.e.,  $|z_k| = 1$, so that $\lambda(z_k)$ is real. The following lemma demonstrates that, unlike the case of vanishing data for focusing NLS, the discrete spectral data takes a very restricted form.

\begin{lemma}\label{lem:simple}
	Let $q - \tanh(x) \in L^{1,2}(\R)$. Then 
	\begin{enumerate}[1.]
		\item The zeros of $a(z)$ in $\C^+$ are simple and finite.
		\item At each $z_k$, a zero of $a(z)$:
			\begin{enumerate}[i.]
				\item
				$\pd{a}{\lambda} (z_k)$ and  $\gamma_k$ are pure imaginary;
				\item their arguments satisfy
				\begin{equation}\label{eq:gamma sign}
					\sgn ( -\im \gamma_k) = - \sgn \( - \im \pd{a}{\lambda}(z_k) \).
				\end{equation}
			\end{enumerate}
	\end{enumerate}
\end{lemma}	

\begin{proof}
 	Suppose $z_k$ is a zero of $a(z)$, and $\gamma_k$ the connection coefficient in \eqref{eq:gamma}. Then as $z_k$ lies on the unit circle we have $\overline{z_k^{-1} } = z_k$. Applying \eqref{eq:symm23}
to \eqref{eq:gamma}	gives
\begin{equation*}
	\begin{gathered}
		\overline{ \psi_1^-(\overline{z}_k^{-1}; x ) }  = \widebar \gamma_k \overline{ \psi_2^+(\overline{z}_k^{-1}; x ) } \\
		-z_k \sigma_1 \psi_1^-( {z}_k; x) = \widebar \gamma_k z_k \sigma_1 \psi_2^+( {z}_k; x) \\
		\psi_1^-( {z}_k; x) = - \widebar \gamma_k \psi_2^+( {z}_k; x).
	\end{gathered}
\end{equation*}	
Comparing this to \eqref{eq:gamma} shows that $\widebar \gamma_k = - \gamma_k$, or $\gamma_k \in i\R$.

To prove the remaining facts, note that $q - \tanh(x) \in L^{1,1}(\R)$ implies $\pd{a}{\lambda}$ exists and we have from \eqref{eq:wron2}
\begin{equation*}
	\pd{a}{\lambda} \bigg|_{z=z_k}
		= \frac{ \det \left[ \partial_\lambda \psi_1^-( {z}  ; x),\ \psi_2^+( {z} ; x) \right] +
		\det \left[  \psi_1^-( {z} ; x),\ \partial_\lambda \psi_2^+( {z} ; x) \right] }{1-z^{-2} } \bigg|_{z=z_k}  \, .
\end{equation*}
Using \eqref{eq:laxL} one finds that
\begin{align*}
	\pd{}{x} \det \left[ \partial_\lambda \psi_1^-,\ \psi_2^+ \right] &=
		\det \left[ \La_\lambda \psi_1^-,\ \psi_2^+ \right] +
		\det \left[ \La \partial_\lambda \psi_1^-,\ \psi_2^+ \right] +
		\det \left[ \partial_\lambda \psi_1^-,\ \La \psi_2^+\right] \\
	 &= -\im \det \left[\sigma_3 \psi_1^-,\ \psi_2^+ \right] 
	\shortintertext{and}
	\pd{}{x} \det \left[ \psi_1^-,\ \partial_\lambda  \psi_2^+ \right] &=
		\det \left[ \psi_1^-,\ \La_\lambda \psi_2^+ \right] +
		\det \left[ \La \psi_1^-,\ \partial_\lambda  \psi_2^+ \right] +
		\det \left[ \psi_1^-,\ \La \partial_\lambda  \psi_2^+\right] \\
	 &= -\im \det \left[\psi_1^-,\ \sigma_3  \psi_2^+ \right]
\end{align*}
where the cancellation in each equality follows from observing that $\adj \La = - \La$.\footnote{
$\adj M$ denotes the adjugate of the matrix $M$, it satisfies $M \adj M = (\det M) I$.}
Recalling that at each $z_k$ the columns are linearly related by \eqref{eq:gamma} and decay exponentially as $|x| \to \infty$,
\begin{align*}
	 \det \left[ \partial_\lambda \psi_1^-,\ \psi_2^+ \right]  &=
	 -\im \gamma_k \int^{x}_{-\infty} \det [ \sigma_3 \psi_2^+( {z}_k;s),
	  \psi_2^+({z}_k; s) ] ds, \\
	  \det \left[  \psi_1^-,\ \partial_\lambda \psi_2^+ \right]  &=
	 -\im \gamma_k \int_{x}^{\infty} \det [ \sigma_3 \psi_2^+({z}_k; s),
	  \psi_2^+( {z}_k; s) ] ds.
\end{align*}	
Then using \eqref{eq:symm23} to write $\psi_2^+( {z}_k; x) = z_k^{-1} \sigma_1 \overline{ \psi_2^+( {z}_k; x) }$
in the first column of the determinants, we have, after putting the terms together,
\begin{equation}\label{eq:dera}
	\pd{a}{\lambda} \bigg|_{z = z_k}
	= \frac{- \im \gamma_k}{2 \zeta(z_k)} \int_\R  | \psi_2^+( {z}_k; x)  |^2 dx.
\end{equation}
Recalling that both $\gamma_k$ and $\zeta(z_k)$ are imaginary, \eqref{eq:dera} is both nonzero and imaginary.
The simplicity of the zeros of $a$ and the signature restriction on $\gamma_k$ follow immediately.

To prove that the number of zeros is finite, observe that if the number were infinite they would necessarily accumulate at one (or both) of $z = \pm 1$.
From \eqref{eq:wronsk}-\eqref{eq:a13} in Lemma~\ref{lem:a3bis} for $q - \tanh(x) \in L^{1,2}(\R)$ the functions
$$
	f_k(\theta) = \left | \partial_\theta^k \det[ \psi_1^-( e^{i\theta}; x), \psi_2^+( e^{i \theta}; x) ] \right|, 	\quad k = 0,1
$$
are continuous for $\theta \in [0,\pi]$.
Now if $z=1$ is an accumulation points there exist sequences $\theta^{(k)}_j$, $k=0,1,$ with $\lim_{j\to \infty} \theta^{(k)}_j = 0$ and $f_k(\theta^{(k)}_j) = 0$ for each
$j$.
It then follows from \eqref{eq:wron2} and the continuity in \eqref{eq:a13} that $a(z) = \littleo{1}$ as $z \to 1$.
This contradicts the fact that $|a(z)|^2 \geq 1$ for $z \in \R \backslash \{-1,0,1\}$ by \eqref{eq:coeff3}.
The proof when $z=-1$ is an accumulation point is identical.
\end{proof}

\begin{remark}
The argument given above to prove that the number of zeros of $a(z)$ is finite is essentially the same as that given in \cite{DPMV}. Our contribution is a weaker condition on the potential in order to obtain smooth derivatives, which allows us to prove the result for $q_0-\tanh(x) \in L^{1,2}(\R)$ instead of the $L^{1,4}(\R)$ condition appearing in \cite{DPMV}.
\end{remark}

The zeros of $a(z)$ are simple, finite and restricted to the circle.
As $a(z)$ is analytic in $\C^+$, and approaches unity for large $z$, it admits an inner-outer factorization, see  \cite{FT} p.50, which using \eqref{eq:coeff3} takes the form
\begin{equation}\label{eq:afactor}
	a(z) = \prod_{k=0}^{N-1} \( \frac{ z - z_k}{z - \widebar z_k} \)
	\exp \( -\frac{1}{ 2 \pi \im} \int_\R \frac{ \log (1 - |r(s)|^2 )}{s-z} ds\) ,
\end{equation}
where $\{ z_k \}_{k=0}^{N-1}$ are the zeros of $a(z)$ in $\C^+$. This trace formula implies a dependence between the discrete spectrum $z_k$ and the reflection coefficient. Using  \eqref{eq:as01},
$a(0):=\lim _{z\to 0}a(z)=-1$, which gives
\begin{equation} \label{eq:thetacond}
	 \prod_{k=0}^{N-1} z_k^2
	 = a(0) \exp \( \frac{1}{ 2 \pi \im} \int_\R \frac{ \log (1 - |r(s)|^2 )}{s} ds\)  =
- \exp \( \frac{1}{ 2 \pi \im} \int_\R \frac{ \log (1 - |r(s)|^2 )}{s} ds\)  .
\end{equation}
The more general case  $a(0) =e^{\im \theta}$ is  the ``$\theta$-condition"  in \cite{FT}, formula  (7.19) in Ch. 2.
\subsection{Time evolution of the scattering data}
\label{subsec:tev}
Thus far we have considered only fixed potentials $q = q(x)$. The advantage of the inverse scattering transform is that if $q(x,t)$  evolves according to \eqref{eq:nls} then the evolution of the scattering data is linear and trivial
as we  see now.

By  Theorem \ref{thm:existence} we have $q(x,t)-\tanh \left ( x \right ) \in  C^1([0,\infty ) ,\Sigma _2) $. It can be shown that this implies that
the Jost functions $m^{\pm}_j( z;x,t)$ in Sect.~\ref{sec:jost} are differentiable in $t$
with $\partial _{t}m^{\pm}_j(z;x,t)\stackrel{x\to \pm \infty}{\rightarrow}0$.
This can be seen applying $\partial _t$ to \eqref{eq:eqm11}--\eqref{eq:eqm12}
and obtaining a  Volterra equation for  $\partial _{t}m^{\pm}_j(z;x,t)$.
By standard arguments, see for example \cite{FT,DZ2,GZ},  which we sketch now, the evolution of the scattering coefficients and discrete data are as follows:
\begin{equation} \label{eq:coeffevo}
	\begin{aligned}
		a(z,t) &= a(z,0) , \\
		b(z,t) &= b(z,0) e^{-4 \im \zeta(z) \lambda(z) t}  ,\\
		r(z,t) & = r(z,0) e^{-4 \im \zeta(z) \lambda(z) t} .
	\end{aligned}	
	\qquad
	\begin{aligned}
	 \qquad z_k(t) &= z_k(0) ,\\
	\gamma_k(t) &= \gamma_k(0)  e^{-4 \im \zeta(z) \lambda(z) t},
	\end{aligned}
\end{equation}
In particular here we sketch the first two equalities on the left.
Due to \eqref{eq:lax0} we can write   equalities  $(\im \partial _t +  \mathcal{B})\Psi ^{\pm} (z;x,t)=\Psi ^{\pm}(z;x,t)C_\pm ( z,t)   $, with the $\Psi ^{\pm}$ in \eqref{eq:JostMatrix}.
 This yields
\begin{equation*}
	\begin{aligned}
	&    \partial _t m ^{\pm} (z;x,t)
	+ \im    \mathcal{B}(z; x,t)   m ^{\pm} (z;x,t)
	= m ^{\pm} (z;x,t) e^{ -\im \zeta x\sigma _3 } C_\pm ( z,t)e^{ \im \zeta x\sigma _3 }  .
	\end{aligned}
\end{equation*}
Using
\begin{equation*}
	\begin{aligned}
	&  \lim_{x \to \pm \infty}  m^{\pm} (z;x,t)
	= 1\pm \frac{\sigma _1}{z}, \  \lim_{x \to \pm \infty} \partial _t m^{\pm} (z;x,t)=0
	\text{  and }
	\lim_{x \to \pm \infty} \mathcal{B}(z; x,t)
	= 2\lambda  \sigma _3(\mp \sigma _1 -\lambda )
\end{aligned}
\end{equation*}
we obtain that  $C_\pm ( z,t) $
  is diagonal  with
\begin{multline*}
	C_\pm ( z, t) = -2\im \lambda \left (1 \pm  \frac{\sigma _1}{z} \right )^{-1}
	\sigma _3 ( \lambda \mp \sigma _1 ) \left( 1 \pm  \frac{\sigma _1}{z} \right )
	=  -\frac{2\im \lambda z }{2 \zeta}
	\left (1 \mp  \frac{\sigma _1}{z}\right )^2  (\lambda \pm \sigma_1 ) \sigma_3 \\
	= -\frac{\im \lambda z}{\zeta} \left(1 +z^{-2}   \mp  2\frac{\sigma _1}{z} \right)
	 (\lambda \pm \sigma_1 ) \sigma _3
	=  -\frac{2\im \lambda  }{\zeta}
	\left ( \lambda   \mp    {\sigma _1} \right ) (\lambda \pm \sigma_1 ) \sigma _3
	= -\frac{2\im \lambda  }{   \zeta} (\lambda ^2 -1)  \sigma _3  = -2 \im \lambda \zeta \sigma _3.
\end{multline*}
  Applying now
  $\im \partial _t +  \mathcal{B}$
to the first equality in \eqref{eq:coeff}, that is to $\Psi ^{-}(z;x,t)=\Psi ^{+}(z;x,t)S ( z,t)$, after elementary computations we get $\partial _t S=2 \im \lambda \zeta [\sigma _3 ,S]$. This yields  the     left column  in \eqref{eq:coeffevo}.

\section{Inverse scattering: set  up of the Riemann Hilbert problem}
\label{sec:rh}

For $z \in \C \backslash \R$, for $q(x,t)$ the solution to  \eqref{eq:nls},
and for $m_j^{\pm }(  z ;x,t)$, $j=1,2,$ the (normalized) Jost functions we set
\begin{equation} \label{eq:defm}
	m(z) = m(z; x,t) := \begin{cases}
		\begin{pmatrix}
			\frac{m_1^{-}(z; x,t )}{a(z)}, & m_2^{+}(z; x,t )
		\end{pmatrix} &  z \in \C _+   \medskip \\
		\begin{pmatrix}
			m_1^{+}(z; x,t ), & \frac{m_2^{-}(z; x,t )} {\widebar a(\widebar z )}
		\end{pmatrix} & z \in \C _- \quad .
	\end{cases}
\end{equation}

\begin{lemma}\label{lem:barm}
We have
\begin{subequations} \label{eq:barm11}
	\begin{align}
	\label{eq:barm}	\overline{m(\widebar z)} = \sigma_1  m(z)  \sigma_1, \\
 	\label{eq:mzinv}	m( z^{-1}) = z m(z) \sigma_1.
 	\end{align}
\end{subequations}
 \end{lemma}

 \begin{proof}	
 Both are immediate consequences of the symmetries contained in
 Lemma~\ref{lem:symm} and Lemma~\ref{lem:coeff}.
 \end{proof}

 Assume $q\in  \tanh(x) + L^1(\R  )$ and $q' (x)\in W^{1,1}(\R )$.
\begin{lemma} \label{lem:Minfty}
For $\pm \Im z >0$
\begin{align}
\label{eq:Minfty1}
	\lim_{z \to \infty}  z \(  m (z;x) - I \)
	&= \begin{pmatrix}
		-\im \int_x^\infty |q(y)|^2 -1\, dy & \overline{q(x)} \\
		q(x) & \im \int_x^\infty |q(y)|^2 -1\, dy
	\end{pmatrix} \quad ,  \\[10pt]
\label{eq:Mat01}	
	\lim_{z \to 0} \( m(z;x) - \frac{\sigma_1}{z} \)
	&= \begin{pmatrix}
		 \overline{q(x)} & - \im \int_x^\infty |q(y)|^2 -1\, dy \\
		 \im \int_x^\infty |q(y)|^2 -1\, dy & q(x)
	\end{pmatrix} \quad .
\end{align}
\end{lemma}	

\begin{proof}
The behavior at infinity follows immediately from Lemma~\ref{lem:zinfty} and \eqref{eq:as11}.
The behavior at the origin is then a consequence of Lemma~\ref{lem:barm}.
\end{proof}

It is an easy consequence of Lemma~\ref{lem:jf1}, Lemma~\ref{lem:zinfty},
Lemma~\ref{lem:coeff}, \eqref{eq:gamma}, and \eqref{eq:coeffevo} that $m(z;x,t)$ satisfies
the following Riemann Hilbert problem.
\begin{rhp} \label{rhp:m}
Find a $2 \times 2$ matrix valued function $m(z;x,t)$ such that
\begin{enumerate}[1.]
	\item $m$ is meromorphic for $z \in \C \backslash \R $.
	\item $\phantom{z}m(z;x,t) = I + \bigo{z^{-1}}$  as  $z \to \infty$. \\
		$z m(z;x,t) = \sigma_1 + \bigo{z}$  as  $z \to 0$
	\item The non-tangential limits
	$m_{\pm}(z; x,t) =\displaystyle\lim_{\C_\pm \ni \varsigma \to z}  m(\varsigma; x,t )$
	exist for any $z\in \R  \backslash \{  0\}$  and satisfy the \emph{jump relation} $m_+(z;x,t) = m_-(z; x,t) V(z)$ where
	\begin{equation} \label{eq:Mvx}
		V(z):= V_{tx}(z) = \begin{pmatrix}
    		1-|r(z)| ^2  &  - \overline{r (z)} e^{-\Phi(z; x,t)} \\
		r(z) e^{\Phi(z; x,t)} & 1
		\end{pmatrix},
	\end{equation}
	and
	\begin{equation*}
		\Phi(z) = \Phi(z; x,t) =  2 \im x \zeta(z) - 4 \im \zeta(z) \lambda(z) t =
		\im x (z-z^{-1})  - \im t (z^2 - z^{-2}).
	\end{equation*}
	\item $m(z;x,t)$ has simple poles at the points
	$\mathcal{Z} = \mathcal{Z}^+ \cup \overline{\mathcal{Z}^+}$,
	$\mathcal{Z}^+ = \{ z_k \}_{k=0}^{N-1} \subset \{ z=e^{i \theta}\, :\, 0 < \theta< \pi \}$,
	with residues satisfying
	\begin{equation} \label{eq:resm}
	 	\begin{aligned}
			\res_{z =z_k} m(z;x,t) &=
			\lim_{z \to z_k} m(z;x,t) \tril[0]{ c_k(x,t)}, \\
			\res_{z = \widebar z_k} m(z;x,t) &=
			\lim_{z \to \widebar z_k} m(z;x,t) \triu[0]{ \overline{c_k}(x,t) },
		\end{aligned}
	\end{equation}
	where
	\begin{equation}\label{explicit-V-k}
		c_k(x,t) = \frac{\gamma_k(0)  } {a'(z_k)} e^{\Phi(z_k; x,t)} 
		= c_k e^{\Phi(z_k; x,t)}, 
		\qquad 
		c_k = \frac{ \gamma_k(0)}{a'(z_k)} 
		= \frac{4 \im z_k}{ \int_\R   | \psi_2^+(z_k;x,0 )  |^2 dx} = \im z_k |c_k| .
	\end{equation}
	
\end{enumerate}
\end{rhp}	

\bigskip

The potential $q(x,t) $  is found by the  \textit{reconstruction
formula}, see Lemma \ref{lem:Minfty},
\begin{equation}\label{eq:invscat1}
q(x,t)    = \lim _{z\to \infty }z  \, m _{21}(z; x,t ).
\end{equation}
$N$--solitons are potentials corresponding to the case when $r(z)\equiv 0$.

\begin{lemma}\label{lem:det}
If a solution $m(z;x,t)$ of RHP~\ref{rhp:m} exists, it is unique if and only if it satisfies the symmetries of Lemma~\ref{lem:barm}, additionally for such a solution $\det m(z;x,t)  = 1 - z^{-2}$.
\end{lemma}

\begin{proof}

Suppose a solution $m(z)$ exists. It is trivial to verify using the symmetry $\overline{r(z)} = r(z^{-1})$ and the condition $\widebar{z_k} c_k \in i \R$ on the norming constants that
both $ \sigma_1 \overline{m(\widebar z)} \sigma_1$ and $ z m(z^{-1}) \sigma_1$ are solutions as well. So uniqueness immediately implies symmetry.

Suppose the solution $m$ possesses the symmetries. Taking the determinant of both sides of the jump relation gives $\det m _+ = \det m_-$ for $z\in \R$ since $\det V \equiv 1$. It follows from this, the normalization condition and the residue conditions that $\det m$ is rational in $z$ with poles at some subset of $\mathcal{Z} \cup \{0\}$. However, the form of the residue relation \eqref{eq:resm} implies that at each $p \in \mathcal{Z}$ a single column has a pole whose residue is proportional to the value of the other column. It follows then that $\det m$ is regular at each point $p \in \mathcal{Z}$. As $z \to 0$ the normalization condition gives $z^2 \det m(z) \to -1$. So $\det m = 1 + \alpha z^{-1} - z^{-2}$ for some constant $\alpha$. However, the symmetry \eqref{eq:mzinv} implies that $\det m(z) = - z^2 \det m(z^{-1})$ so $\alpha \equiv 0$.

Uniqueness then follows from applying Liouville's theorem to the ratio $m (\widetilde m)^{-1}$ of any two solutions $m$, $\widetilde m$ of RHP~\ref{rhp:m}, noting that at the origin we have
\[
\lim_{z \to 0} m(z) (\widetilde m(z) )^{-1} =
\lim_{z \to 0} (z^2-1)^{-1} z m(z) \sigma_2 z \widetilde m(z)^T \sigma_2
= -(\sigma_1 \sigma_2)^2 = I,
\]
where by elementary computation $ M ^{-1} =  (\det M) ^{-1}\sigma_2   M^T \sigma_2$ for any invertible
$2\times 2$ matrix.
\end{proof}

\section{ The long time analysis }
\label{sec:time}
\subsection{Step 1: Interpolation and conjugation}
\label{sec:step1}
In order to perform the long time analysis using the Deift-Zhou steepest descent method we need to perform two essential operations:
\begin{enumerate}[(i)]
	 \item interpolate the poles by trading them for jumps along small closed loops enclosing each pole;
	
	 \item use factorizations of the jump matrix along the real axis to deform the contours onto those on which the oscillatory jump on the real axis is traded for exponential decay.
\end{enumerate}

The second step is aided by two well known factorizations of the jump matrix $V$  in \eqref{eq:Mvx} :
\begin{gather}
	\label{factorizations}
	V(z) = \begin{pmatrix} 1 - |r(z) |^2 & - \overline{r(z)} e^{- \Phi } \\ r(z) e^{\Phi} & 1 \end{pmatrix}
		= b(z) ^{-\dagger} b(z) = B(z) T_0(z) B(z)^{-\dagger} \\
	\label{UL factor}
	b(z)^{-\dagger} = \triu{ -\overline{r(z)} e^{- \Phi}} \quad
	b(z) = \tril { r(z) e^{\Phi} }  \\
	\label{LDU factor}
 	B(z) = \tril { \frac { r(z)} { 1 - |r(z)|^2} e^\Phi },   \quad
	T_0(z) = (1- |r(z) |^2)^{\sigma_3}, \quad
	B(z)^{-\dagger} = \triu{ -\frac{ \overline{r(z)}}{1-|r(z)|^2} e^{-\Phi} }
 \end{gather}
where $A^\dagger$ denotes the Hermitian conjugate of $A$. Briefly, the leftmost term in the factorization can be deformed into $\C^-$, the rightmost factor into $\C^+$, while any central terms remain on the real axis.
These deformations are useful when they deform the factors into regions in which the corresponding off-diagonal exponentials $e^{\pm \Phi}$ are decaying. We will use a gently modified version of these factorizations to deform the jump contours, but first we introduce the pole interpolate which help account for these small changes.

Our method for dealing with the poles in the Riemann-Hilbert problem follows the ideas in \cite{DKKZ}, \cite{GT}, \cite{BKMM}.
To motivate the method we observe that on the unit circle the phase appearing in the residue conditions \eqref{eq:resm}-\eqref{explicit-V-k} satisfies
\begin{equation}\label{soliton phase}
	\Re(\Phi( e^{\im \theta}; x, t))=\Phi( e^{\im \theta}; x, t)  = - 4 t \sin \theta ( \xi - \cos \theta)  , \qquad \xi := \frac{x}{2t}	 .
\end{equation}
It follows that the poles $z_k \in  \mathcal{Z}$ are naturally split into three sets: those for which $\Re(z_k) < \xi$, corresponding to a connection coefficient $c_k(x,t) = c_k e^{\Phi(z_k,x,t)}$ which is exponentially decaying as $t \to \infty$; those for which $\Re(z_k) >\xi$, which have growing connection coefficients; and the singleton case $\Re(z_k) = \xi$ in which the connection coefficient is bounded in time. Given a finite set of discrete data $ \mathcal{Z} = \mathcal{Z}^+ \cup \overline{\mathcal{Z}^+}$, $\mathcal{Z}^+ = \{ z_k \in \C^+\,: \, k=0,1,\dots,{N-1}\}$ and $\overline{\mathcal{Z}^+}$ formed by the complex conjugates of $\mathcal{Z}^+$,
fix $\rho> 0$ small enough that
 \begin{equation}\label{conrho}
	\text{the sets $ | \Re (z - z_k) | \leq \rho$ are pairwise disjoint and
$\min\limits_{z_k \in \mathcal{Z^+} } \Im( z_k) > \rho$.} .
\end{equation}

We partition the set $\{0,\dots, {N-1} \}$ into the pair of sets
\begin{equation}\label{delta}
	\Delta = \{ j \, : \, \Re z_j > \xi  \} 
	\quad \text{and} \quad
	\nabla = \{ j \, : \, \Re z_j \leq \xi \}.
\end{equation}
These sets index all of the discrete spectra in the upper (and lower) half-plane. Those $j \in \Delta$ correspond to poles $z_j$ for which $| e^{\Phi(z_j)} | > 1$ and $j' \in \nabla$ to poles $z_{j'}$ for which $| e^{\Phi(z_{j'})} | \leq 1$. Additionally, we define
\begin{equation}\label{j0}
	j_0 = j_0(\xi) =
	\begin{cases}
		j & \text{if } | \Re (z_j) - \xi | < \rho \text{ for some } j \in \{0,\dots, {N-1} \} \\
		-1 & \text{otherwise}
	\end{cases}
\end{equation}
which is nonnegative only when some $z_{j_0}$ is near the line $\Re z = \xi$, so that $e^{\Phi(z_{j_0})} = \bigo{1}$.

The connection coefficients $c_j(x,t)$ for $j \in \Delta$ are exponentially large for $t \gg 1$ and for the purpose of steepest descent analysis we want our pole interpolate to ``exchange" the $e^\Phi$ in these residues for $e^{-\Phi}$ in the new jump matrix.

Define the function
\begin{equation}\label{T def}
	T(Z) = T(z, \xi) =
		\prod_{ k \in \Delta} \lp \frac{ z - z_k} {z z_k - 1} \rp
    		\exp \lp - \frac{1}{2\pi \im} \int_0^\infty  \log ( 1 - |r(s) |^2 )
		\lp \frac{1}{s-z} - \frac{1}{2s} \rp   ds \rp  \, .
\end{equation}
\begin{lemma}\label{lem:T}
The function $T(z,\xi)$ is meromorphic in $\C \backslash [0,\infty)$ with simple poles at the $\widebar z_k$ and simple zeros at the $z_k$ such that $\Re (z_k) > \xi$, and satisfies the jump condition
\begin{equation}\label{T jump}
	T_+(z,\xi) = T_-(x,\xi) (1 - |r(z)|^2), \qquad z \in ( 0, \infty).
\end{equation}
Additionally, the following propositions hold:
\begin{enumerate}[i.]
	\item $\overline{T( \widebar z, \xi)}  = T(z,\xi)^{-1} = T( z^{-1}, \xi)$;
	\item $\displaystyle T(\infty, \xi):=\lim_{z \to \infty} T(z, \xi)
		=  \lp \prod_{k \in \Delta} \widebar z_k \rp \,
		\exp \lp \frac{1}{4\pi \im} \int_0^\infty \frac{ \log(1 - |r(s)|^2)}{s} ds \rp $
		and $|T(\infty, \xi)|^2 =1$;

	\item $|T(z,\xi) | = 1$ for $z \leq 0$;

\item As $z\to \infty$ we have the asymptotic expansion
\begin{equation}\label{T asy}  T(z ,\xi) =
    T(\infty ,\xi)     \left ( I -z^{-1}  \left ( \sum_{k\in \Delta} 2\im \Im (z_k)  -
 \frac{1}{2\pi \im }   \int_0^\infty \log (1 - |r(s)|^2) ds\right  )   +o(z^{-1} ) \right  ) ;
\end{equation}

	\item  The ratio $\frac{a(z)}{T(z,\xi)}$ is holomorphic in $\C _+$ and there is a constant $ C (q_0)$ s.t.
 \begin{equation}\label{eq:lemT1}
	 \left | \frac{a(z)}{T(z,\xi)}\right |  \le C (q_0) \text{ for $z\in \C _+$ s.t. $\Re z >0$ }.
\end{equation}
 Additionally, the ratio extends as a continuous  function on $\R _+$  with $|\frac{a(z)}{T(z,\xi)}| =1$ for  $z \in (0, \infty)$.
\end{enumerate}
\end{lemma}

\begin{proof}
From \eqref{T def} it's obvious that $T$ has simple zeros at each $z_k$ and poles at each $\widebar z_k$, $k \in \Delta$. The jump relation \eqref{T jump} follows from the Plemelj formula. The first symmetry property follows immediately from the symmetry \eqref{eq:symm31} of $r(z)$ .
The second and  third  properties are simple computations,
as is the fourth  property, using Lemma~\ref{lem:log bound} with $p=1$.
Finally, consider the ratio $\frac{a(z)}{T(z,\xi)}$.
Using the representation \eqref{eq:afactor} for $a(z)$ we can write
\begin{equation} \label{eq:lemT2}
	\frac{a(z)}{T(z,\xi)} =
	\Big( \prod_{j \in \Delta} z_j \Big)
	e^{ -\frac{1}{2\pi \im} \int_0^\infty \frac { \log(1 - |r(s)|^2)}{2s} ds }
	\prod_{k \in \nabla} \lp \frac{z- z_k}{ z - \widebar z_k} \rp
	\exp \lp -\frac{1}{2\pi \im} \int_{-\infty}^0 \frac{ \log(1 - |r(s)|^2) }{ s- z} ds \rp .
\end{equation}
 In the r.h.s. all factors before the last one  have absolute value $\le 1$ for $z\in \C _+$ while the real
 part of the exponential can be bounded as follows,
 \begin{multline*}
	\frac{ -\Im (z)}{2\pi } \lp \int _{-\infty} ^{-\frac{1}{2}}  + \int ^{0} _{-\frac{1}{2}} \rp
	\frac{\log(1 - |r(s)|^2 ) }{(s- \Re( z) )^2+ \Im (z)^2} ds  \\
	\le  \frac{4 \Im(z)}{1+4 \Im(z)^2} \| \log (1-|r|^2)\| _{L^1(\R _- )}
	+ 2^{-1}\| \log (1-|r|^2)\| _{L^\infty (-2^{-1} ,0  )}
\end{multline*}
where we  bound the 1st term of the r.h.s. using  Lemma~\ref{lem:log bound}
and the 2nd  using $\| r\| _{L^\infty (-2^{-1} ,0  )}  <1  $.
Obviously the function in \eqref{eq:lemT2}   extends in a continuous way to $\R_+$ where it has absolute value 1.
\end{proof}

We are now ready to implement the interpolations and conjugations discussed
at the beginning of this section.
\begin{figure}[t]
	\centering
		\includegraphics[width = .4\textwidth]{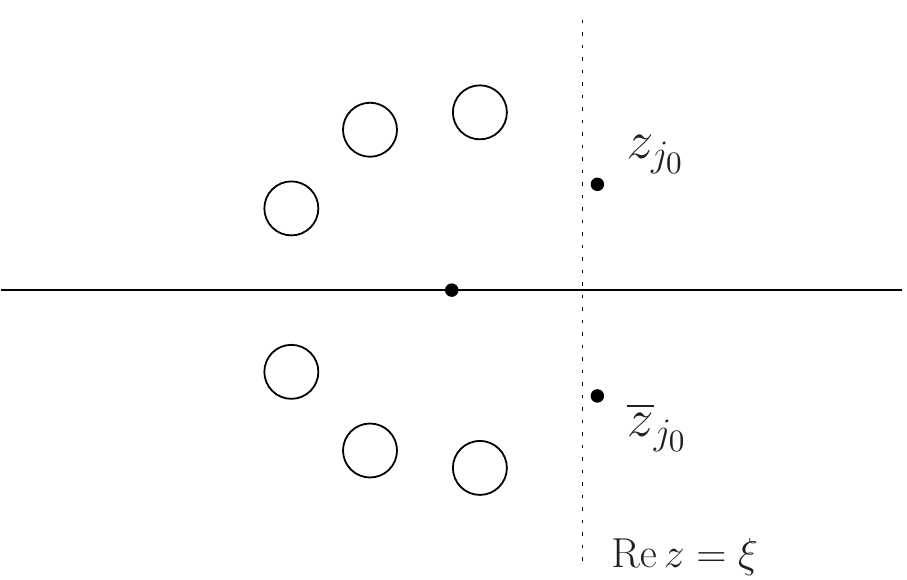}
		\hfill
		\includegraphics[width=.4\textwidth]{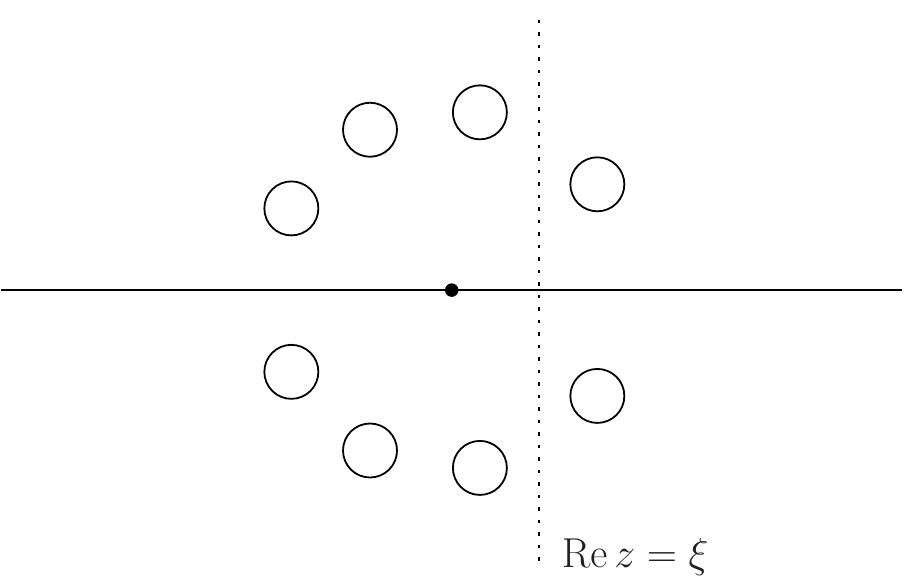}
	\caption{
	The contours defining the interpolating transformation $m \mapsto \mk{1}$ (c.f. \eqref{m1}).
	Around each of the poles $z_k \in \mathcal{Z}^+$ and its conjugate $\bar z_k \in \mathcal{Z}_-$
	we insert a small disk, oriented counterclockwise in $\C^+$ and clockwise in $\C^-$,
	of fixed radius $\rho$ sufficiently small such that the disks intersect neither
	each other nor the real axis. The set $\Delta$ (resp. $\nabla$) consist of those poles
	to the right (resp. left) of the line $\Re z = \xi$.
	If a pair $z_{j_0}, \ \overline{z}_{j_0}$ lies within $\rho$ of the line $\Re z = \xi$
	we leave that pair uninterpolated (left figure),
	otherwise all poles are interpolated (right figure).
	In either case, the singularity at the origin remains.
	\label{fig:m1contours}
	}
\end{figure}
Let $T(z)=T(z,\xi)$.
We remove the poles by the following transformation which trades the poles for jumps on small contours encircling each pole
\begin{equation}\label{m1}
	\mk{1}(z) = \begin{cases}
		T(\infty )^{-\sigma _3}    m(z) \tril{ -\frac{c_j e^{\Phi(z_j)}}{z-z_j}} T(z )^{ \sigma _3},
		& | z - z_j | < \rho , \ j \in \nabla \text{ and } | \Re(z_j ) - \xi | > \rho, \smallskip \\
	T(\infty )^{-\sigma _3} m(z) \triu{ -\frac{z-z_j}{ c_j e^{\Phi(z_j)}} }  T(z )^{ \sigma _3},
		& |z-z_j | < \rho, \ j \in \Delta \text{ and } | \Re(z_j ) - \xi |> \rho, \smallskip  \\
		T(\infty )^{-\sigma _3} m(z)
		\triu{ -\frac{ \widebar{c}_j e^{-\Phi_j(\widebar{z}_j)}}{z- \widebar{z}_j} }  T(z )^{ \sigma _3},
		&| z - \widebar{z}_j | < \rho, \ j \in \nabla \text{ and } | \Re(z_j ) - \xi |> \rho, \smallskip \\
	T(\infty )^{-\sigma _3} m(z)
		\tril{ -\frac{z- \widebar z_j}{  \widebar{c}_j e^{-\Phi(\widebar{z}_j)}} }  T(z )^{ \sigma _3},
		& | z - \widebar{z}_j | < \rho, \ j \in \Delta \text{ and } | \Re(z_j ) - \xi |> \rho, \\
		T(\infty )^{-\sigma _3} m(z) T(z )^{ \sigma _3} & \text{elsewhere}.
	\end{cases}
\end{equation}

Consider the following contour, depicted in Figure~\ref{fig:m1contours}:
\begin{equation}\label{eq:Sigma1}
	\Sigma^{(1)} = \R \cup \left (\bigcup_{  \substack{j \in \nabla \cup \Delta \\ j \neq j_0(\xi)}   }
	\{ z\in \C \,:\,    | z - z_j | = \rho \text{ or } | z - \widebar z_j | = \rho \} \right ).
\end{equation}
Here, $\R$ is oriented left-to-right and the disk boundaries are oriented counterclockwise in $\C^+$ and clockwise in $\C^-$.
\begin{lemma}
\label{lem:m1}
The Riemann-Hilbert problem for $\mk{1}(z)$ resulting from \eqref{m1} is
RHP~\ref{rhp:m1} formulated below. Furthermore, $\mk{1}(z)$
satisfies the symmetries of Lemma~\ref{lem:barm}.
\end{lemma}

\begin{rhp}\label{rhp:m1}
Find a $2 \times 2$ matrix-valued function $\mk{1}(z;x,t)$ such that
\begin{enumerate}[1.]
\item $\mk{1}(z;x,t)$ is meromorphic in $\C \backslash \Sigma^{(1)}$.

\item $\phantom{z} \mk{1}(z;x,t) = I + \bigo {z^{-1} }$ as $z \to \infty, \\
	z \mk{1}(z;x,t) = \sigma_1 + \bigo{z}$ as $z \to 0$.
	
\item The non-tangential boundary values $\mk{1}_\pm(z;x,t)$ exist for $z \in \Sigma^{(1)}$, and satisfy the jump relation $m_+(z;x,t) = m_-(z;x,t) \vk{1}(z)$ where
\begin{equation*}
	\vk{1}(z) = \begin{cases}
		\triu{ - \overline{r(z)} T(z)^{-2} e^{-\Phi} } \tril{  {r(z)} T(z)^{2} e^{\Phi} }
		& z \in (-\infty, 0) \smallskip \\
		\tril{  \frac{r(z)}{1-|r(z)|^2} T_-(z)^2 e^{\Phi} }
		\triu{ \frac{-\overline{r(z)}}{1-|r(z)|^2} T_+(z)^{-2}  e^{-\Phi} }
		& z \in (0, \infty) \smallskip \\
		\tril{ - \frac{ c_j}{z-z_j} T(z)^2 e^{\Phi(z_j)} }
		& | z - z_j | = \rho, \ j \in \nabla \smallskip \\
		\triu{ - \frac{ z-z_j}{c_j} T(z)^{-2} e^{-\Phi(z_j)}}
		& | z - z_j | = \rho, \ j \in \Delta \smallskip \\
		\triu{  \frac{   \overline{c_j}}{z-   \overline{z_j}} T(z)^{-2} e^{-\Phi(\widebar{z}_j)} }
		& | z - \widebar z_j | = \rho, \ j \in \nabla \smallskip \\
		\tril{  \frac{ z-  \overline{z_j}}{  \overline{c_j}} T(z)^{2} e^{\Phi(\widebar{z}_j)}}
		& | z - \widebar z_j | = \rho, \ j \in \Delta.
	\end{cases}
\end{equation*}		

\item If $(x,t)$ are such that there exist (at most one) $j_0 \in \{ 0,\dots, N-1 \}$
such that $| \Re z_{j_0} - \xi | \leq \rho$, $\xi = \frac{x}{2t}$,
then $\mk{1}(z;x,t)$ has simple poles at the points $z_{j_0}, \widebar z_{j_0} \in \mathcal{Z}$ satisfying one of the following alternatives:
\begin{enumerate}
	\begin{subequations}
	\item If $j_0 \in \nabla$,
		\begin{equation}\label{eq:dism1a}
			\begin{aligned}
			\res_{z_{j_0}} \mk{1}(z;x,t) &= \lim_{z \to z_{j_0}} \mk{1}(z;x,t)
			\tril[0]{ c_{j_0} T(z_{j_0})^2 e^{\Phi(z_{j_0})} }, \\
			\res_{\widebar z_{j_0}} \mk{1}(z;x,t) &= \lim_{z \to \widebar z_{j_0}} \mk{1}(z;x,t)
			\triu[0]{   \overline{c}_{j_0} \overline{T}(z_{j_0})^{ 2} e^{ \Phi(  z_{j_0})} },
			\end{aligned}	
		\end{equation}
	\item If $j_0 \in \Delta$,
		\begin{equation}
			\begin{aligned}\label{eq:dism1b}
			\res_{z_{j_0}} \mk{1}(z;x,t) &= \lim_{z \to z_{j_0}} \mk{1}(z;x,t)
			\triu[0]{ c_{j_0}^{-1}    T'(z_{j_0})   ^{-2} e^{ -\Phi(z_{j_0})} }, \\
			\res_{\widebar z_{j_0}} \mk{1}(z;x,t) &= \lim_{z \to \widebar z_{j_0}} \mk{1}(z;x,t)
			\tril[0]{  \overline{ c}_{j_0}^{-1}    \overline{T'}(z_{j_0})   ^{- 2} e^{-\Phi(  z_{j_0})} }.
			\end{aligned}	
		\end{equation}
	\end{subequations}
\end{enumerate}	
Otherwise, $\mk{1}$(z;x,t) is analytic in $\C \backslash \Sigma^{(1)}$.
\end{enumerate}
\end{rhp}

\bigskip

\begin{remark}
The function $T(z,\xi)$ and the transformation $m \mapsto \mk{1}$ defined by \eqref{m1} can be thought of in two parts. In the first step the Blaschke product in $T$ swaps the columns in which the poles $z_j,\ j \in \Delta$, appear and gives new connection coefficient proportional to $c_j(x,t)^{-1}$ as desired. The triangular factors in \eqref{m1} then interpolate the poles trading them for jumps on the disk boundaries $|z - z_j| = \rho$. In the second step, the Cauchy integral term in $T$ is responsible for removing the diagonal factor \eqref{LDU factor} from the jump matrix factorization $V = B T_0 B^{-\dagger}$ (cf. \eqref{factorizations}) on the half-line $(0,\infty)$.
Finally, we point out that factors $(z z_k -1)$ in the Blaschke product---instead of simply $(z - \widebar z_k)$---and the $\frac{1}{2s}$ term in the Cauchy integral are introduced so that $T$ satisfies property $i.$ in Lemma~\ref{lem:T} which is needed to preserve the symmetries in Lemma~\ref{lem:barm}.
\end{remark}

\textit{Proof of Lemma \ref{lem:m1}}. The proof consists of  a  lengthy but elementary series of computations which we sketch only partially.
First of all we start  with the symmetries of Lemma \ref{lem:barm}.
For instance, the region outside the disks in \eqref{m1}  is invariant
by the transformations $z\to \overline{z}$ and $z\to z ^{-1}$. We have
\begin{gather*}     \begin{aligned} &
   \overline{\mk{1}}(\overline{z} )=  \overline{T}(\infty)^{-\sigma _3}     \overline{m}(\overline{z} )   \overline{T}(\overline{z})^{ \sigma _3} ={T}(\infty)^{ \sigma _3} \sigma _1 {m}( {z} ) \sigma _1 {T}( {z})^{- \sigma _3}= \sigma _1 \mk{1}( {z} ) \sigma _1
\end{aligned} 
\shortintertext{and}
\begin{aligned} &
    {\mk{1}}( z ^{-1} )=   {T}(\infty)^{-\sigma _3}     {m}(z ^{-1})   {T}(z ^{-1})^{ \sigma _3} =z T(\infty)^{-\sigma _3}   {m}( {z} ) \sigma _1 {T}( {z})^{- \sigma _3}= z \mk{1}( {z} ) \sigma _1,
\end{aligned} 
\end{gather*}
where we have used the symmetries of $m(z)$ and of $T(z)$.
Using    also  the symmetries of $\Phi (z)$  these
equalities  can be similarly extended on the whole domain of $\mk{1}( {z} )$.

While Claim 1  and the 1st equality in Claim 2  in RHP~\ref{rhp:m1}  are  obvious
consequences of the corresponding ones  in RHP~\ref{rhp:m},
the 2nd  equality in Claim 2 follows from
\begin{equation*}
 \begin{aligned} &  z \mk{1}(z ) =   T(\infty)^{-\sigma _3}   z m(z )   T(z)^{ \sigma _3}  =  T(\infty)^{-\sigma _3}(   \sigma_1 + \bigo{z})  T(z ^{-1})^{ -\sigma _3} \\& =  T(\infty)^{-\sigma _3}(   \sigma_1 + \bigo{z}) (  T(\infty)+ \bigo{z} ) ^{-\sigma _3}= \sigma_1 + \bigo{z},
    \end{aligned}
\end{equation*}
where we used the symmetry and the expansion $T(z)= T(\infty)+ \bigo{z ^{-1}}$ as $z\to \infty$ in Claims $i$ and $iv$ of Lemma~\ref{lem:T} respectively.

We skip the proof of Claim 3 which is an immediate consequence of
\eqref{m1} and of Claim 3  in RHP~\ref{rhp:m}.
The proof of the 1st limit in  \eqref{eq:dism1a} follows immediately
from
\begin{multline*}
	\res_{z_{j_0}} \mk{1}(z ) =  	
	\res_{z_{j_0}}  T(\infty)^{-\sigma _3}  m(z )T(z)^{ \sigma _3} = \\
 	\lim_{z \to z_{j_0}}  T(\infty)^{-\sigma _3}  m(z ) T(z)^{ \sigma _3}T(z)^{ -\sigma _3}
 	\tril[0]{ c_{j_0}   e^{\Phi(z_{j_0})}} T(z)^{ \sigma _3}
	= \lim_{\mathclap{z \to z_{j_0} } }  \mk{1}(z ) \tril[0]{ c_{j_0} T(z_{j_0})^2  e^{\Phi(z_{j_0})}} .
\end{multline*}
 We now turn to the 1st limit in  \eqref{eq:dism1b}. We have
\begin{equation*}
				\begin{aligned}&
		  \res_{z_{j_0}} m ^{(1)}(z ) =\lim_{z \to z_{j_0}}   (z-z_{j_0}) T(\infty )^{-\sigma _3} \left (
			\frac{m_1^{-}(z )T(z)}{a(z)} ,   \frac{m_2^{+}(z )}{T(z)}
		\right )   = T(\infty )^{-\sigma _3}\left (
			0 ,   \frac{m_2^{+}( z_{j_0} )}{T'(z_{j_0})}
	\right ) \\& =  \lim_{z \to z_{j_0}}    T(\infty )^{-\sigma _3}\left (
			\frac{m_1^{-}(z )T(z)}{a(z)} ,   \frac{m_2^{+}(z )}{T(z)}
		\right )   \triu[0]{ c_{j_0}^{-1}   T'(z_{j_0})   ^{-2} e^{-\Phi(z_{j_0})} }
			\end{aligned}	
		\end{equation*}
 which yields the 1st limit in  \eqref{eq:dism1b}. In the last equality we've used the fact that
 \begin{equation*}
  m_2^{+}(z_{j_0}) T'(z_{j_0})   ^{-1}  =  \gamma ^{-1}_{j_0}(0)  T'(z_{j_0})   ^{-1} e^{-\Phi(z_{j_0})}   m_1^{-}(z_{j_0} )=  {a'(z_{j_0}) }^{-1} c_{j_0}^{-1}   T'(z_{j_0})   ^{-1} e^{-\Phi(z_{j_0})} {m_1^{-}(z ) }
 \end{equation*}
 which follows from  \eqref{eq:gamma}, \eqref{eq:coeffevo} and \eqref{explicit-V-k}. The limits in the  2nd lines of \eqref{eq:dism1a}--\eqref{eq:dism1b} follow from
 the symmetry \eqref{eq:barm}, which is satisfied by $m ^{(1)}(z )$.

\qed

\subsection{Step 2: opening $\dbar$ lenses}
\label{subsec:lense}

We now want to remove the jump from the real axis in such a way that the new problem takes advantage of the decay/growth of $\exp \lp \Phi(z) \rp$ for $z \not\in \R$. Additionally we want to ``open the lens"  in such a way that the lenses are bounded away from the disks introduced previously to remove the poles from the problem.

To that end, fix an angle $\theta_0 >0$ sufficiently small such that the set
$\{ z \in \C\,:\, \left| \frac{\Re z}{z} \right| > \cos \theta_0 \}$ does not intersect any of the disks $|z- z_k| \leq  \rho$. For any $ \xi \in (-1,1) $, let
\begin{equation*}
	 \phi(\xi) = \min \left\{ \theta_0, \  \arccos \lp \frac{2 |\xi|}{1+|\xi|} \rp \right\},
\end{equation*}
and define $\Omega = \bigcup_{k=1}^4 \Omega_k$, where
\begin{equation*}
	\begin{aligned}
		\Omega_1 &= \{z \, : \, \arg z \in (0, \phi(\xi) ) \}, \qquad \qquad  \quad \
		\Omega_2 = \{z \, : \, \arg z \in (\pi - \phi(\xi), \pi) \}, \\
		\Omega_3 &= \{z \, : \, \arg z \in (-\pi , -\pi+ \phi(\xi) ) \}, \qquad
		\Omega_4 = \{z \, : \, \arg z \in ( - \phi(\xi), 0) \}.
	\end{aligned}
\end{equation*}
Finally, denote by
\begin{equation*}
	\begin{aligned}
		\Sigma_1 &=e^{\im \phi(\xi)} \R_+ , \ \qquad \qquad
		\Sigma_2 = e^{\im (\pi - \phi(\xi) )} \R_+ \\
		\Sigma_3 &= e^{-\im (\pi - \phi(\xi) )} \R_+, \qquad
		\Sigma_4 = e^{-\im \phi(\xi) } \R_+,
	\end{aligned}
\end{equation*}
the left-to-right oriented boundaries of $\Omega$, see Figure~\ref{fig:m2 regions}.

\begin{lemma} \label{lem:phase1a}
	Set $ \xi :=\frac{x}{2 t}$  and let $|\xi | < 1$.
	Then for $z= |z| e^{\im \theta}$ and $F(s) = s+ s^{-1}$,
	the phase $\Phi$ defined in \eqref{eq:Mvx} satisfies
  	\begin{equation}\label{eq:phase1}
		\begin{aligned}
		\Re [\Phi (z;x,t)] & \ge \phantom{-}\frac{t}{4}( 1 -| \xi |)  F(|z|)^2 \left|  \sin 2 \theta \right|
		\quad \text{for }   z  \in \Omega _1 \cup \Omega _3   , \\
		\Re [\Phi (z;x,t)] & \le -\frac{t}{4}( 1 -| \xi |)  F(|z|)^2 \left| \sin 2 \theta \right|
		\quad \text{for }   z  \in \Omega _2 \cup \Omega _4   . \\
		\end{aligned}
	\end{equation}
\end{lemma}

\begin{proof}
We will  consider only the case $z\in \Omega _1$. By elementary computation we have
\begin{equation} \label{eq:phase11}
	\Re [\Phi (z;x,t)]
	=  t \ \psi(z) \sin 2\theta
	\quad \text{with}
	\quad \psi(z) =  F(|z|)^2 - \xi F(|z|) \sec \theta - 2.  \\
\end{equation}
Then, observing that $F(|z|) \geq 2$, we have for $z \in \Omega_1$,
\begin{equation*}
	\begin{aligned}
		\psi(z) & \geq  F(|z|)^2 - \frac{1+| \xi |}{2} F(|z|) - 2   \geq \frac{1-| \xi |}{4} F(|z|)^2	
	\end{aligned}
\end{equation*}
so that
\begin{equation*}
	\Re \Phi(z;x,t) \geq \frac{t}{4}(1-| \xi |) F(|z|)^2  \sin 2\theta.	
\end{equation*}
\qed
\end{proof}

The estimates suggest that we should open lenses using (modified versions of) factorization \eqref{UL factor} for $z<0$ and \eqref{LDU factor} for $z>0$. To do so, we need to define extensions of the off-diagonal entries of $b(z)$ and $B(z)$ off the real axis, which is the content of the following lemma.

\begin{lemma}\label{lem:extR}
Let $q_0 \in  \tanh(x) + L^{1,3}(\R)$ and $q_0' \in W^{1,1}(\R)$.
Then it is possible to define functions $R_j: \overline{\Omega}_j \to \C$, $j=1,2,3,4,$ continuous on $\overline{\Omega_j}$, with continuous first partials on $\Omega_j$, and boundary values
\begin{align*}
	&\begin{cases}
		R_1(z) = \frac{   \overline{r (z)} T_+(z)^{-2} }{1 - |r(z) |^2}  & z \in (0,\infty) \\
		R_1(z) = 0 & z \in \Sigma_1
	 \end{cases} \\
	&\begin{cases}
		R_2(z) = r(z) T(z)^2 & z \in (-\infty, 0) \\
		R_2(z) = 0 & z \in \Sigma_2
	\end{cases} \\
	&\begin{cases}
		R_3(z) = \overline{r(z)} T(z)^{-2} & z \in (-\infty, 0) \\
		R_3(z) = 0 & z \in \Sigma_3
	\end{cases} \\
	&\begin{cases}
		R_4(z) = \frac{r(z) T_-(z)^2}{1 - |r(z)|^2}  & z \in (0, \infty) \\
		R_4(z) = 0 & z \in \Sigma_4
	\end{cases}
\end{align*}	
such that for $j=1$ and $4$; a fixed constant  $c_1= c_1(q_0)$; and a fixed cutoff function $\varphi \in C_0^\infty (\R , [0,1])$ with small support near $1$; we have
\begin{align}
	\left| \dbar R_j(z) \right|
	& \leq c_1    |z|^{-1/2} + c_1  | r'(|z| ) | + c_1  \varphi (|z| )     \text{ for all $z \in  \Omega_j$ and} \label{eq:extR2}\\ \left| \dbar R_j(z) \right|
	&    \le   c_1  |z-1|  \text{ for all $z \in  \Omega_j$ in a small fixed neighborhood of 1}   \label{eq:extR21} \end{align}		
while for $j=2,3$ we have \eqref{eq:extR2} with $|z|$ replaced by $-|z|$  in the argument
of   $r'$  and without the term  $c  _1   \varphi (|z| )  $.

Setting $R: \Omega \to \C$ by $R(z) \Big|_{z \in \Omega_j} = R_j(z)$,
the extension can preserve the symmetry $R(z) = \widebar{R(\widebar{z^{-1}})} $.
\end{lemma}

\begin{proof}

We will give the details of the proof for $R_1$. The estimates for the $\dbar$-derivative for $ j=4$ are nearly identical to the case $ j=1$;
the definitions of $R_2$ and $R_3$ and their $\dbar$ estimates are similar and are a simpler version of
\cite[Proposition 2.1]{DM}.

\noindent As observed in \eqref{eq:coeff sing}-\eqref{eq:r sing},   $a(z)$ and $b(z)$ are singular at $z = \pm 1$, and $r(z) \to \mp 1$ as $z\to \pm 1$.
This suggests that $R_1(z)$ is singular at $z = 1$. However, the singular behavior is exactly balanced by the factor $T(z)^{-2}$.   From \eqref{eq:r}-\eqref{eq:coeff3} we have
\begin{equation}\label{eq:extR4}
	\frac{ \overline{ r(z) }}{ 1 - |r(z)|^2 } T_+(z)^{-2}
	= \frac{ \overline{ b(z) }}{a(z)} \lp \frac{ a(z) } { T_+(z) } \rp^2
	= \frac{ \overline{ J_b(z) } } { J_a(z) }
	\lp \frac{ a(z) } { T _+(z) } \rp^2,
\end{equation}
where we have temporarily introduced the notation
\begin{equation}\label{Jab}
	J_b(z) = \det \lb \psi_1^+( z;x,t), \psi_1^-(z;x,t) \rb,
	\qquad
	J_a(z) = \det \lb \psi_1^-(z;x,t), \psi_2^+(z;x,t ) \rb .
\end{equation}
Recall that though the columns of the right/left normalized Jost functions, $\psi_j^\pm(z;x,t),\, j =1,2,$ depend on $x$, the determinants are independent of $x$ as $\Tr \mathcal{L} =0$.
Using Lemmas~\ref{lem:jf1} and \ref{lem:T}, the denominator of each factor in the r.h.s. of  \eqref{eq:extR4}  is non-zero and analytic in $\Omega_1$, with a well defined nonzero limit on $\partial \Omega_1$. Notice also that in $\Omega _1$ away from
the point $z=1$ the  factors in the  l.h.s. of \eqref{eq:extR4} are well behaved.

We introduce cutoff functions $\chi_0, \chi_1 \in C^\infty_0(\R, [0,1])$ with small support near $0$ and $1$ respectively, such that for any sufficiently small real $s$, $\chi_0(s) = 1 = \chi_1(1+s)$. Additionally, we impose the condition that $\chi_1(s) = \chi_1(s^{-1})$ to preserve symmetry.
We then rewrite the function $R_1(z)$  in $\R _+$  as  $R_1(z)= R _{11}(z)+R_{12}(z)  $ with
\begin{equation}\label{eq:extR5} \begin{aligned}
	& R _{11}(z):= (1-\chi_1(z))
	\frac{ \widebar{ r(z) }}{ 1 - |r(z)|^2 } T_+(z)^{-2}\ , \quad
	R _{12}(z):= \chi_1(z)
	\frac{ \widebar{ J_b(z) } } { J_a(z) }
	\lp \frac{ a(z) } { T _+(z) } \rp^2. \end{aligned}
\end{equation}
The purpose of \eqref{eq:extR5} is to neutralize the effect of the singularity
at $1$ due to $|r(1)|=1$.
Fix a small $\delta _0> 0$.
Then extend the functions $R _{11}$ and $R _{12}$  in  $\Omega _1$  by
\begin{align}
        R _{11}(z) &= \lp 1-\chi_1(|z|) \rp
	\frac{ \widebar{ r(|z|) }}{ 1 - |r(|z|)|^2 } T (z)^{-2}
	\cos \lp k \arg z \rp, \label{R11 def} \\
	R _{12}(z)&= f(|z|) g(z) \cos (k  \arg z) +
	\frac{{\im}|z|}{ k}     \chi_0  \lp \frac{ \arg z }{\delta _0} \rp f'(|z|) g(z) \sin(  k \arg(z) )
	\label{R12 def}
\end{align}

where $f'(s) $ is the derivative of $f(s)$ and
\begin{equation*}
    k :=\frac{ \pi}{2 \theta_0} \, , \quad
    g(z):=  \lp \frac{ a(z) } { T (z) } \rp^2  \, , \quad
    f(s):=   \chi_1 (s) \frac{\overline{ J_b ( s) } }{J_a(s)}.
\end{equation*}
Both extensions are similar to Prop. 2.1 \cite{DM}, but \eqref{R12 def} is somewhat more
elaborate.
Observe that the definition of $R_1$ above preserves the symmetry $R_1(s) = \widebar{R_1(\widebar{s^{-1}})}$.  Aided by the symmetry conditions \eqref{eq:symm31}, \eqref{eq:symm23}, Claim $i$ of Lemma~\ref{lem:T}, and $\chi_1(s) = \chi_1(s^{-1})$ one shows that $R_{11}$, $f$, and $g$ satisfy the desired symmetry; the rest is a trivial exercise.

We now bound the $\dbar $ derivatives of \eqref{R11 def}--\eqref{R12 def}.
We have
\begin{align}\label{dbar1}
	\dbar R _{11}(z)  &=
 	-\frac{ \dbar  \chi _1(|z|)}{T (z)^{ 2}} \,
	\frac{ \overline{ r(|z|)} \cos \lp k \arg z \rp}{ 1 - |r(|z|)|^2 }   +  	
	\frac{   1- \chi _1(|z|)}{T (z)^{ 2}} \,
	\dbar  \lp \frac{ \overline{ r(|z|) }  \cos \lp k \arg z \rp }{ 1 - |r(|z|)|^2 } \rp.	
\end{align}
Observe that $ 1-|r(|z|)|^2 >c>0$ in $\supp (1-\chi_1(|z|))$ and  $|T (z)^{-2} |\le C e ^{   -\log c   } $ in $\Omega _1\cap \supp (1-\chi_1(|z|))$  for some  fixed  constants $c$ and $C$. Then for some new fixed constant $C$ we have
\begin{equation*}
   \left |  \dbar   \frac{ \overline{ r(|z|) }  \cos \lp k \arg z \rp }{ 1 - |r(|z|)|^2 } \right | \le C    |r'(|z|)|
   +C  \left |\sin \lp k \arg z \rp \right | \frac{|r (|z|)|}{|z|} .
\end{equation*}
As $r(0)=0$ it follows that  $|r (|z|)|\le \sqrt{|z|} \| r '\| _{L^2(\R )} $.
Notice also that the first term in the r.h.s. of \eqref{dbar1} can be bounded by $c_1 \varphi  (|z|)$  for an appropriate $\varphi \in C_0^\infty (\R , [0,1])$ with  a small support near $1$ and with $\varphi =1$
on $\text{supp}  \chi_1$.
It follows that the r.h.s. of \eqref{dbar1} satisfies  \eqref{eq:extR2} .

We turn now to $\dbar  R _{12}$.  For $z = u + \im v = \rho e^{\im \phi}$ we have $
\dbar = \frac{1}{2} \lp \partial_u + \im  \partial_v \rp
= \frac{ e^{\im   \phi} }{2} \lp \partial_\rho + \frac{\im     }{\rho} \partial_\phi \rp$ . Then
\begin{multline*}
	\dbar  R _{12}(z)  =  \frac{e^{\im \phi} g(z)}{2} \Bigg[
 		f'(\rho ) \cos ( k \phi )  \lp 1 - \chi_0 \lp \frac{ \phi }{\delta_0}  \rp \rp
		- \frac{\im k  f(\rho ) }{ \rho} \sin ( k \phi ) \\
		+ \frac{ {\im} }{k} (\rho f'(\rho ))'   \sin ( k \phi ) \chi_0 \lp \frac{ \phi}{\delta_0} \rp
		+ \frac{\im }{ k \delta_0} f'(\rho )  \sin ( k \phi )   \chi_0' \lp \frac{ \phi }{ \delta_0 }  \rp
		\Bigg]
\end{multline*}
where the 1st term in the bracket is obtained by applying $\partial _\rho$ to $f(\rho )$ in  the 1st term in \eqref{R12 def}   and   $\im \rho ^{-1}\partial _\phi$ to $\sin ( k \phi )$ in the 2nd term of \eqref{R12 def}.

Then we claim $|\dbar  R _{12}(z) |\le c  _1   \varphi (|z| ) $  for a  $\varphi \in C_0^\infty (\R , [0,1])$ supported near 1, thus yielding  \eqref{eq:extR2}.
The prefactor including $g(z)$ is bounded by \eqref{eq:lemT1}.
The first, third, and fourth terms in the brackets are bounded by observing that, for $q$ satisfying the hypotheses of the Lemma, we have $\widebar{J_b(s)}/J_a(s) \in W^{\infty,2}(\R)$---this follows from a small modification of Lemma~\ref{lem:a3bis} where the extra moment is needed for second derivatives in the term $(\rho f'(\rho ))'$ appearing in the expression for $\dbar  R _{12}$ above (\cf\ \eqref{eq:a13}). The second term is bounded because $\supp \chi_1$ is bounded away from zero.
Finally, for $z \sim 1$, 
$|\dbar R_{12}| \leq C \lb \, |\sin ( k \phi ) | + (1 - \chi_0(\phi/\delta_0)\, \rb = \bigo{ \phi } $, from which
\eqref{eq:extR21}  follows immediately.

\end{proof}

We use the extensions of Lemma~\ref{lem:extR} to define modified versions of the factorizations \eqref{factorizations} which extend into the lenses $\Omega_j$.
We have on the real axis
\begin{equation*}
	V^{(1)}(z) = \widehat{b}^{-\dagger}(z) \widehat{b}(z) = \widehat{B}(z) \widehat{B}^{-\dagger}(z)
\end{equation*}
where
\begin{equation*}
	\begin{aligned}
		\widehat{b}(z) &= \tril{ R_2(z) e^{ \Phi} }, \qquad
		\widehat{b}^\dagger(z) = \triu{R_3(z) e^{-\Phi} }, \\
		\widehat{B}(z) &= \tril{ R_4(z) e^{ \Phi} }, \qquad
		\widehat{B}^\dagger(z) = \triu{R_1(z) e^{-\Phi} }.
	\end{aligned}
\end{equation*}
We use these to define a new unknown
\begin{equation}\label{m2 def}
	\mk{2}(z) = \begin{cases}
		\mk{1}(z) \widehat{B}^\dagger(z) & z \in  \Omega_1 \\
		\mk{1}(z) \widehat{b}(z)^{-1} & z \in \Omega_2 \\
		\mk{1}(z) \widehat{b}(z)^{-\dagger} (z) & z \in \Omega_3 \\
		\mk{1}(z) \widehat{B}(z) & z \in \Omega_4 \\  	\mk{1}(z)   & z \in \C \backslash \overline{\Omega} .
	\end{cases}
\end{equation}

\begin{figure}
	\centering
		\includegraphics[width=.4\textwidth]{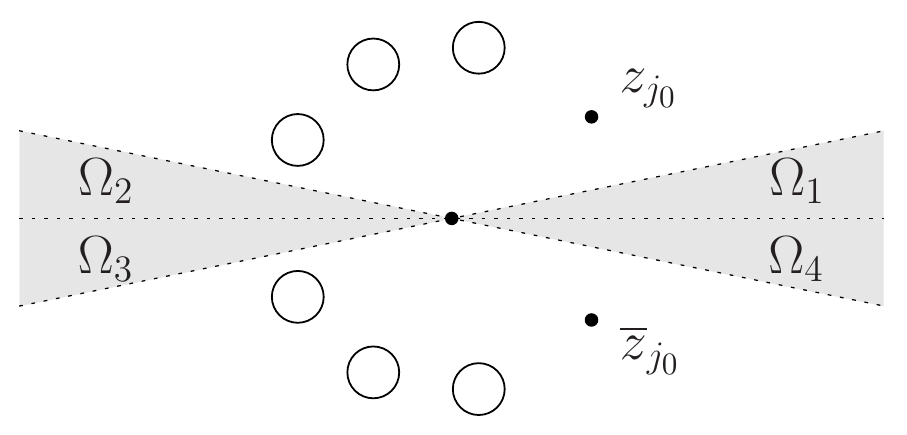}
	\caption{
	The unknown $\mk{2}(z)$ defined by \eqref{m2 def} has nonzero $\dbar$ derivatives
	in the regions $\Omega_j$, and jump discontinuities on the disk boundaries $|z-z_j| = \rho$.
	The dashed boundaries of $\Omega_j$ indicate that $\mk{2}$ is continuous at these boundaries.
	\label{fig:m2 regions} 	
	}
\end{figure}

Let
\begin{equation}\label{Sigma2}
	\Sigma^{(2)} = \bigcup_{\mathclap{ \substack{j \in \nabla \cup \Delta \\ j \neq j_0(\xi)}}}
		\{ z \in \C\, : \, |z-z_j| = \rho \text{ or } |z - \widebar z_j| = \rho  \}
\end{equation}
be the union of the circular boundaries of each interpolation disk oriented as in $\Sigma^{(1)}$.
It is an immediate consequence of \eqref{m2 def} and Lemmas \ref{lem:m1} and \ref{lem:extR} that
$\mk{2}$ satisfies the following $\dbar$--Riemann-Hilbert problem.

\begin{rhp}[$\dbar$--\!]\label{rhp: m2}
Find a $2\times2$ matrix-valued function $\mk{2}(z; x,t)$ such that
\begin{enumerate}[1.]
	\item $\mk{2}$ is continuous in $\C \backslash \Sigma^{(2)}$
	 and takes continuous boundary values $\mk{2}_+(z;x,t)$
	 (respectively $\mk{2}_-(z;x,t)$) on $\Sigma^{(2)}$ from the left (respectively right).
	
	\item $\phantom{z} \mk{2}(z;x,t) = I + \bigo {z^{-1} }$ as $z \to \infty, \\
	z \mk{2}(z;x,t) = \sigma_1 + \bigo{z}$ as $z \to 0$.
	
	 \item The boundary values are connected by the jump relation $\mk{2}_+(z;x,t) = \mk{2}_-(z;x,t) \vk{2}(z)$ where
	 \begin{equation}\label{V2}
	 	\vk{2}(z) = \begin{cases}
			\tril{ - \frac{ c_j}{z-z_j} T(z)^2 e^{\Phi (z_j)} }
			& | z - z_j | = \rho, \ j \in \nabla \smallskip \\
			\triu{ - \frac{ z-z_j}{c_j} T(z)^{-2} e^{-\Phi (z_j)}}
			& | z - z_j | = \rho, \ j \in \Delta \smallskip \\
			\triu{ \frac{ \widebar c_j}{z- \widebar z_j} T( z)^{-2} e^{\Phi (z_j)} }
			& | z - \widebar z_j | = \rho, \ j \in \nabla \smallskip \\
			\tril{ \frac{ z-\widebar z_j}{\widebar c_j} T( z)^{2} e^{-\Phi (z_j)}}
			& | z - \widebar z_j | = \rho, \ j \in \Delta   .
		\end{cases}
	\end{equation}
	
	\item For $z \in \C$ we have
	\begin{gather}\label{W2}
		\dbar \mk{2}(z;x,t) = \mk{2}(z;x,t) W(z)
	\shortintertext{where}
		W(z) = \begin{cases}
				\triu[0]{ \dbar R_1(z) e^{-\Phi} } & z \in \Omega_1 \smallskip \\
				\tril[0]{ -\dbar R_2(z) e^{\Phi} } & z \in \Omega_2 \smallskip \\
				\triu[0]{ -\dbar R_3(z) e^{-\Phi} } & z \in \Omega_3 \smallskip \\
				\tril[0]{\dbar R_4(z) e^{\Phi} } & z \in \Omega_4 \smallskip \\
				\qquad \boldsymbol{0} & z { \ \rm{ elsewhere} .}
			\end{cases}	 \nonumber
	\end{gather} 	
		
	\item   $\mk{2}(z ;x,t)$ is analytic in  the region $\C \backslash ( \widebar{\Omega}  \cup \Sigma^{(2)})$ if $j_0=-1$.
If $(x,t)$ are such that there exists $j_0 \in \{ 0,\dots,N-1 \}$
	such that $| \Re z_{j_0} - \xi | \leq \rho$, $\xi = \frac{x}{2t}$,
	then $\mk{2}(z;x,t)$ is meromorphic in $\C \backslash  ( \widebar{\Omega}  \cup \Sigma^{(2)})  $ with exactly two poles, which are simple,  at the points $z_{j_0}, \widebar z_{j_0} \in \mathcal{Z}$ satisfying one of the following cases.
\begin{itemize}
\item[(a)]
	  If $j_0 \in \nabla$, letting $C_{j_0}= c_{j_0} T(z_{j_0})^2$, we have
		\begin{equation}\label{m2 nabla con}
			\begin{aligned}
			\res_{z_{j_0}} \mk{2}(z;x,t) &= \lim_{z \to z_{j_0}} \mk{2}(z;x,t)
			\tril[0]{ C_{j_0}   e^{\Phi(z_{j_0})} }  , \\
			\res_{\widebar z_{j_0}} \mk{2}(z;x,t) &= \lim_{z \to   \overline{z}_{j_0}} \mk{2}(z;x,t)
			\triu[0]{   \overline{C}_{j_0}  e^{ \Phi(  z_{j_0})} }.
			\end{aligned}	
		\end{equation}
	\item[(b)] If $j_0 \in \Delta$, letting $C_{j_0} :=c_{j_0}^{-1} T'(z_{j_0})^{-2}$, we have
		\begin{equation}\label{m2 delta con}
			\begin{aligned}
			\res_{z_{j_0}} \mk{2}(z;x,t) &= \lim_{z \to z_{j_0}} \mk{2}(z;x,t)
			\triu[0]{C_{j_0}  e^{-\Phi(z_{j_0})} }      \\
			\res_{\widebar z_{j_0}} \mk{2}(z;x,t) &= \lim_{z \to \widebar z_{j_0}} \mk{2}(z;x,t)
			\tril[0]{   \overline{C}_{j_0}       e^{-\Phi(  z_{j_0})} }.
			\end{aligned}	
		\end{equation}
\end{itemize}

\end{enumerate}
\end{rhp}

\subsection{Step 3: removing the poles; the asymptotic N-soliton solution}
\label{subsec:rem}
Our next step is to remove the Riemann-Hilbert component of the solution, so that all that remains is a new unknown with nonzero $\dbar$-derivatives in $\Omega$, and is otherwise bounded and approaching identity for $|z| \to \infty$. Once this is complete, the remaining problem is analyzed using the ``small-norm" theory for the solid Cauchy operator. This is done in the following section, Section~\ref{sec: step4}.

\begin{lemma}\label{lem: m sol}
Let $\mk{sol}$ denote the solution of the Riemann-Hilbert problem which results from simply ignoring the $\dbar$ component of RHP~\ref{rhp: m2}, that is, let
\begin{equation*}
	\text{$\mk{sol}(z)$ solves $\dbar$--RHP~\ref{rhp: m2} with $W \equiv 0$.}
\end{equation*}
For any admissible scattering data $\{ r(z), \{ z_j, c_j\}_{j=0}^{N-1} \}$ in RHP~\ref{rhp: m2}, the solution $\mk{sol}$ of this modified problem exists, and is equivalent, by an explicit transformation, to a reflectionless solution of the original Riemann Hilbert problem, RHP~\ref{rhp:m}, with the modified scattering data $\{ 0, \{ z_j, \widetilde{c}_j\}_{j=0}^{N-1} \}$
where, the modified connection coefficients $\widetilde{c}_j$ are given by
\begin{equation}\label{perturbed phase}
	\widetilde{c}_j(x,t) = c_j(x,t) \exp \lp -\frac{1}{\im \pi }
	\int_0^\infty \log (1 - |r(s)|^2) \lp \frac{1}{s-z_j} - \frac{1}{2s} \rp ds \rp.
\end{equation}
where $r(s)$ is the reflection coefficient, generated by the initial datum $q_0(x)$, given in RHP~\ref{rhp: m2}
\end{lemma}

\begin{proof}
With $W \equiv 0$, the $\dbar$-RHP for $\mk{sol}$ reduces to a Riemann Hilbert problem for a sectionally meromorphic function with jump discontinuities on the union of circles $\Sigma ^{(2)}$, see \eqref{Sigma2}.
The following transformation contracts each of the circular jumps so that the result $\widetilde{m}(z)$ has simple poles at each $z_k$ or $\overline z_k$ in $\mathcal{Z}$, and reverses the triangularity effected by \eqref{T def} and \eqref{m1}:
\begin{equation}\label{m tilde}
	\widetilde{m}(z) =
	\lb \prod_{k \in \Delta} \lp \frac{ 1}{z_k} \rp \rb^{\sig}	
	\mk{sol}(z) F(z)
	\lb \prod_{k \in \Delta} \lp \frac{ z - z_k}{z z_k -1} \rp \rb^{-\sig},
\end{equation}
where
\begin{equation*}
	F(z) = \begin{cases}
			\tril{ \frac{ c_j}{z-z_j} T(z)^2 e^{\Phi (z_j)} }
			& | z - z_j | = \rho, \ j \in \nabla \smallskip \\
			\triu{ \frac{ z-z_j}{c_j} T(z)^{-2} e^{-\Phi (z_j)}}
			& | z - z_j | = \rho, \ j \in \Delta \smallskip \\
			\triu{ \frac{ \widebar c_j}{z- \widebar z_j} T( z)^{-2} e^{\Phi (z_j)} }
			& | z - \widebar z_j | = \rho, \ j \in \nabla \smallskip \\
			\tril{ \frac{ z-\widebar z_j}{\widebar c_j} T( z)^{2} e^{-\Phi (z_j)}}
			& | z - \widebar z_j | = \rho, \ j \in \Delta   .
		\end{cases}
\end{equation*}
Clearly, the transformation to $\widetilde m$ preserves the normalization conditions at the origin and infinity.
Comparing \eqref{m tilde} to \eqref{V2} it is clear that the new unknown $\widetilde{m}$ has no jumps. From \eqref{T def}, RHP~\ref{rhp: m2}, and \eqref{m tilde} it follows that $\widetilde m(z)$ has simple poles at each of the points in $\mathcal{Z}$, the discrete spectrum of the original Riemann Hilbert problem, RHP~\ref{rhp:m}. A straightforward calculation shows that the residues satisfy \eqref{eq:resm}, but with \eqref{explicit-V-k} replaced by \eqref{perturbed phase}. Thus, $\widetilde m(z)$ is precisely the solution of RHP~\ref{rhp:m} with scattering data
$\{  \{z_k, \widetilde c_k\}_{k=0}^{N-1}, r \equiv 0 \}$.
The symmetry $r(s^{-1}) = \overline{ r(s) },\ s \in \R$, implies that the argument of the exponential in \eqref{perturbed phase} is purely real so that the perturbed connection coefficients maintain the reality condition $\widetilde{c}_j  = \im z_j | \widetilde{c}_j | $. Thus, $\mk{sol}$ is the solution of RHP~\ref{rhp:m} corresponding to an $N$-soliton, reflectionless, potential $\widetilde{q}(x,t)$ which generates the same discrete spectrum $\mathcal{Z}$ as our initial data, but whose connection coefficients \eqref{perturbed phase} are perturbations of those for the original initial data by an amount related to the reflection coefficient of the initial data. The solution of this discrete RHP is a rational function of $z$, whose (unique) exact solution always exists and can be obtained as described in Appendix~\ref{sec:msol}.
\end{proof}

As claimed above and proven in Appendix~\ref{sec:msol}, the RHP for $\mk{sol}$ can be solved exactly in closed form, but we will instead give the solution using the small norm theory of Riemann-Hilbert problems, see Appendix B \cite{KMM},  as this more naturally leads to the asymptotic form of the solution for $t \gg 1$.

The Riemann Hilbert problem for $\mk{sol}$ is ideally set up for asymptotic analysis.
The jump matrix $\vk{2}(z)$ satisfies
\begin{equation}\label{v2 bound}
	\| \vk{2}  -I \|_{L^p(\Sigma^{(2)})} \leq
	K_p \sup _{j \in \nabla \cup \Delta} e^{- C t \Im z_k \ | \xi - \Re z_j  |  }
	\le K_p e^{-C \rho^2 t}, \quad 1 \leq p \leq \infty,
\end{equation}
for some constant $K_p \geq 0$
independent of $(x,t)$. This implies that the jump matrices do not meaningfully, contribute to the asymptotic behavior of the solution. Instead,
the dominant contribution to the solution comes from the simple poles of $\mk{sol}$; those at $z = 0$ and, if the critical line $\Re z = \xi$ is passing through the neighborhood of one of the discrete spectra $z_{j} \in \mathcal{Z}$ of the original problem RHP~\ref{rhp:m}, those at $z_{j_0}$ and $\widebar z_{j_0}$.
Indeed, the following lemma describes this further simplification of $m^{(sol)}$ explicitly.

\begin{lemma}\label{lem: ex3rad}
Let $\xi = \frac{x}{2t}$ and let $j_0 =j_0(\xi) \in \{-1,0,1,\dots,N-1 \}$, be defined by \eqref{j0}. Suppose
\begin{equation}\label{mjsol 1}
	\text{
	$m_{j_0}^{(sol)}(z)$ solves RHP~\ref{rhp: m2} with $W(z) \equiv 0$ and $V^{(2)} \equiv I $.
	}
\end{equation}
Then, for any $(x,t)$ such that $|x/t| < 2$ and $t \gg 1$,  uniformly for $z \in \C$ we have
\begin{gather}\nonumber
  \mk{sol}(z)   =    \mk{sol}_{j_0}(z)   \left[
	I + \bigo{e^{-2\rho^2 t} }
	\right],
\shortintertext{and, in particular, for large $z$ we have}
	\mk{sol}(z) = \mk{sol}_{j_0}(z) \left[I + z^{-1}  {\bigo{e^{-2\rho^2 t} } }  + \bigo{z^{-2}} \right].\label{solas}
\end{gather}
Moreover, the unique solution $m_{j_0}^{(sol)}(z)$ to the above Riemann Hilbert problem, \eqref{mjsol 1}, is as follows:

\begin{subequations}\label{soliton solution}
\begin{enumerate}[i.]
\item if $j_0(\xi) = -1$, then all the $z_j$ are away from the critical line and
	\begin{equation}\label{eq:eq10}
		\mk{sol}_0(z) = I + z^{-1} {\sigma_1} ;	
	\end{equation}

\item if $j_0(\xi) \in \nabla$, then
\begin{gather}\label{eq:eq1a}
	\begin{gathered}
	\mk{sol}_{j_0}(z) = I + \frac{\sigma_1}{z}  +
	\begin{pmatrix}
		\frac{\alpha^{\nabla}_{j_0}(x,t)}{ z - z_{j_0} } & \frac{ \overline{ \beta^{\nabla}_{j_0}}(x,t)}{ z - \widebar z_{j_0} } \smallskip \\
		\frac{\beta^{\nabla}_{j_0}(x,t) }{ z-z_{j_0} } & \frac{\overline{ \alpha^{\nabla}_{j_0}}(x,t)}{ z-\widebar z_{j_0} }
	\end{pmatrix} \\
	\alpha^\nabla_{j_0}(x,t) = - z_{j_0}   \overline{\beta_{j_0}^\nabla}(x,t), \qquad
	\beta^\nabla_{j_0}(x,t) = \frac{ 2 \im \Im( z_{j_0}) z_{j_0} e^{-2 \varphi_{j_0}} }{1 + e^{-2 \varphi_{j_0}} };
	\end{gathered}
\end{gather}	
\item if $j_0(\xi) \in \Delta$, then
\begin{gather}\label{eq:eq1b}
	\begin{gathered}
		\mk{sol}_{j_0}(z) = I + \frac{\sigma_1}{z} +
		\begin{pmatrix}
			\frac{   \alpha^{\Delta}_{j_0}(x,t)}{z - \widebar z_{j_0}} & \frac{ \overline{\beta^{\Delta}_{j_0}}(x,t) }{z - z_{j_0} } \smallskip \\
			\frac{   \beta^{\Delta}_{j_0}(x,t) }{z - \widebar z_{j_0}} & \frac{ \overline{\alpha^{\Delta}_{j_0}}(x,t) }{z- z_{j_0} }
		\end{pmatrix}   \\
		\alpha^{\Delta}_{j_0}(x,t) = -   \overline{z} _{j_0}  \overline{\beta^{\Delta}_{j_0} } (x,t), \qquad
		\beta^{\Delta}_{j_0}(x,t) =- \frac{2 \im \Im(z_{j_0} ) \overline{z}_{j_0} e^{2 \varphi_{j_0}} }{1 + e^{2\varphi_{j_0}} }.
	\end{gathered}
\end{gather}
In cases ii. and iii. the real phase $\varphi_{j_0}$ is given by
\begin{gather}\label{eq:eq1c}
	\begin{gathered}
		\varphi_{j_0} = \Im(z_{j_0}) (x - 2 \Re (z_{j_0}) t - x_{j_0}) \\
		x_{j_0} = \frac{1}{2 \Im(z_{j_0})} \lp \log \lp \frac{|c_{j_0} |}{2 \Im(z_{j_0})}
		\prod_{\substack{{k \in \Delta} \\ k \neq j_0}}
		\left| \frac{ z_{j_0} - z_k}{z_{j_0} z_k - 1} \right|^2 \rp
		- \frac{\Im(z_{j_0})}{\pi} \int_0^\infty \frac{\log(1 - | r(s) |^2)}{| s- z_{j_0}|^2} ds \rp .
	\end{gathered}
\end{gather}
\end{enumerate}
\end{subequations}
\end{lemma}

\begin{proof}
We begin by proving that \eqref{soliton solution} solves \eqref{mjsol 1}.
The assumption that $V \equiv I$ and $W \equiv 0$ implies that $m_{j_0}^{(sol)}(z)$ is meromorphic with simple poles at $z=0$ and, if $j_0 \ne -1$, at both $z_{j_0}$ and $\widebar z_{j_0}$. If $j_0 = -1$, then \eqref{eq:eq10} is an immediate consequence of the condition $2$ in RHP~\ref{rhp: m2} and Liouville's theorem. For $j \neq -1$, observe that $C_0 := c_{j_0} T(z_{j_0})^2$ satisfies $C_0 = \im z_{j_0} |C_0|$ since $c_{j_0} = \im z_{j_0} |c_{j_0}|$ and $T(z) \in \R$ for $|z| = 1$,which follows from claim $i.$ in Lemma~\ref{lem:T}. For $j_0 \in \nabla$, this means that the RHP for $m_{j_0}^{(sol)}(z)$, is equivalent to the reflectionless, i.e., $r=0$, version of RHP~\ref{rhp:m} with poles at the origin and at the points $z_{j_0}$ and $\overline{z_{j_0}}$ with associated connection coefficient $C_0$. Then the symmetries \eqref{eq:barm}-\eqref{eq:mzinv} inherited by $\mk{sol}_{j_0}$ and \eqref{m2 nabla con} imply that  $\alpha^\nabla_{j_0} = - z_{j_0} \overline{ \beta^{\nabla}_{j_0}}$ and
\begin{equation*}
	\mk{sol}_{j_0} (z) = I + \frac{\sigma_1}{z} +
	\begin{pmatrix} \overline{ \beta^{\nabla}_{j_0}} & \beta^{\nabla}_{j_0} \end{pmatrix}
	\begin{pmatrix}
		-z_{j_0} (z- z_{j_0})^{-1} & \phantom{-z_{j_0}} (z- \overline{z}_{j_0})^{-1} \\
		\phantom{-z_{j_0}} (z- z_{j_0})^{-1} & - \overline{z}_{j_0} (z- \overline{z}_{j_0})^{-1}
	\end{pmatrix}.
\end{equation*}
The residue conditions \eqref{m2 nabla con} then yield four linearly dependant equations for the single unknown $\beta^{\nabla}_{j_0}$, each equivalent to $\beta^{\nabla}_{j_0} = C_0( 1 - \overline{z}_{j_0} \beta^{\nabla}_{j_0} (z_{j_0}- \overline{z_{j_0}})^{-1} )$, which gives \eqref{eq:eq1a} upon setting $\frac{|C_0|}{2\Im(z_j)} = e^{-2 \varphi_{j_0}}$.

For $j \in \Delta$, the computation is similar, but the new pole conditions \eqref{m2 delta con} exchanges the columns in which the two poles occur; we have $\alpha_{j_0}^{\Delta} = - \overline{z_{j_0}} \, \overline{ \beta^{\Delta}_{j_0}}$ and
\begin{equation*}
	\mk{sol}_{j_0} (z) = I + \frac{\sigma_1}{z} +
	\begin{pmatrix} \overline{ \beta^{\Delta}_{j_0}} & \beta^{\Delta}_{j_0} \end{pmatrix}
	\begin{pmatrix}
		- \overline{z}_{j_0} (z- \overline{z}_{j_0})^{-1} & \phantom{-z_{j_0}} (z- z_{j_0})^{-1} \\
		\phantom{-z_{j_0}} (z- \overline{z}_{j_0})^{-1} & - z_{j_0} (z- z_{j_0})^{-1}
	\end{pmatrix}.
\end{equation*}
Then residue relation \eqref{m2 delta con} leads to one linearly independent equation which can be solved trivially yielding the second line of \eqref{eq:eq1b}.

Now we show that $\mk{sol}_{j_0}$ gives the leading order behavior to $\mk{sol}$ for $t \gg 1$.
The ratio $\mk{err}(z) = \mk{sol}(z) \lp \mk{sol}_{j_0}(z) \rp^{-1}$ has no poles
(the computation proving this is identical to \eqref{eq:m3 sol free 1}-\eqref{eq:m3 sol free 2} below) and its jump matrix $\vk{err}(z) = \lp \mk{sol}_{j_0}(z)\rp \vk{2}(z) \lp \mk{sol}_{j_0}(z)\rp^{-1}$ satisfies the same estimate as in \eqref{v2 bound} since $\left|\mk{sol}_{j_0}(z) - I  - \frac{\sigma_1}{z} \right| = \bigo{ e^{-2\rho^2 t} }$ for $z \in \Sigma^{(2)}$.
	
It then follows from the small norm theory for Riemann Hilbert problems, \cite[Appendix B]{KMM} \cite[Appendix A]{DKMVZ},  that
\begin{equation*}
	\mk{sol}(z) = \mk{sol}_{j_0}(z)
	\left[
	I + \frac{1}{2\pi \im} \int_{\Sigma^{(2)}} \frac{ (I + \mu(s))(\vk{err}(s) - I)}{s-z} ds
	\right]
\end{equation*}
where $\mu \in L^2(\Sigma^{(2)})$ is the unique solution of $(1 - C_\vk{err}) \mu = C_\vk{err} I$, where $C_\vk{err}: L^2(\Sigma^{(2)}) \to L^2(\Sigma^{(2)})$ is the Cauchy projection operator
\begin{equation*}
	C_\vk{err}[f](z) = C_-[f (\vk{err}-I)] = \lim_{z' \to z}
	\int_{\Sigma^{(2)} } \frac{ f(s) (\vk{err}(s) - I)}{s-z'}  ds
\end{equation*}
where the limit is understood (possibly in the $L^2$ sense) to be taken non-tangentially from the minus (right) side of the oriented contour $\Sigma^{(2)}$. Existence and uniqueness of $\mu$ follows from the boundedness of the Cauchy projection operator $C_-$, which immediately implies
\[
	\| C_\vk{err} \|_{L^2(\Sigma^{(2)}) \to L^2(\Sigma^{(2)})} = \bigo{ e^{-2\rho^2 t}}.
\]
\end{proof}

\begin{remark}~\label{remark: solitons}
The different formulae for $\mk{sol}_j(z)$ for $j \in \nabla$ or $j \in \Delta$ in Lemma~\ref{lem: ex3rad} is an artifact of the conjugation by $T(z)$ in \eqref{m1} which transforms exponentially growing pole residues into decaying residues. As is shown below, near the line $x = 2t \Re(z_j)$ the dominant contribution to $m(z)$ the solution of the original Riemann Hilbert problem is
of the form
\begin{equation}
	\begin{aligned}
	q_j^{(sol)}(x,t) &\equiv T(\infty,\xi)^{-2} \lim_{z \to \infty} z (\mk{sol}_j)_{21} (z; x,t) \\
		&= \begin{cases}
		\phantom{z_j^2}e^{\im \vartheta_+} (1 + \beta_j^\nabla(x,t) ) & x < 2t \Re(z_j) \\
		z_j^2 e^{\im \vartheta_+}  (1 + \beta_j^\Delta(x,t) ) & x > 2t \Re(z_j),
	\end{cases}
	\end{aligned}
\end{equation}
where $\vartheta_+$ is a real constant, and $\beta_j^\nabla$ and $\beta_j^\Delta$ are given by \eqref{eq:eq1a} and \eqref{eq:eq1b} respectively and the extra factor of $z_j^2$ for $x> 2t \Re(z_j)$ accounts for the additional factor in $T(\infty, \xi)$ for $j \in \Delta$.
However, since $\overline{z_j} = z_j^{-1}$, it's a simple algebraic exercise to show that the two formulae are identical, so that either formula gives
\begin{equation*}
	q_j^{(sol)}(x,t) = e^{i \vartheta_+} \frac{1 + z_j^2 e^{-2\varphi_j} }{1+ e^{-2\varphi_j}} = e^{i \vartheta_+} \sol(x-x_j, t ; z_j  )
\end{equation*}
where $\sol(x,t ;  z)$ defined by \eqref{eq:defsol} is the formula for the dark 1-soliton.
\end{remark}

We now complete the original goal of this section by using $\mk{sol}$ to reduce $\mk{2}$ to a pure $\dbar$-problem which will be analyzed in the following section.

\begin{lemma}\label{lem: m3}
Define the function
\begin{equation}\label{m3 def}
	\mk{3}(z) = \mk{2}(z) \lp \mk{sol}(z) \rp^{-1}.
\end{equation}
Then, $\mk{3}$ satisfies the following $\dbar$-problem.

\begin{DBAR}\label{dbar: m3}
Find a $2\times2$ matrix-valued function $\mk{3}(z)$ such that
\begin{enumerate}
	\item $\mk{3}(z)$ is continuous in $\C$, and analytic in $\C \backslash \overline{\Omega}$.
	\item $\mk{3}(z) = I + \bigo{z^{-1}}$ as $z \to \infty$.
	\item For $z \in \C$ we have
	\begin{equation} \label{eq:RHm3}
		\dbar \mk{3}(z) = \mk{3}(z) W^{(3)}(z)
	\end{equation}
	where
	$W^{(3)} := \mk{sol}(z) W(z) \lp \mk{sol}(z) \rp^{-1}$---with
	  $W(z)$ defined after \eqref{W2}---is supported in $ \Omega$.
\end{enumerate}
\end{DBAR}
\end{lemma}

\begin{proof}
It follows directly from \eqref{m3 def} that $\mk{3}$ has no jumps on the
disk boundaries $| z- z_j | = \rho$ nor $| z- \overline{z}_j | = \rho$
since $\mk{sol}$ has exactly the same jumps as $\mk{2}$ on these contours. The normalization condition and $\dbar$ derivative of $\mk{3}$ follow immediately from the properties of $\mk{2}$ and $\mk{sol}$. It remains to show that the ratio also has no isolated singularities.
At the origin we have  $\lp \mk{sol}(z) \rp^{-1} = (1 - z^{-2})^{-1} \sigma_2 \lp \mk{sol}(z) \rp^T \sigma_2$, formula already used in Lemma \ref{lem:det},
so that
\begin{equation}\label{eq:m3 sol free 1}
	\lim_{z \to 0} \mk{3}(z) = \lim_{z \to 0}  \frac{ \lp z \mk{2}(z) \rp \sigma_2 \lp z \mk{sol}(z)^T  \rp \sigma_2}{z^2-1} = -(\sigma_1 \sigma_2)^2 = I
\end{equation}
so $\mk{3}(z)$ is regular at the origin. If $\mk{2}$ has poles at $z_{j_0}$ and $\widebar z_{j_0}$ on the unit circle then from the form of the residue relation we have local expansions in a neighborhood of $z_{j_0}$ of the form
\begin{equation}\label{eq:m3 sol free 2}
	\begin{gathered}
	\mk{2}(z) =
		\begin{bmatrix} \mk{2}_{12} (z_{j_0}) \bigskip \\ \mk{2}_{22} (z_{j_0}) \end{bmatrix}
		\begin{bmatrix} \frac{ c_{j_0}}{z-z_{j_0}} &1 \end{bmatrix}
		+ \begin{pmatrix} *_{11} & 0 \\ *_{21} & 0 \end{pmatrix}
		+ \bigo{z-z_{j_0}} \\
	\mk{sol}(z)^{-1} =
		\frac{z_{j_0}^2}{z_{j_0}^2-1} \lp
	 	\begin{bmatrix}  1 \medskip \\ \frac{-c_{j_0}}{z-z_{j_0}} \end{bmatrix}
	 	\begin{bmatrix} \mk{sol}_{22}(z_{j_0}) & -\mk{sol}_{12}(z_{j_0}) \end{bmatrix}
	 	+\begin{pmatrix} 0 & 0 \\ \$_{21} & \$_{22} \end{pmatrix}
	 	+\bigo{ z-z_{j_0} } \rp \\
	\end{gathered}
\end{equation}
where $*_{jk}$ and $\$_{jk}$ are constants. Taking the product gives
\begin{equation*}
	\mk{3}(z) =\frac{z_{j_0}^2}{z_{j_0}^2-1} \lp
		\begin{bmatrix} *_{11} \medskip \\ *_{21} \end{bmatrix}
		\begin{bmatrix} \mk{sol}_{22}(z_{j_0}) & -\mk{sol}_{12}(z_{j_0}) \end{bmatrix}
		+ \begin{bmatrix} \mk{2}_{12} (z_{j_0}) \bigskip \\ \mk{2}_{22} (z_{j_0}) \end{bmatrix}
		\begin{bmatrix} \$_{21} & \$_{22} \end{bmatrix} + \bigo{1} \rp
\end{equation*}
which shows that $\mk{3}(z)$ is bounded locally and the pole is removable. A similar argument shows that the pole at $\widebar z_{j_0}$ is removable. Finally, because $\det \mk{sol}(z) = (1 - z^{-2})$ we must check that the ratio is bounded at $z = \pm 1$. This follows from observing that the symmetries $m(z) = \sigma_1 \overline{ m(\widebar z) } \sigma_1 = z^{-1} m(z^{-1}) \sigma_1$, given in Lemma~\ref{lem:barm}, applied to the local expansion of $\mk{2}$ and $\mk{sol}$ imply that
\begin{equation*}
		\mk{2}(z) = \begin{pmatrix} c & \pm c \\ \pm \overline{c} & \overline{c} \end{pmatrix}
			+ \bigo{z\mp 1}
		\qquad
		\mk{sol}(z)^{-1} = \frac{\pm 1}{2(z \mp 1)} \begin{pmatrix} \overline{\gamma} & \mp \gamma \\
		\mp \overline{\gamma} & \gamma \end{pmatrix}
			+ \bigo{1}	
\end{equation*}	
for some constants $c$ and $\gamma$. Taking the product it's immediately clear the singular  part of $\mk{3}(z)$ vanishes at $z =\pm 1$.
\end{proof}

In Sect. \ref{sec: step4} we will prove the following lemma.
\begin{lemma}\label{lem:0crE1}
There exist constants $t_1 $ and $c $ such that the $z$--independent coefficient $\mk{3}_1(x,t)$  in the asymptotic expansion
\begin{equation*}
 \mk{3}(z) =I+   \frac{\mk{3}_1(x,t)}{z} + o( z ^{-1})
\end{equation*}
satisfies
\begin{equation*}
 \text{$|\mk{3}_1(x,t) |\le c   t ^{-1}$ for $|x/ t |<2$ and $t\ge t_1$.}
\end{equation*}
\end{lemma}

\subsection{Step 4: Solution of the  $\dbar$ problem \ref{dbar: m3} and asymptotics as $t \to \infty$}\label{sec: step4}

\begin{lemma}\label{lem:lemJ}
Consider  the    following   operator $J$:
\begin{equation}  \label{eq:lemJ1}
	JH( z):= \frac{1}{\pi} \int _{\C}
		\frac{H( \varsigma )W^{(3)}( \varsigma )}{\varsigma -z} dA(\varsigma ).
\end{equation}
Then  we have  $J:L^ \infty (\C )\to L^ \infty  (\C )\cap C^0 (\C )$
and for any fixed $\xi _0 \in (0,1)$  there exists a    $C =C(q_0, \xi _0)$  s.t.
\begin{equation}  \label{eq:lemJ2}
	\| J \| _{L^ \infty  (\C ) \to L^ \infty  (\C )} \le C t^{-\frac{1}{2}}
	\text{ for all $t \gg1$ and for $\left |\frac{x}{2t}\right |\le \xi _0 $}.
\end{equation}
\end{lemma}

\begin{proof}
To prove \eqref{eq:lemJ2} we follow the argument in  Prop. 2.2 \cite{DM}.
It is not restrictive to consider only the proof of
$\| J H \| _{ L^\infty (\C )} \le C t^{-\frac{1}{2}}  \| H  \| _{L^\infty (\C)}$
for  $H\in L^\infty (\Omega _1)$.
Recall the definition of $W^{(3)}(z) := \mk{sol}(z) W(z) \lp \mk{sol}(z) \rp^{-1}$.
From Lemma~\ref{lem:det} we have $\det \mk{sol}(z)  = 1 - z^{-2}$,
and Lemma~\ref{lem: ex3rad} implies that for $z \in \overline{\Omega_1}$
there exists a fixed constant $C_1$ s.t. the matrix norm
$| \mk{sol}(z) | \leq C_1 |z|^{-1} \sqrt{ 1 + |z|^2 } = C |z|^{-1} \langle z \rangle$.
Then
\begin{equation} \label{eq:lemJ-22}
    | W^{(3)}(z) | \le |\mk{sol}(z) | ^2   |1-z ^{-2}|^{-1}   | W (z) |
    \le C_1     \langle z \rangle^2     |z^2-1|^{-1}   | W (z) |.
\end{equation}
Since
$\langle \varsigma  \rangle | \varsigma +1| ^{-1} =\bigo{1}$  in $ \Omega _1$,
for a fixed constant $c_1$ we have
 \begin{equation} \label{eq:lemJ22}
 	|JH (z)| \le c _1 \| H  \| _{L^\infty (\C)}  \int _{  \Omega _1  }
  	\frac{  \langle \varsigma  \rangle \left| \dbar R (\varsigma  )  e^{- \Re \Phi (\varsigma )} \right|}
	{| \varsigma - z| \ |\varsigma -1|} dA(\varsigma  ).
\end{equation}
By Lemma~\ref{lem:phase1a} the hypothesis that there is a  constant $ \xi _0   \in (0, 1)$
s.t. $| \xi |\le  \xi _0  $ is crucial  in order to have
$|\Re \Phi (\varsigma )  | \ge c t |uv|$ for a fixed $c=c(\xi _0 )>0$.
Notice also that \eqref{eq:lemJ22} contains an extra singularity with respect to
Proposition 2.2 in \cite{DM}.
It is to offset this that our extensions of $R(z)$ in Lemma~\ref{lem:extR},
in particular formula \eqref{R12 def}, are somewhat more elaborate than in \cite{DM}.
To  simplify notation we will normalize the problem and  suppose $\theta _0 =\pi /4$
so  that $\Omega _1$ is the sector defined by $\arg (z)\in [0, \pi /4]$.
Into the integral in the r.h.s. of \eqref{eq:lemJ22},  we insert the partition of unity:
$\chi _{[0,1)}(|\varsigma |)+\chi _{[ 1, 2)}(|\varsigma |) + \chi _{[ 2, \infty )}(|\varsigma |) $.
We prove first  the following, where the 1st inequality is obvious since
$\langle \varsigma  \rangle | \varsigma -1| ^{-1} \le \kappa$
for $|\varsigma |\ge 2$, for a fixed $\kappa$:
\begin{equation} \label{eq:lemJ22c}
	\int _{  \Omega _1  }
	\frac{ \left| \dbar R (\varsigma  )  e^{- \Re \Phi (\varsigma )} \right|
	\chi _{[ 2, \infty )}(|\varsigma | ) }
	{| \varsigma   -z| \ |\varsigma -1|} \langle \varsigma  \rangle dA(\varsigma  )
	\le
	\kappa \int _{  \Omega _1  }
	\frac{  \left| \dbar R  (\varsigma ) e^{- \Re \Phi (\varsigma )} \right|
	\chi _{[ 1, \infty )}(|\varsigma | ) }{| \varsigma   -z|  } dA(\varsigma  )
	\le C t^{-\frac{1}{2}}  .
\end{equation}
Set  $ \varsigma =u+\im v$, $z=z_R +\im z_I$, $1/q+1/p=1$  with $p>2$.
To prove the 2nd inequality in \eqref{eq:lemJ22c} we replace $| \dbar R |$ by the 3 terms
in the r.h.s. of \eqref{eq:extR2}.
For $\varsigma \in \Omega_1$ with $| \varsigma| \geq 1$ we use Lemma~\ref{lem:phase1a} to write $ \Re \Phi(\varsigma) > ct u v > c' t v$.

When replacing $| \dbar R  (\varsigma )  |$, the terms in
 in \eqref{eq:extR2} involving
$f(| \varsigma |) =  r'(| \varsigma |)$ or $f(| \varsigma |) = \varphi (| \varsigma |)$ give
\begin{equation} \label{eq:lemJ22b}
	\begin{aligned}
	&\int _0^\infty dv  e^{- c' t v}\int _{v}^{\infty}
  	\frac{   \chi _{[ 1, \infty )}(|\varsigma |)   |f(| \varsigma |)|}{  | \varsigma   -z|} du
	\le
	c'' \|  f\| _{L^2(\R)}  \int _0^\infty dv  e^{- c' t v}  | v-z_I  |^{-\frac{1}{2} }
   	\le C t^{-\frac{1}{2}} \|  f\| _{L^2(\R) }.
	\end{aligned}
\end{equation}
Here we have used
\begin{equation}  \label{eq:lemJ22b1}
	\int _{v}^{\infty}| f(\sqrt{u^2+v^2}) |^2 du =
	\int _{\sqrt{2}v}^{\infty} | f(\tau ) |^2  \frac{\sqrt{u^2+v^2}}{u}  d\tau
	\le  \sqrt{2}\int _{\sqrt{2}v}^{\infty} | f(\tau ) |^2     d\tau  .
\end{equation}
The term  $|\varsigma| ^{-\frac{1}{2}}$ in \eqref{eq:extR2} gives
\begin{equation} \label{eq:lemJ22a}
	\begin{aligned} &
	\int _0^\infty dv \ e^{- c' t v}\int _{v}^{\infty}
  	\frac{   \chi _{[ 1, \infty )}(|\varsigma |) }{| \varsigma | ^{ \frac{1}{2}} | \varsigma   -z|} du
	\le
	\int _0^\infty dv \ e^{- c' t v}
	\| \, |\varsigma | ^{- \frac{1}{2}} \|_{L^p(v,\infty )}
	\|  \, |\varsigma-z | ^{- 1} \|_{L^q(v,\infty )} \\&
	\le
	c''  \int _0^\infty dv\, e^{- c' t v} v^{1/p-1/2}  \left| v-  z_I \right|^{-\frac{1}{p}}
	\le
	4 c'' t^{-1/2} \int _0^\infty ds \ e^{- c' s}   s^{ -\frac{1}{2}}
	\le C t^{-\frac{1}{2}}.
\end{aligned}
\end{equation}
In the penultimate step above we've made the elementary observation that for any $a,b,c > 0$,
\[
	\begin{aligned}
		\int_{0}^\infty e^{-c v} v^{-a} |v-v_0|^{-b} dv
		& \leq \int\limits_{\mathclap{v> |v-v_0|}} e^{-c |v-v_0|} |v-v_0|^{-(a+b)} dv
		+\int\limits_{\mathclap{0< v< |v-v_0|}} e^{-c v} v^{-(a+b)} dv  \\
		& \leq 2\int_{-\infty}^\infty e^{-c |s|} |s|^{-(a+b)} ds = 4\int_{0}^{\infty} e^{-cs} s^{-(a+b)} ds.
	\end{aligned}
\]
Thus we have proved \eqref{eq:lemJ22c}.
The next inequality is
\begin{equation} \label{eq:lemJ22z}
	\int _{  \Omega _1  } \frac{ \langle \varsigma \rangle \, |\dbar R (\varsigma  )  e^{- \Re \Phi (\varsigma )}|
	 \chi _{[ 1, 2)}(|\varsigma | ) }
	{| \varsigma   -z| \ |\varsigma -1|}dA(\varsigma  )
	\le \sqrt{5} c_1 \int _{  \Omega _1  } \frac{   e^{- \Re \Phi (\varsigma )} \chi _{[ 1, 2)}(|\varsigma | ) }
	{| \varsigma   -z|  } dA(\varsigma  )
	\le C t^{-\frac{1}{2}}  .
\end{equation}
The first inequality is obtained from \eqref{eq:extR21},
that is $\left| \dbar R_j(\varsigma ) \right| \le   c_1  |\varsigma -1| $, and
noting that $ \langle \varsigma \rangle \leq \sqrt{5}$ for $|\varsigma| \leq 2$. The second inequality is    \eqref{eq:lemJ22b} applied to $f(|z | )=  \chi _{[ 1, 2)}(|z | )$.
From \eqref{eq:lemJ22c} and \eqref{eq:lemJ22z} we conclude that for some $C(q_0, \xi _0)$
\begin{equation} \label{eq:lemJ22x}
	\int _{  \Omega _1  } \frac{ \left| \dbar R (\varsigma  )  e^{- \Re \Phi (\varsigma )} \right|
	\chi _{[ 1, \infty)}(|\varsigma | ) } {| \varsigma   -z| \ |\varsigma -1|} dA(\varsigma  )
	\le  C(q_0, \xi _0) t^{-\frac{1}{2}}  .
\end{equation}

Finally, consider the last inequality, namely
\begin{equation} \label{eq:lemJ22d}
	\int _{  \Omega _1  } \frac{ \left| \dbar R (\varsigma  )  e^{- \Re \Phi (\varsigma )} \right|
	\chi _{[ 0,1)}(|\varsigma | ) } {| \varsigma   -z| \ |\varsigma -1|} dA(\varsigma  )
	\le  C t^{-\frac{1}{2}}.
\end{equation}
Introducing the change of variables $w = \widebar{1/z}$ and $\tau = \widebar{1/ \varsigma}$,
noting that $dA(\varsigma) = | \tau  | ^{- 4}dA(\tau)$,  
$\Phi (  \widebar{\tau^{-1}}; x,t) = \widebar{\Phi ( \tau ; x,t)}$
(\cf\ \eqref{eq:Mvx}), 
and using the symmetry $R(\widebar{ \tau^{-1}}) = \widebar{R(\tau)}$ (\cf\ Lemma~\ref{lem:extR}), 
equation \eqref{eq:lemJ22d} becomes
\begin{equation} \label{last1}
	\int_{  \Omega _1  } \frac{ \left| \partial_{\tau} \widebar{R ( \tau) } e^{ -\Re \Phi (\tau )} \right|
	\chi _{[ 1,\infty)}( |\tau | ) } { | \tau^{-1}   - w^{-1}| \ |\tau^{-1} -1| |\tau|^4 } 
	\left| \pd{\tau}{\widebar{\varsigma}} \right| dA(\tau)
	=
	|w| \int_{  \Omega _1  } \frac{ \left| \dbar R ( \tau) e^{ -\Re \Phi (\tau )} \right|
	\chi _{[ 1,\infty)}( |\tau | ) } { | \tau   - w| \ |\tau -1|  } dA(\tau) .
\end{equation}
%
%
%
Now consider separately large and small values of $|w|$:
if $ |w| \le 3$ we are back to \eqref{eq:lemJ22x};
if $ |w| \ge 3$ we can bound the r.h.s. of \eqref{last1} by
\begin{equation*}
 	3 \int\limits_{\mathclap{  |\tau | \ge  \frac{|w|}{2}  }}
  	\frac{ |   \dbar R (\tau ) e^{ -\Re \Phi ( \tau )} |  }{| \tau - w |    }
	\chi _{\Omega_1} (\tau) dA( \tau )
	+ 2  \int\limits_{\mathclap{ 1\le |\tau | \le \frac{|w|}{2}  }}
  	\frac{ |  \dbar R (\tau ) e^{ -\Re \Phi ( \tau )} |  }{| \tau-1|    }
	\chi _{\Omega_1 } (\tau) dA(\tau ).
\end{equation*}
Both terms are bounded by $C t^{-\frac{1}{2}}$ for a fixed $C=C(q_0,\xi _0)$
since they can be treated like the middle term in \eqref{eq:lemJ22c}.
So we have proved \eqref{eq:lemJ2}.
\end{proof}

Lemma~\ref{lem:lemJ} implies $\mk{3}=I+J \mk{3} $.
Indeed, since $\frac{1}{\pi}  \ \frac{1}{z}*\overline{\partial} \phi = \phi$
for any test function $\phi \in C^\infty _0 (\C , \C ),$ see \cite[Proposition 4.8 p.210 ]{taylor},
we can write
\begin{multline*}
	\int _\C  \mk{3}(w) W^{(3)}(w)  \phi (w) dA(w) =
	\int _\C  \mk{3}(w) W^{(3)}(w) \lb \frac{1}{\pi} \int_{\C}\frac{\dbar \phi( z )}{z -w}dA(z) \rb dA(w) \\
	= - \int _\C  J\mk{3}(z) \dbar \phi( z )  dA(z )
\end{multline*}
where we exploit the fact, proved in the course of Lemma~\ref{lem:lemJ},
that $\frac{\mk{3}(w) W^{(3)}(w) \dbar \phi( z ) }{w -z} \in L^1 (\C ^2)$,
so that we can exchange order of integration.
Since  Lemma~\ref{lem:lemJ}   implies that $J\mk{3}(z)$ is a continuous function in $z$
uniformly bounded in $\C$,  we conclude that
$\dbar (\mk{3} - J\mk{3})=0 $ in the distributional sense.
By elliptic regularity  $\mk{3} - J\mk{3}$ is smooth, see \cite[Theorem 11.1  p.379]{taylor},
and so it is holomorphic in $\C$.
Finally, by point 2. in RHP~\ref{dbar: m3} we get $\mk{3}=I+J \mk{3} $.

\begin{proof}[Lemma \ref{lem:0crE1}] \
The above discussion allows us to write
\begin{equation} \label{eq:eqE11}
	\mk{3}_1  = - \frac{1}{\pi} \int _{\C}  \mk{3}(z) W^{(3)}( z )  dA(z ).
\end{equation}
Since $\mk{3}=I+J \mk{3} $, Lemma~\ref{lem:lemJ} implies that for $t$ large
we have $ \| \mk{3}   \| _{L^\infty (\C )} \le c $ for a fixed constant  $c$ and for all $|\xi |\le \xi _0$.
The proof proceeds along the same lines as the proof of Lemma~\ref{lem:lemJ}. 
Again, we restrict to $z \in \Omega_1$ for simplicity, the proof in the rest of plane being similar. 
Using \eqref{eq:lemJ-22}, like in  \eqref{eq:lemJ22}, we have
\[
	\left| \frac{1}{\pi} \int _{\Omega_1}  \mk{3}(z) W^{(3)}( z )  dA(z ) \right| 
	\leq C \int _{  \Omega _1  }  \langle z \rangle \left|  \dbar R (z) \right| e^{- \Re \Phi(z) } 
	|z-1| ^{-1} \chi _{[1,\infty )}(|z|)   {dA(z)}.
\]
Inserting the partition of unity $\chi _{[0,1)}(| z |)+\chi _{[ 1, 2)}(| z |) + \chi _{[ 2, \infty )}(| z |) $ into the above integral we consider each term separately. 
For the term with $\chi _{[2,\infty )}(|z|)$ the factor $\langle z \rangle   |z-1| ^{-1} = \bigo{ 1}$, and
fixing a $p>2$ (so that $q \in (1,2)$\,) we get the upper bound 
\begin{equation}  \label{eq:eqE121}
	\begin{aligned}
	& \int_{\Omega _1}  e^{- \Re \Phi(z) } \left|  \dbar R (z) \right|  \chi _{[2,\infty )}(|z|) dA(z)
	\leq C \int _{  \Omega _1  } e^{- \Re \Phi  (z)}  
	\lp |z | ^{-\frac{1}{2}} + \sum_{\mathclap{f \in \{ r',\varphi \}} } | f (| z |)| \rp  
	\chi _{[1,\infty )}(|z|) {dA(z)}   \\
	& \le   C_1 \lb
	\int _0^\infty dv \, \| e^{-c t uv}\| _{L^2 ( \max\{v, 1/\sqrt{2}\}, \infty)}
	+ \int _0^\infty dv \,  \| e^{-c t uv}\| _{L^p ( \max\{v, 1/\sqrt{2}\}, \infty)}
	\||z | ^{-\frac{1}{2}} \| _{L^q ( v , \infty)} \rb \\
	& \le C_2 \int _0^\infty dv \, e^{- c' t v}
	( t ^{-\frac{1}{2}} v^{-\frac{1}{2}}  + t ^{-\frac{1}{p}} v^{-\frac{1}{p}+\frac{1}{q}-\frac{1}{2}})
	\le C_3 (t ^{-1} + t ^{-\frac{1}{2}-\frac{1}{q}})  \le C_3 t^{-1}.
\end{aligned}
\end{equation}
For $z \in [0,2]$, $\langle z \rangle \leq \sqrt{5}$, so it will be omitted from the remaining estimates. 
For the term with $\chi_{[1,2]}(|z|)$, using  \eqref{eq:extR21} for the first inequality and applying the
inequalities in \eqref{eq:eqE121} to  $f= \chi _{[1,2 ]}$,  we obtain
\begin{equation}  \label{eq:eqE122}
	\int _{  \Omega _1  } e^{-\Re \Phi } | \dbar R (z)|  \, 
	|z-1| ^{-1}  \chi _{[1,2 ]}(|z|)   {dA(z)}
	\le c_1 \int _{  \Omega _1  }  e^{-\Re \Phi }   \chi _{[1,2 ]}(|z|)   {dA(z)}  \le C t ^{-1} .
\end{equation}
For the term $\chi_{[0,1]}$, the change of variables $ w= \widebar{z^{-1}}$ gives, as in \eqref{last1} 
\begin{multline*}
 	\int _{  \Omega _1  } e^{-\Re\Phi(z) } 
	| \dbar R (z)|  |z-1| ^{-1} \chi _{[0,1] )} (|z|)  {dA(z)}  \\
	= \int _{  \Omega _1  } e^{ -\Re \Phi(w) } | \dbar R (w)|   | w -1| ^{-1}
	\chi _{[1,\infty )}(|w|) | w | ^{-1} {dA(\zeta )},
\end{multline*}
which is bounded by the previous estimates \eqref{eq:eqE121}-\eqref{eq:eqE122}. 
Summing the last three inequalities yields the desired estimate.
\end{proof}

\subsection{Proofs of Theorems~\ref{thm:main1} and \ref{thm:main2}}\label{sec:pfmain1} \ \\

\textit{Proof of Theorem~\ref{thm:main1}.} 
For $z$ large and in $\C \backslash \overline{\Omega}$ we have
$m^{(1)}(z)=m^{(2)}(z)$. So by  \eqref{T asy}  and  \eqref{m3 def}
\begin{multline*}
	m(z) = T(\infty ,\xi)^\sig  \mk{2}(z)    T(z,\xi) ^{-\sig} =
	T(\infty ,\xi)^\sig   \mk{3}(z)  \mk{sol}(z)   T(\infty ,\xi) ^{-\sig} \\
	\times  \left ( I -z^{-1}  \left ( \sum_{k\in \Delta} 2\im \Im (z_k)  -
	\frac{1}{2\pi \im }   \int_0^\infty \log (1 - |r(s)|^2) ds\right  ) ^{-\sig}
	+\littleo{z^{-1} } \right  ).
\end{multline*}
Since the first two terms of the factor in the last line are  diagonal,
by $\mk{3}(z) = I + z ^{-1}\bigo{t^{-1}}+\littleo{z ^{-1}}$,
by \eqref{solas} and by  \eqref{eq:eq10}--\eqref{eq:eq1b}
we obtain  for $ | x - 2 \Re(z_{j_0}) t | \leq \rho t$ and  $j_0\in \nabla $
 \begin{equation}\label{pfmain11}
	q(x,t) = \lim_{z \to \infty} zm_{21}(z)
	=  -T(\infty ,\xi)^{-2}\im z _{j_0} \left( \im  \Re (z _{j_0}) +\Im (z _{j_0}) \tanh \varphi  _{j_0} \right )
	+ \bigo{t^{-1}} .
\end{equation}
For  $ | x - 2 \Re(z_{j_0}) t | \leq \rho t$ and $j_0\in \Delta $ we have instead
\begin{equation}\label{pfmain12}
	q(x,t)= \lim_{z \to \infty} zm_{21}(z)
	=  -T(\infty ,\xi)^{-2}\im \overline{z} _{j_0} \left ( \im  \Re (z _{j_0})
	+ \Im (z _{j_0}) \tanh \varphi  _{j_0} \right ) +  \bigo{t^{-1}} .
\end{equation}
In  \eqref{pfmain11}, the main term can be written as
\begin{equation} \label{pfmain13}
       \delta _+^{-1}\prod_{k <  {j_0} }   z_k  ^2 \, \sol ( x-x _{j_0}, t ; z _{j _0}  )
	\text{ where  }
	\delta_+ := \exp  \lp \frac{1}{2\pi \im} \int_0^\infty  \frac{ \log(1 - |r(s)|^2 )}{s} ds \rp
\end{equation}
using the formula for $T(\infty ,\xi)$, the obvious fact that
$\Delta  =\Delta \backslash \{ {j_0} \}$ for  $j_0\in \nabla $ and  by \eqref{eq:ord1},
which implies  $ \Delta \backslash \{  {j_0} \}=\{  k: k< {j_0}   \}$.
Equation~\eqref{pfmain13} also represents the main term in \eqref{pfmain12}.
By $\displaystyle \lim _{x \to \infty} \sol ( x  _{j }, t; z _{j } )=1$  and
$\displaystyle \lim _{x \to - \infty} \sol (x  _{j }, t; z _{j } )=z_j^2$
it  is elementary to see that  \eqref{pfmain13} differs from the r.h.s. of
\eqref{soliton separation} by  $\bigo{ t^{-1} }$.
We obtain similarly  \eqref{soliton separation}  also when $j_0=-1$,
that is when  $ | x - 2 \Re(z_{j_0}) t |  > \rho t$, where we have
\begin{equation}\label{pfmain121}
	q(x,t) =\lim_{z \to \infty} zm_{21}(z)  =
	-T(\infty ,\xi)^{-2} + \bigo{t^{-1}} =
	\delta _+^{-1}\prod_{k \le   \sup \Delta }   z_k  ^2  +    \bigo{t^{-1}} .
\end{equation}
Clearly, \eqref{pfmain121} differs from the r.h.s. of \eqref{soliton separation} by   $\bigo{ t^{-1} }$.
Finally notice that for $q^{(sol),N}(x,t)$,  the $N$--soliton potential related to the solution $\mk{sol}(z)$ in Lemma~\ref{lem: m sol}, our analysis proves \eqref{eq:thmmain1} since  formulas  \eqref{pfmain11},
\eqref{pfmain12} and \eqref{pfmain121}  hold also for $q^{(sol),N}(x,t)$.
\qed

\textit{Proof of Theorem~\ref{thm:main2}.}
Given $q_0$ close to the $M$--soliton $ q^{(sol),M} (x ,0)$  we obtain the information
on the poles and coupling constants in \eqref{eq:prop1} by the
Lipschitz continuity of maps such \eqref{eq:eqm13} in Lemma~\ref{lem:jf1b}
and \eqref{eq:eqm13jf3} in Lemma~\ref{lem:jf3}.
Furthermore, we can apply  Lemma~\ref{lem:a3bis} to $q_0$.
Hence we can apply Theorem~\ref{thm:main1} to $q_0$ obtaining \eqref{soliton separation}.
By elementary computations \eqref{soliton separation} yields  \eqref{soliton separation0}.
\qed

\appendix

\section{$N$-solitons}\label{sec:msol}
  Consider $N$   points
$z_j =  e^{\im \theta_j}$, labeled such that $0 < \theta_0 < \dots < \theta_{N-1} < \pi$ and set
\begin{equation}\label{eq:sol a}
	a(z) =  \prod_{k=0}^{N-1}  \frac{z- z_k }{z - \widebar z_k}  .
\end{equation}
Notice that
\begin{equation} \label{eq:thetacond2}
	 \prod_{k=0}^{N-1} z_k^2 =a(0) .
\end{equation}
Consider also corresponding coupling constants $c_j$ with $c_j=  \im z_j|c_j|$
and let $c_j(x,t)= c_j   e^{\Phi(z_k; x,t)} $ like in \eqref{explicit-V-k}.
Then consider the unique (by the proof of Lemma \ref{lem:det}) solution of the corresponding RHP~\ref{rhp:m} (with $r(z)\equiv 0$) satisfying the symmetries of    Lemma~\ref{lem:barm}.
It is   a meromorphic function approaching identity as $z \to \infty$ with $2N+1$ simple poles $m (z; x,t)$ with a partial fraction expansion of the form
\begin{equation} \label{eq:mpartial1}
	m (z;x,t) = I + \frac{\sigma_1}{z} + \sum_{k=0}^{N-1} \frac{1}{z-z_k}
	\twovec{ \alpha_k(x,t) & 0 }{\beta_k(x,t) & 0}
	+ \sum_{k=0}^{N-1} \frac{1}{z- \overline{ z}_k  }
	\twovec{ 0 &     \widehat{\beta} _k(x,t)  }{ 0 &    \widehat{\alpha}_k(x,t) }.
\end{equation}	
Assuming for a moment that $m (z; x,t)$ exists we will consider the $N$-\textit{soliton}
 \begin{equation} \label{msol}
	q^{(sol),N} (x,t)  :=  \lim _{z\to \infty }z  \, m _{21}(z; x,t ) = 1 + \sum_{k=0}^{N-1} \beta_k(x,t).
\end{equation}
 Before discussing the boundary values of $q^{(sol),N} (x,t)$ and proving Lemma \ref{lem:inval}  we study  the existence of $m (z; x,t)$. By \eqref{eq:barm}  we have
\begin{equation}\label{eq:solsymm1}
	\widehat \alpha_k (x,t)= \widebar \alpha_k(x,t),
	\qquad
	\widehat \beta_k (x,t)= \widebar \beta_k (x,t)
\end{equation}
and by \eqref{eq:mzinv}    the additional symmetry
\begin{equation}\label{eq:solsymm2}
	\alpha_k (x,t) = -z_k   \overline{\beta}_k(x,t) .
\end{equation}
Inserting  \eqref{eq:mpartial1} into \eqref{eq:resm} and using \eqref{eq:solsymm1}-\eqref{eq:solsymm2} we arrive at the reduced linear system:
\begin{equation}\label{eq:sol linsys}
	(I - \mathbf{C}_{tx} \mathbf{Z}) \cdot  \vect{\beta}_{tx} = \mathbf{C}_{tx} \cdot \vect{1} \\
\end{equation}
where $\vect{\beta}_{tx}, \vect{1} \in \C^N$ and $\mathbf{C}_{tx}, \mathbf{Z} \in M(\C,N)$ are given by
\begin{equation}\label{eq:Mjk}
	\begin{gathered}
	\vect{\beta}_{tx} = \{ \beta_0  (x,t),\dots, \beta_{N-1}  (x,t) \}^\intercal,
	\qquad \vect{1} = \{1, \dots, 1 \}^\intercal \\
	\mathbf{C}_{tx} = \Diag (c_0 (x,t) \dots, c_{N-1} (x,t)) \qquad
	\{ \mathbf{Z}_{jk} \}_{j,k=0}^{N-1}
	= \dfrac{ \widebar z_j}{\widebar z_j -  z_k}.
	\end{gathered}
\end{equation}
  For general $\mathbf{C}_{tx}$ the matrix $I-\mathbf{C}_{tx} \mathbf{Z}$ need not be invertible. However, under the reality condition $c_j (x,t) = \im z_j |c_j (x,t)|$, the system can be expressed in the more symmetric form
\begin{equation}\label{sol linsys 2}
	(I+\mathbf{Y}_{tx} ) \cdot  \widehat{ \vect{\beta} }_{tx} = \vect{b}_{tx}
\end{equation}
where
\begin{equation*}
	\begin{aligned}
		& \widehat{\vect{\beta}}_{tx} :=
		\{ |c_0 (x,t)|^{-1/2} \beta_1, \dots, |c_{N-1} (x,t)|^{-1/2} \beta_{N-1} \}^\intercal \\
		& \vect{b} _{tx} :=
		\{ \im |c_0 (x,t)|^{1/2} z_1,\im |c_2 (x,t)|^{1/2} z_2, \dots,
		\im |c_{N-1} (x,t)|^{1/2} z_{N-1} \}^\intercal .
	\end{aligned}
\end{equation*}
Letting $y_j = - \im z_j$ ( $ \Im z_j >0 \Rightarrow \Re y_j > 0$)  we have
\begin{equation*}
	\begin{gathered}
		(\mathbf{Y}_{tx}) _{jk} = \frac{ |c_j(x,t)|^{1/2} |c_k(x,t)|^{1/2} }{\widebar y_j +  y_k}
		= |c_j(x,t)|^{1/2} |c_k(x,t)|^{1/2} \int_0^\infty e^{-( \widebar y_j +  y_k)s} ds.
	\end{gathered}
\end{equation*}
Invertibility of the system then follows from the observation that $\mathbf{Y}_{tx}$ is positive definite:
\begin{equation*}  \begin{aligned}   w^\dag \mathbf{Y}_{tx} w
	&= \int_0^\infty \lp \sum_{j,k=0}^{N-1} |c_j (x,t) c_k(x,t)|^{1/2} e^{-(  \overline{ y}_j +  y_k) s}   \overline{w}_j w_k \rp ds
	 \\& = \int_0^\infty \left| \sum_{k=0}^{N-1} |c_k(x,t)|^{1/2} e^{- y_k s} w_k \right|^2 ds \geq 0.
\end{aligned}
\end{equation*}
Using \eqref{msol} and Cramer's rule, the solution of the NLS corresponding  to the given discrete scattering data is given by
\begin{equation}\label{eq:Nsoliton}
	q^{(sol),N} (x,t) = 1 - \frac{ \det (I - ( \mathbf{C}_{tx}  \mathbf{Z})_1) }{ \det( I - \mathbf{C} _{tx} \mathbf{Z}) }
\end{equation}
where $( \mathbf{C} _{tx} \mathbf{Z})_1$ is the $(N+1) \times (N+1)$ matrix
\begin{equation}\label{eq:Nsoliton1}
	( \mathbf{C} _{tx}\mathbf{Z})_1:= \( \begin{array}{ccc|c}
		& & & c_0(x,t) \\
		& \mathbf{C} _{tx} \mathbf{Z} & & \vdots \\
		& & & c_{N-1}  (x,t) \\ \hline
		1 &\cdots &1 & 1
	\end{array} \).
\end{equation}

	
\section{Global existence of solution of the NLS equation}\label{sec:existence}
	
Here we establish the global existence of solutions for \eqref{eq:nls} with initial data $q_0 \in \tanh(x)+\Sigma_4$ and show that the $N$-soliton solutions $q^{(sol),N}(x,t)$ constructed in Appendix \ref{sec:msol} lie in this class of data.

\begin{theorem}\label{thm:existence}
Consider the initial value problem \eqref{eq:nls} with
$q_0-\tanh \left ( x \right ) \in \Sigma _4$.
Then  \eqref{eq:nls} admits a unique global solution $q$
such that $q(x,t)-\tanh \left ( x \right ) \in C^0([0,\infty ) ,H^4 (\R)) \cap C^1([0,\infty ) ,H^2(\R)) $.
Furthermore we have
$q(x,t)  -\tanh \left ( x \right ) \in C^0([0,\infty ) ,\Sigma _4) \cap C^1([0,\infty ) ,\Sigma _2) $.
\end{theorem}

\begin{proof}
By Gallo \cite{gallo} there is a unique global solution $q(x,t)$ of \eqref{eq:nls} s.t.
the function $v(x,t):= q(x,t)-\tanh \left ( x \right )$  is in   $C^0([0,\infty ) ,H^1(\R)) $.
Furthermore  since $v(x,0)\in X^4(\R )\subset X^1(\R )$, by \cite{gallo1,gallo} we also have $v(x,t) \in  C^0([0,\infty ) ,X^1(\R)) $, where $X^k(\R ) := L^\infty (\R ) \cap  ( \cap _{l=1}^k\dot H ^l (\R ))$.
In   \cite{BGSS} it is proven that $v(x,t)\in  C^0([0,\infty ) ,X^4(\R)) $.
All these facts together imply $v(x,t)\in  C^0([0,\infty ) ,H^4(\R))  \cap C^1([0,\infty ) ,H^2(\R))$.

The fact that $v(x,t)\in  C^0([0,\infty ) ,\Sigma _4) $  can now be proved by standard arguments; multiplying the equation for $v $ by $x^4 e^{-\varepsilon x^2}$  and,  taking the limit $\varepsilon \to 0^+$,  one shows that $ x^4 v(x,t) \in L^{\infty}([0,T], L^2(\R))$ for any $T$.
 Indeed, $v(x,t)$ solves (for $v_R=\Re v)$
 \begin{equation}\label{eqv}
 	\im \dot v + v_{xx}-2(|v|^2 +2v_R \tanh \left ( x \right ) )
		(v+\tanh \left ( x \right ) ) - \text{sech}^2 (x) v=0.
 \end{equation}
Multiplying the equation by $x ^{2j} e^{-2\varepsilon x^2}\overline{v} $ for $1\le j\le 4$, taking the imaginary  part and integrating in $ x$  on $\R$ we obtain, for 
$[  \partial ^2_x, x^j  e^{- \varepsilon x^2}  ] v 
= (x^j  e^{- \varepsilon x^2})^{ \prime\prime} v 
+ 2  (x^j  e^{- \varepsilon x^2})^{ \prime }   v_x$,
\begin{equation}\label{ms1}
	\frac{d}{dt} \| x^j  e^{- \varepsilon x^2} {v}\| _{L^2}
	\le C  ( \| [  \partial ^2_x, x^j  e^{- \varepsilon x^2}  ] v \|_{L^2} 
	+ \| \langle x \rangle ^j  e^{- \varepsilon x^2} {v}\| _{L^2} ).
\end{equation}
We have $\|    (x^j  e^{- \varepsilon x^2})^{ \prime\prime} v \|  _{L^2} \le
\| v \| _{\Sigma _{j-1}}$ , where we assume the r.h.s. is bounded by induction.

So, for fixed constants we have
\begin{equation} \label{ms2}
	\|   (x^j  e^{- \varepsilon x^2})^{ \prime }   v_x \| _{L^2} \le c' \| x ^{j-1}
	e^{- \varepsilon x^2} v_x  \|  _{L^2}   \le c  (\|  x ^{j } e^{- \varepsilon x^2} v \|  _{L^2}
	+ \|  \partial ^2 _x v   \|  _{L^2}).
\end{equation}
The 2nd inequality follows by  the identity for   $f$ real, see \cite{MS} p.1069,
\begin{equation*}
	\begin{aligned}
  	&       \int  x ^{2j-2} e^{- 2\varepsilon x^2} (f_x) ^2  dx=  2^{-1}
	\int f (  x ^{2j-2} e^{- 2\varepsilon x^2}) ^{\prime\prime} f ^2 dx
	+ \int x ^{2j-2} e^{- 2\varepsilon x^2} f f _{xx} dx  .
	\end{aligned}
\end{equation*}
Then, by Gronwall's inequality, \eqref{ms1}--\eqref{ms2} imply that
$\| x^j  e^{- \varepsilon x^2} v (\cdot , t)\| _{L^2(\R )}\le C_T$   for $t\in [0,T]$ and all $j=1,...,4$.
By Fatou's lemma we conclude $v(x,t) \in  L^\infty ( [0,T], \Sigma _4)$ for all $T\ge 0$.
But then by dominated convergence $x^j  e^{-  \varepsilon x^2} v\to x^j v$ in $ L^\infty ( [0,T], L^2(\R ))$
and since $x^j  e^{-  \varepsilon x^2} v \in  C^ 0 ( [0,T], L^2(\R ))$, we have also  $x^j    v \in  C^ 0 ( [0,T], L^2(\R ))$ for all $j\le 4$.
So we conclude   that $v(x,t) \in  C^ 0 ( [0,T], \Sigma _4)$.
From \eqref{eqv}  we have also   $v(x,t) \in  C^ 1 ( [0,T], \Sigma _2)$.

\end{proof}

The  global existence for \eqref{eq:nls}--\eqref{eq:nls1} for the initial data in  Theorem    \ref{thm:main2} is   guaranteed by  {Theorem} \ref{thm:existence} and the  following lemma.
\begin{lemma}\label{lem:inval}
     $ q^{(sol),N}(x,t)  - \tanh (x) \in \Sigma _k$ for all $k\in \N$
     for any $N$--soliton satisfying the boundary conditions \eqref{eq:nls1}.
\end{lemma}

\begin{proof}
Formulas \eqref{eq:Nsoliton}--\eqref{eq:Nsoliton1} imply immediately that
$q^{(sol),N}  \in \C ^\infty (\R ^2, \C )$.
Since for $|x| \to \infty$ we have $\Phi(z_j ; x,t) = -2x\Im [z_j] (1+o(1))$ it is elementary that $q^{(sol),N} (x,t)-1$ with all its derivatives  approaches 0 exponentially fast  as $x\to \infty$.
We assume now
\begin{equation}\label{q-ifnty}
	\lim _{x\to -\infty}q^{(sol),N}  (x,t)=a(0) \text{ for any fixed $t\ge 0  $}
\end{equation}
(where in the set up of Lemma \ref{lem:inval}  we have $a(0)=-1$).
Then  for any fixed $t$ it is elementary to conclude from  \eqref{eq:Nsoliton}--\eqref{eq:Nsoliton1} that for a fixed    $c>0$ and for $x\ll -1$
\begin{equation*}
	\begin{aligned}
	&   \det( I - \mathbf{C}_{tx}  \mathbf{Z}) =
	(-1)^N \prod _{j=0}^{N-1}c_j(x,t) \det ( \mathbf{Z})  (1+\bigo{e ^{cx}}),  \\
	&  \det (I - ( \mathbf{C}  \mathbf{Z})_1) =
	(-1) ^{N+1}  \prod _{j=0}^{N-1}c_j(x,t)
	\det \( \begin{array}{ccc|c}
		& & & 1 \\
		& \mathbf{Z} & & \vdots \\
		& & & 1 \\ \hline
		1 &\cdots &1 & 0
	\end{array} \)   (1+\bigo{e ^{cx}}).
	\end{aligned}
\end{equation*}
This implies that  $q^{(sol),N} (x,t)-a(0)$ and all of its derivatives  approaches 0 exponentially fast as $x\to - \infty$.
 To see \eqref{q-ifnty} we associate to  our  $m(z; x,t)$ the function $m ^{(1)}(z; x,t) $
 in \eqref{m1}.
 Notice that  $m ^{(1)}(z ) $ solves a Riemann--Hilbert problem in
 $\Sigma ^{(2)}$ since the jump matrix in $\R$ is the identity.
 In other words, here $m ^{( 1)}(z )=m ^{(sol)}(z ) $.
 Now, since as $x \to -\infty$ we have $\xi \to -\infty$, in this case $\Delta =\{  0,...,N-1\}$ and $j_0(\xi )=-1$, see \eqref{delta}--\eqref{j0}.
It is also easy to see, following the proof of Lemma \ref{lem: ex3rad}, that  for fixed $t$  for a fixed $c>0$   and all $x\ll -1 $ we have
 \[
 	m ^{(sol)}(z ) =I+z^{-1}\sigma _1 + z^{-1}\bigo{e ^{-c|x|}} + o(z^{-1}).
\]
Finally, proceeding as in Sect. \ref{sec:pfmain1}  as in \eqref{pfmain11} and using \eqref{eq:thetacond2} we have
\begin{equation*}
	\begin{aligned}
		  q^{(sol),N} (x,t) &= \lim_{z \to \infty} zm_{21}(z)
		=   T(\infty ,\xi)^{-2}\lim_{z \to \infty} zm_{21}^{(sol)}(z) \\
		&= a(0) (1+ \bigo{e ^{-c|x|}}) \to a(0) \text{ as $x\to -\infty$}.
	\end{aligned}
\end{equation*}
\end{proof}

\section{Singularity of $a(z)$ in $z=\pm 1$ for generic $q_0$} \label{app:sing}

We check here that for initial data $q_{0 \epsilon} =\tanh \left (   x
\right )  +  \varepsilon f$ with $f = f  _R+\im f _I$, $ f _A \in C_c^{\infty}( \R , \R )$ for $A=R,I$  generic, then the function $ a(z) $ blows up at $z=\pm 1$.   
Let $\psi _{j\varepsilon}^{\pm } (z;x )$ denote the Jost functions corresponding to initial data $q_{0\epsilon}$ (\cf\, \eqref{eq:jf1}).  
In particular,  by $\psi _{j0}^{\pm } (z;x )$ we denote the Jost functions associated to the black soliton  $q_{00}(x) =  \tanh \left (   x \right ).$

These functions  extend  to $(z,x)\in (\C \backslash \{ 0,-\im \})  \times \R$ and they are smooth.  Recalling \eqref{eq:wron2} we   denote
\begin{equation}\label{eq:wron2sin}
\begin{aligned}
&   a(\varepsilon ,z )  = \frac{W(\varepsilon ,z)}{1-z^{-2} }  \text{ where }  W(\varepsilon ,z):=\det [ \psi ^{-} _{1\varepsilon }(z;x),\psi _{2\varepsilon }^{+} (z;x) ].
\end{aligned}
\end{equation}
 Recall $ a(0,z ) = \frac{z-\im }{z+\im }$. This yields $ W(0 ,z) =\frac{(z-\im )(z^2-1) }{(z+\im ) z}  $.
 We have the following fact.
\begin{lemma}   \label{lem:sin1}
  We have
   \begin{equation} \label{eq:sin2} \begin{aligned} &W(\varepsilon ,z) =  \mp 2\im (z\mp 1)   - 2\varepsilon C(\pm  , f )
    +F_{\pm }(z, \varepsilon )
  \\& C(\pm  , f ):= \int
  _\R  \frac{      ( e^{-4y}-1)f _R(y)\mp 2e^{-2y} f _I(y) }{(1+e^{-2y})^2}    { dy}
  \end{aligned} \end{equation}
  where, for $|z\mp 1|<c_{f}$ and $ |\varepsilon |<c_{f}$  for a sufficiently  small constant $c_{f}>0$, the function
   $F_{\pm }(z, \varepsilon ) $ is analytic in $z$
  and   for a fixed constant $C_{f}$
   \begin{equation} \label{eq:sin3} \begin{aligned} &
   |F_{\pm }(z, \varepsilon )|\le C_{f} (|z\mp 1|^2+ \varepsilon ^2).
  \end{aligned} \end{equation}

\end{lemma}
For generic $f \in C_c^{\infty}( \R , \C )$   we have $ C(\pm  ,f )\neq 0$. Then replacing  $F_{\pm }(z, \varepsilon ) $ with 0 we obtain a function
with a zero in
 \begin{equation*}  \begin{aligned} &  \widetilde{z}_{\pm }(\varepsilon ) = \pm (1+\im \varepsilon  C(\pm  , f ))
  \end{aligned} \end{equation*}
and by Rouch\'e Theorem we have that $W(\varepsilon ,z)$ has for $\varepsilon $ small a zero
 \begin{equation*} \label{eq:sin4} \begin{aligned} &  z_{\pm }(\varepsilon ) = \pm (1+\im \varepsilon  C(\pm  , f )) +\bigo{\varepsilon ^2}.
  \end{aligned} \end{equation*}
If $\pm  \varepsilon  C(\pm  , f ) >0$ this yields a new zero in $\C _+$ of $ a(\varepsilon ,z )$ near
$\pm 1$ and a corresponding almost white soliton.  If   $\pm  \varepsilon  C(\pm  , f ) <0$
this is a new zero of the analytic continuation of $ a(\varepsilon ,z )$ below $\R$, does not yield a new soliton but nonetheless  makes $ a(\varepsilon ,z )$ singular at $\pm 1$. All four cases can occur.

{\it Proof of Lemma \ref{lem:sin1}}. Recall the definitions of $\psi^\pm_{j\eps}$ and $\psi^\pm_{j 0}$ for $j=1,2$, from the first paragraph of this appendix. For $(\psi ^{\pm } _{10}(z ; x))_j$ the $j$--th component of $\psi ^{\pm } _{10}(z; x)$ for $j=1,2$, 
we set
\begin{equation*}  \begin{aligned} &  \Delta Q(x):= \begin{pmatrix} 0 & \overline{f}(x)\\ f(x) & 0 \end{pmatrix} \, , \, U(x,y,z) =[\psi ^{- } _{10}(z;x) , \psi ^{+ } _{20}(z;x)][\psi ^{- } _{10}(z;y) , \psi ^{+ } _{20}(z;y)] ^{-1},
  \end{aligned} \end{equation*}
  with $[\psi ^{- } _{10}  , \psi ^{+ } _{20} ]  $ the matrix with first column $\psi ^{- } _{10} $ and second column   $\psi ^{+ } _{20} $ and with the last  the inverse of one such
  matrix.
$ U(x,y,z)$ is well defined for any  $ z\neq 0,\pm \im $ in $\C$.  We have $(\partial _x -\mathcal{L}(z;x)) U(x,y,z) =0 $   and $U(y,y,z) =1$, i.e. $U(x,y,z)$
is the fundamental solution of  equation
  \eqref{eq:laxL} with $Q$ defined using $\tanh (x)$.
Let $f \in C^{\infty}_c((-M,M), \C)$.
Notice then that for  $\psi ^{- } _{1\varepsilon} $ and  $\psi ^{+ } _{2\varepsilon} $
Jost functions associated to $q_{0\eps}$, we have  $ \psi ^{+ } _{2\varepsilon}(z; x)= \psi ^{+ } _{20}(z; x)$ for $x>M$ and $\psi ^{- } _{1\varepsilon}(z; x)= \psi ^{- } _{10}(z; x)$ for $x<-M$.
Then  for $x>M$
   we have  for any preassigned $x_0<-M$
\begin{equation} \label{eq:sin5} \begin{aligned}
 \psi ^{+ } _{2\varepsilon}(z;x)&= \psi ^{+ } _{20}(z;x)\\
 \psi ^{- } _{1\varepsilon}(z;x)&=  U(x,x_0,z) \psi ^{- } _{1\varepsilon}(z; x_0)
  +\im \varepsilon \int _{x_0} ^x U(x,y,z)\sigma _3 \Delta Q (y) \psi ^{- } _{1\varepsilon }(z;y) dy \\& =
  \psi ^{- } _{10}(z;x)+\im \varepsilon \int _{\R}  U(x,y,z)\sigma _3 \Delta Q (y) \psi ^{- } _{1\varepsilon }(z;y) dy.
  \end{aligned} \end{equation}
 Picking $x>M$ and substituting  \eqref{eq:sin5}
   we can write
\begin{equation*} \label{eq:sin6} \begin{aligned} &  W(\varepsilon ,z)= \det [ \psi ^{-} _{1\varepsilon }(z;x),\psi _{2\varepsilon }^{+} (z;x ) ] =  \det [ \psi ^{-} _{1\varepsilon }(z;x),\psi _{20 }^{+} (z;x)]  \\&
=W(0 ,z) + \im \varepsilon \int _\R I'(y,z) dy  \text{ where }   I'(y,z):=   \det [  U(x,y,z)\sigma _3 \Delta Q (y) \psi ^{- } _{1\varepsilon }(z;y)  ,\psi ^{+} _{20 }(z;y)  ] .
  \end{aligned} \end{equation*}
  Notice that we have
  \begin{equation*}  \begin{aligned} &    I'(y,z) =       \det \left [ \  [\psi ^{- } _{10}(z;x) , \psi ^{+ } _{20}(z;x)] F(y,z) ,\psi ^{+} _{20 }(z;x) \right ] \\& = \det [ F_1(y,z) \psi ^{- } _{10}(z;x) + F_2(y,z) \psi ^{+ } _{20}(z;x)  ,\psi ^{+} _{20 }(z;x) ]   =F_1(y,z)W(0 ,z),
  \end{aligned} \end{equation*}
    for  $F(y,z)  $ the 2 components column vector
\begin{equation*} \label{eq:sin7} \begin{aligned} &
 \begin{pmatrix}F_1(y,z)\\ F_2(y,z) \end{pmatrix} = [\psi ^{- } _{10}(z;y) , \psi ^{+ } _{20}(z;y)] ^{-1}\sigma _3 \Delta Q (y) \psi ^{- } _{1\varepsilon }(z;y)
 \\& = \frac{1} {W(0 ,z)} \begin{pmatrix}(\psi ^{+ } _{20}(z;y))_2  &- (\psi ^{+ } _{20}(z;y))_1  \\-(\psi ^{-} _{10}(z;y))_2 & (\psi ^{- } _{10}(z;y))_1  \end{pmatrix}   \begin{pmatrix}\overline{f}(y)(\psi ^{- } _{1\varepsilon }(z;y))_2 \\ -{f}(y) (\psi ^{- } _{1\varepsilon }(z;y))_1 \end{pmatrix} ,
  \end{aligned} \end{equation*}
  $(\psi ^{+ } _{20}(z;y))_j$  the $j$--th component of $\psi ^{+ } _{20}(z;y)$ and similar notation for the other Jost functions.
So
\begin{equation*} \label{eq:sin8} \begin{aligned} &     I'(y,z) = \overline{f}(y)(\psi ^{+ } _{20}(z;y))_2 (\psi ^{- } _{1\varepsilon }(z;y))_2+ {f}(y)(\psi ^{+ } _{20}(z;y))_1(\psi ^{- } _{1\varepsilon }(z;y))_1.
  \end{aligned} \end{equation*}
Furthermore by the Lipschitz continuity  in  $q $ in   Lemmas \ref{lem:jf1b} and \ref{lem:jf3}, in particular the analogues for
 $z\to  m ^{- } _{1  }(z;x)$ of the maps \eqref{eq:eqm13}
and
\eqref{eq:eqm13jf3}, for a fixed $C$
  and when $z$ is in a preassigned compact subset of $\C \backslash \{ 0, -\im \}$
  we have
 \begin{equation*} \label{eq:sin9} \begin{aligned} &
 \| \psi ^{- } _{1\varepsilon }(z;\cdot )-\psi ^{- } _{10 }(z;\cdot )\| _{L^\infty (-\infty , M)}=\| \psi ^{- } _{1\varepsilon }(z;\cdot )-\psi ^{- } _{10 }(z;\cdot )\| _{L^\infty (-M , M)} <C \varepsilon   .
  \end{aligned} \end{equation*}
This yields
\begin{equation*} \label{eq:sin10} \begin{aligned} &  W(\varepsilon ,z) =W(0,z) + \im \varepsilon \int _\R   I (y,\pm 1)  dy+
\widetilde{F}_{\pm}(z,\varepsilon)   \\&  I (y,z) =   \overline{f}(y)(\psi ^{+ } _{20}(z;y))_2 (\psi ^{- } _{10 }(z;y))_2+ {f}(y)(\psi ^{+ } _{20}(z;y))_1(\psi ^{- } _{10 }(z;y))_1\\& \widetilde{F}_{\pm}(z,\varepsilon)=
 \im \varepsilon \int _\R [I (y,z)-I (y,\pm 1)] dy + \bigo{\varepsilon ^2}
  \end{aligned} \end{equation*}
where $\widetilde{F}_{\pm}(z,\varepsilon)  $ has the  properties claimed in the statement for $ {F}_{\pm}(z,\varepsilon)  $.

\noindent We have
\begin{equation}  \label{eq:sin11} \begin{aligned} &
\psi ^{+ } _{20}( \pm 1  ;y) =\im \psi ^{- } _{10 }( \pm 1 ;y) , \\&
\psi ^{- } _{10 }(  - 1 ;y)=\frac{1}{1+e^{-2y}}
 \begin{pmatrix}\im +e^{-2y}\\ -\im +e^{-2y} \end{pmatrix}
  \text{ and }  \psi ^{- } _{10 }(   1 ;y) =\frac{1}{1+e^{-2y}}
 \begin{pmatrix}-\im + e^{-2y}\\ -\im -  e^{-2y}  \end{pmatrix}
  .\end{aligned} 
  \end{equation}
  \eqref{eq:sin11} can be derived in an elementary fashion  by first substituting $z _{j_0} =\im$ and $\varphi  _{j_0} = x$    in  formula \eqref{eq:eq1a}   for $x>0$ and
  formula \eqref{eq:eq1b} for $x<0$. 
  This yields the formula for the matrix  $m(z; x)$ in \eqref{eq:defm}. 
  To obtain the Jost functions one then  multiplies by $a(z)= \frac{z-\im}{z+\im} $ the 1st (resp. 2nd) column of $m(z; x)$ if $\Im (z) >0$ (resp. $\Im (z) <0$), uses formulas \eqref{eq:jf1}  and exploits $\zeta (\pm 1)=0$ for the
  function in \eqref{eq:zeta} getting    \eqref{eq:sin11} with simple computations.
  After other elementary computations we get the formulas for $C(\pm  , f )$
in \eqref{eq:sin2}.
\qed

\section*{acknowledgements}   
S.C. was partially funded    by the grant FIRB 2012 (Dinamiche Dispersive)  from the Italian Government   and by a grant FRA 2013 from the University of Trieste.  We wish to thank Prof. Tamara Grava for useful discussions at the initial stages of this work.

%
%
%
%

\end{document}